\documentclass[11pt]{amsart}
\usepackage{lmodern}
\usepackage{ifthen}
\usepackage{amsfonts}
\usepackage{amsxtra}
\usepackage{amssymb}
\usepackage{array}
\usepackage{eucal}
\usepackage[margin=1.0in]{geometry}
\usepackage{xcolor}
\definecolor{cite}{rgb}{0.30,0.60,1.00}
\definecolor{url}{rgb}{0.00,0.00,0.80}
\definecolor{link}{rgb}{0.40,0.10,0.20}
\usepackage[colorlinks,linkcolor=link,urlcolor=url,citecolor=cite,pagebackref,breaklinks]{hyperref}
\usepackage{bbm}
\usepackage{mathtools}
\usepackage{mathrsfs}
\usepackage{appendix}
\usepackage[all]{xy}
\usepackage[lite,abbrev,msc-links,alphabetic]{amsrefs}

\usepackage{graphicx}
\usepackage{multirow}
\usepackage{pstricks}
\usepackage{pst-pdf}

\usepackage[OT2,T1]{fontenc}
\DeclareSymbolFont{cyrletters}{OT2}{wncyr}{m}{n}
\DeclareMathSymbol{\Sha}{\mathalpha}{cyrletters}{"58}

\newcommand{\yichao}[1]{{\color{red}  $\heartsuit\heartsuit\heartsuit$ Yichao: [#1]}}
\newcommand{\yifeng}[1]{{\color{blue}  $\spadesuit\spadesuit\spadesuit$ Yifeng: [#1]}}

\numberwithin{equation}{subsection}

\theoremstyle{plain}
\newtheorem{proposition}{Proposition}[section]

\newtheorem{corollary}[proposition]{Corollary}
\newtheorem{lem}[proposition]{Lemma}
\newtheorem{theorem}[proposition]{Theorem}

\theoremstyle{definition}
\newtheorem{definition}[proposition]{Definition}

\newtheorem{notation}[proposition]{Notation}
\newtheorem{assumption}[proposition]{Assumption}

\theoremstyle{remark}
\newtheorem{remark}[proposition]{Remark}
\newtheorem{example}[proposition]{Example}

\renewcommand{\b}[1]{\mathbf{#1}}
\renewcommand{\c}[1]{\mathcal{#1}}
\renewcommand{\d}[1]{\mathbb{#1}}
\newcommand{\f}[1]{\mathfrak{#1}}
\renewcommand{\r}[1]{\mathrm{#1}}

\renewcommand{\sf}[1]{\mathsf{#1}}
\renewcommand{\(}{\left(}
\renewcommand{\)}{\right)}
\newcommand{\res}{\mathbin{|}}

\newcommand{\ang}[1]{\langle{#1}\rangle}

\newcommand{\bA}{\b A}

\newcommand{\bC}{\b C}

\newcommand{\bE}{\b E}
\newcommand{\bF}{\b F}

\newcommand{\bH}{\b H}
\newcommand{\bI}{\b I}

\newcommand{\bL}{\b L}

\newcommand{\bP}{\b P}
\newcommand{\bQ}{\b Q}
\newcommand{\bR}{\b R}

\newcommand{\bZ}{\b Z}

\newcommand{\bff}{\b f}

\newcommand{\cA}{\c A}
\newcommand{\cB}{\c B}

\newcommand{\cD}{\c D}

\newcommand{\cI}{\c I}

\newcommand{\cM}{\c M}

\newcommand{\cO}{\c O}
\newcommand{\cP}{\c P}

\newcommand{\cS}{\c S}

\newcommand{\cX}{\c X}
\newcommand{\cY}{\c Y}

\newcommand{\dG}{\d G}

\newcommand{\dP}{\d P}

\newcommand{\dS}{\d S}
\newcommand{\dT}{\d T}

\newcommand{\fB}{\f B}

\newcommand{\fD}{\f D}

\newcommand{\fH}{\f H}

\newcommand{\fM}{\f M}
\newcommand{\fN}{\f N}

\newcommand{\fX}{\f X}

\newcommand{\fa}{\f a}
\newcommand{\fb}{\f b}
\newcommand{\fc}{\f c}
\newcommand{\fd}{\f d}
\newcommand{\fe}{\f e}

\newcommand{\fl}{\f l}
\newcommand{\fm}{\f m}

\newcommand{\fp}{\f p}
\newcommand{\fq}{\f q}
\newcommand{\fr}{\f r}
\newcommand{\fs}{\f s}

\newcommand{\fw}{\f w}

\newcommand{\rB}{\r B}

\newcommand{\rG}{\r G}
\newcommand{\rH}{\r H}
\newcommand{\rI}{\r I}

\newcommand{\rM}{\r M}
\newcommand{\rN}{\r N}

\newcommand{\rR}{\r R}
\newcommand{\rS}{\r S}
\newcommand{\rT}{\r T}

\newcommand{\sfM}{\sf M}

\newcommand{\sfh}{\sf h}

\newcommand{\tR}{\mathtt{R}}
\newcommand{\tS}{\mathtt{S}}
\newcommand{\tT}{\mathtt{T}}

\newcommand{\CC}{\mathbf{C}}
\newcommand{\QQ}{\mathbf{Q}}
\newcommand{\ZZ}{\mathbf{Z}}
\newcommand{\RR}{\mathbf{R}}
\newcommand{\FF}{\mathbf{F}}

\newcommand{\xra}{\xrightarrow}
\newcommand{\hra}{\hookrightarrow}

\newcommand{\deR}{\mathrm{dR}}
\newcommand{\bfSh}{\mathbf{Sh}}

\newcommand{\es}{\mathrm{es}}
\newcommand{\Gys}{\mathrm{Gys}}

\newcommand{\pres}[2]{\prescript{#1}{}{#2}}

\newcommand{\ac}{\r{ac}}

\newcommand{\Cl}{\r{Cl}}

\newcommand{\dr}{\r{dR}}

\newcommand{\inv}{\r{inv}}
\newcommand{\loc}{\r{loc}}

\newcommand{\op}{\r{op}}

\newcommand{\rat}{\r{rat}}
\newcommand{\sing}{\r{sing}}

\newcommand{\TF}{\widetilde{F}}
\newcommand{\tor}{\r{tor}}

\newcommand{\unr}{\r{unr}}
\newcommand{\ur}{\r{ur}}
\newcommand{\pr}{\mathrm{pr}}
\newcommand{\tri}{\r{tri}}
\newcommand{\Iw}{\mathrm{Iw}}

\newcommand{\tcD}{\tilde\cD}

\DeclareMathOperator{\AJ}{AJ}

\DeclareMathOperator{\Aut}{Aut}

\DeclareMathOperator{\CH}{CH}

\DeclareMathOperator{\coker}{coker}

\DeclareMathOperator{\disc}{disc}

\DeclareMathOperator{\End}{End}

\DeclareMathOperator{\Fr}{Fr}

\DeclareMathOperator{\Frob}{Frob}

\DeclareMathOperator{\Gal}{Gal}

\DeclareMathOperator{\GL}{GL}

\DeclareMathOperator{\Hom}{Hom}

\DeclareMathOperator{\Ind}{Ind}

\DeclareMathOperator{\Ker}{ker}
\DeclareMathOperator{\Lie}{Lie}

\DeclareMathOperator{\Mat}{Mat}

\DeclareMathOperator{\Nm}{Nm}

\DeclareMathOperator{\ord}{ord}

\DeclareMathOperator{\Res}{Res}

\DeclareMathOperator{\Sh}{Sh}
\DeclareMathOperator{\SL}{SL}

\DeclareMathOperator{\Spec}{Spec}

\DeclareMathOperator{\Tr}{Tr}

\DeclareMathOperator{\Art}{Art}

\begin{document}

\title[Supersingular locus, level raising, and Selmer groups]
{Supersingular locus of Hilbert modular varieties, arithmetic level raising, and Selmer groups}

\author{Yifeng Liu}
\address{Department of Mathematics, Northwestern University, Evanston IL 60208, United States}
\email{liuyf@math.northwestern.edu}

\author{Yichao Tian}
\address{Mathematisches Institut, Universit\"{a}t Bonn, 53115 Bonn, Germany}
\email{tian@math.uni-bonn.de}

\date{\today}

\begin{abstract}
  This article has three goals. First, we generalize the result of Deuring and Serre on the characterization of supersingular locus to all Shimura varieties given by totally indefinite quaternion algebras over totally real number fields. Second, we generalize the result of Ribet on arithmetic level raising to such Shimura varieties in the inert case. Third, as an application to number theory, we use the previous results to study the Selmer group of certain triple product motive of an elliptic curve, in the context of the Bloch--Kato conjecture.
\end{abstract}

\maketitle

\tableofcontents

\section{Introduction}

The study of special loci of moduli spaces of abelian varieties starts from Deuring and Serre. Let $N\geq 4$ be an integer and $p$ a prime not dividing $N$. Let $Y_0(N)$ be the coarse moduli scheme over $\bZ_{(p)}$ parameterizing elliptic curves with a cyclic subgroup of order $N$. Let $Y_0(N)_{\bF_p}^{\r{ss}}$ denote the supersingular locus of the special fiber $Y_0(N)_{\bF_p}$, which is a closed subscheme of dimension zero. Deuring and Serre proved the following deep result (see, for example \cite{Ser96}) characterizing the supersingular locus:
\begin{align}\label{E:deuring}
Y_0(N)_{\bF_p}^{\r{ss}}(\bF_p^\ac)\cong B^\times\backslash\widehat{B}^\times/\widehat{R}^\times.
\end{align}
Here, $B$ is the definition quaternion algebra over $\bQ$ ramified at $p$, and $R\subseteq B$ is any Eichler order of level $N$. Moreover, the induced action of the Frobenius element on $B^\times\backslash\widehat{B}^\times/\widehat{R}^\times$ coincides with the Hecke action given by the uniformizer of $B\otimes_\bQ\bQ_p$.

One main application of the above result is to study congruence of modular forms. Let $f=q+a_2q^2+a_3q^3+\cdots$ be a normalized cusp new form of level $\Gamma_0(N)$ and weight $2$. Let $\fm_f$ be the ideal of the away-from-$Np$ Hecke algebra generated by $\rT_v-a_v$ for all primes $v\nmid Np$. We assume that $f$ is not dihedral. Take a sufficiently large prime $\ell$, not dividing $Np(p^2-1)$. Using the isomorphism \eqref{E:deuring} and the Abel--Jacobi map (over $\bF_{p^2}$), one can construct a map
\begin{align}\label{E:ribet}
\Gamma(B^\times\backslash\widehat{B}^\times/\widehat{R}^\times,\bF_\ell)/\fm_f\to\rH^1(\bF_{p^2},\rH^1(Y_0(N)\otimes\bF_p^\ac,\bF_\ell(1))/\fm_f)
\end{align}
where $\Gamma(B^\times\backslash\widehat{B}^\times/\widehat{R}^\times,\bF_\ell)$ denotes the space of $\bF_\ell$-valued functions on $B^\times\backslash\widehat{B}^\times/\widehat{R}^\times$. In \cite{Rib90}, Ribet proved that the map \eqref{E:ribet} is surjective. Note that the right-hand side is nonzero if and only if $\ell\mid a_p^2-(p+1)^2$, in which case the dimension is $1$. From this, one can construct a normalized cusp new form $g$ of level $\Gamma_0(Np)$ and weight $2$ such that $f\equiv g\mod\ell$ when $\ell\mid a_p^2-(p+1)^2$.

This article has three goals. First, we generalize the result of Deuring and Serre to all Shimura varieties given by totally indefinite quaternion algebras over totally real number fields. Second, we generalize Ribet's result  to such Shimura varieties in the inert case. Third, as an application to number theory, we use the previous results to study Selmer groups of certain triple product motives of elliptic curves, in the context of the Bloch--Kato conjecture.

For the rest of Introduction, we denote $F$ a totally real number field, and $B$ a \emph{totally indefinite} quaternion algebra over $F$. Put $G\coloneqq\Res_{F/\bQ}B^\times$ as a reductive group over $\bQ$.

\subsection{Supersingular locus of Hilbert modular varieties}

Let $p$ be a rational prime that is unramified in $F$. Denote by $\Sigma_p$ the set of all places of $F$ above $p$, and put $g_\fp\coloneqq[F_\fp:\bQ_p]$ for every $\fp\in\Sigma_p$.
Assume that $B$ is unramified at all $\fp\in \Sigma_p$. Fix a maximal order $\cO_B$ in $B$. Let $K^p\subseteq G(\bA^\infty)$ be a neat open compact subgroup in the sense of Definition \ref{D:neat-subgroup}. We have a coarse moduli scheme $\bfSh(G,K^p)$ over $\bZ_{(p)}$ parameterizing abelian varieties with real multiplication by $\cO_B$ and $K^p$-level structure (see Section \ref{S:quaternion-PEL} for details). Its generic fiber is a Shimura variety; in particular, we have the following well-known complex uniformization
\[
\bfSh(G,K^p)(\bC)\cong G(\QQ)\backslash(\bC-\bR)^{[F:\bQ]}\times G(\bA^{\infty})/K^pK_p,
\]
where $K_p$ is a hyperspecial maximal subgroup of $G(\bQ_p)$. The supersingular locus of $\bfSh(G,K^p)$, that is, the maximal closed subset of $\bfSh(G,K^p)\otimes\bF_p^\ac$ on which the parameterized abelian variety (over $\bF_p^\ac$) has supersingular $p$-divisible group, descends to $\bF_p$, denoted by $\bfSh(G,K^p)^{\r{ss}}_{\bF_p}$. Our first result provides a global description of the subscheme $\bfSh(G,K^p)^{\r{ss}}_{\bF_p}$.

To state our theorem, we need to introduce another Shimura variety. Let $B'$ be the quaternion algebra over $F$, unique up to isomorphism, such that the Hasse invariants of $B'$ and $B$ differ exactly at all archimedean places and all $\fp\in\Sigma_p$ with $g_\fp$ odd.
Similarly, put $G'\coloneqq\Res_{F/\bQ}B'^\times$ and identify $G'(\bA^{\infty,p})$ with $G(\bA^{\infty,p})$.
We put
\[
\bfSh(G',K^p)(\bF_{p}^{\ac})\coloneqq G'(\QQ)\backslash G'(\bA^{\infty})/K^pK'_p,
\]
where $K'_p$ is the a maximal open compact subgroup of $G'(\bQ_p)$. We denote by  $\bfSh(G',K^p)_{\bF_{p}^{\ac}}$  the corresponding   scheme over $\FF_p^{\ac}$, i.e. copies of $\Spec(\FF_p^{\ac})$ indexed by  $\bfSh(G',K^p)(\bF_{p}^{\ac})$.

\begin{theorem}(Theorem \ref{T:supersingular-quaternion})\label{T:main1}
Let $h$ be the least common multiple of $(1+g_\fp-2\lfloor g_\fp/2\rfloor)g_\fp$ for $\fp\in\Sigma_p$. We have\footnote{The notation here is simplified. In fact, in the main text and particularly Theorem \ref{T:supersingular-quaternion}, $B'$, $G'$, $\fB$, $\fa$ and $W(\fa)$ are denoted by $B_{\tS_{\r{max}}}$, $G_{\tS_{\r{max}}}$, $\fB_\emptyset$, $\underline\fa$ and $W_\emptyset(\underline\fa)$, respectively.}
\[
\bfSh(G,K^p)_{\FF_p}^{\mathrm{ss}}\otimes\bF_{p^h}=\bigcup_{\fa\in\fB}W(\fa).
\]
Here
\begin{itemize}
  \item $\fB$ is a set of cardinality $\prod_{\fp\in\Sigma_p}\binom{g_\fp}{\lfloor g_\fp/2\rfloor}$ equipped with a natural action by $\Gal(\bF_{p^h}/\bF_p)$;

  \item the base change $W(\fa)\otimes \FF_p^{\ac}$ is a $(\sum_{\fp\in\Sigma_p}\lfloor g_{\fp}/2\rfloor)$-th iterated $\dP^1$-fibration over $\bfSh(G',K^p)_{\bF_{p}^{\ac}}$, equivariant under prime-to-$p$ Hecke correspondences. \footnote{One should consider $\bfSh(G',K'^p)_{\FF_p^{\ac}}$ as the $\FF_p^{\ac}$-fiber of a Shimura variety attached to $G'$. However, it seems impossible to define the correct Galois action on $\bfSh(G',K'^p)_{\FF_p^{\ac}}$ using the formalism of Deligne homomorphisms when $g_{\fp}$ is odd for at least one $\fp\in \Sigma_p$. When  $g_{\fp}$ is odd for all $\fp\in \Sigma_p$, we will define the correct Galois action by $\Gal(\FF_p^{\ac}/\FF_p)$ using  superspecial locus. See  the discussion after Theorem~\ref{T:main2}.}
\end{itemize}
In particular, $\bfSh(G,K^p)_{\FF_p}^{\mathrm{ss}}$ is proper and of equidimension $\sum_{\fp\in\Sigma_p}\lfloor g_{\fp}/2\rfloor$.
\end{theorem}

If $p$ is inert in $F$ of degree $2$ and $B$ is the matrix algebra, then the results was first proved in \cite{BG99}. If $p$ is inert in $F$ of degree $4$ and $B$ is the matrix algebra, then the results was due to \cite{Yu03}. Assume that $p$ is inert in $F$ of even degree. Then the strata $W(\fa)$ have already been constructed in \cite{TX2}, and the authors proved there that, under certain genericity conditions on the Satake parameters of a fixed automorphic cuspidal representation $\pi$, the cycles $W(\fa)$ give all the $\pi$-isotypic Tate cycles on $\bfSh(G,K^p)_{\FF_p}$.

Similarly, one can define the superspecial locus $\bfSh(G,K^p)^{\r{sp}}_{\bF_p}$ of $\bfSh(G,K^p)$, that is, the maximal closed subset of $\bfSh(G,K^p)\otimes\bF_p^\ac$ on which the parameterized abelian variety has superspecial $p$-divisible group. It is a reduced scheme over $\bF_p$ of dimension zero. We have the following result.

\begin{theorem}[Theorem \ref{P:superspecial-quaternion}]\label{T:main2}
Assume that $g_\fp$ is odd for every $\fp\in\Sigma_p$. For each $\fa\in\fB$ as in the previous theorem, $W(\fa)$ contains the superspecial locus $\bfSh(G,K^p)^{\r{sp}}_{\FF_p}\otimes\bF_{p^h}$, and the iterated $\dP^1$-fibration $\pi_\fa\colon W(\fa)\otimes \FF_p^{\ac}\to \bfSh(G',K^p)_{\bF_{p}^{\ac}}$ induces an isomorphism
\[
\bfSh(G,K^p)^{\r{sp}}_{\bF_p^{\ac}}\xra{\sim}\bfSh(G',K^p)_{\bF_{p}^{\ac}}
\]
compatible with prime-to-$p$ Hecke correspondences.
\end{theorem}

We will always identify $\bfSh(G,K^p)^{\r{sp}}_{\bF_p^{\ac}}$ with $\bfSh(G',K^p)_{\bF_{p}^{\ac}}$. Via this identification, $\bfSh(G',K^p)_{\bF_{p}^{\ac}}$ descends to an (\'etale) $\FF_p$-scheme  $\bfSh(G',K^p)_{\bF_{p}}$, and a direct description for the action of $\Gal(\FF_p^{\ac}/\FF_{p^2})$ on $\bfSh(G',K^p)_{\bF_{p}^{\ac}}$ is given in Theorem~\ref{P:superspecial-quaternion}(2).
Then it is easy to see that the iterated $\dP^1$-fibration $\pi_{\fa}$ descends to a morphism of $\FF_{p^{h}}$-schemes:
\[
\pi_{\fa}\colon W(\fa)\to \bfSh(G',K^p)_{\FF_p^{h}}.
\]

A main application of the global description of the supersingular locus is to study the level raising phenomenon, as we will explain in the next section.

\subsection{Arithmetic level raising for Hilbert modular varieties}

We suppose that $g=[F:\bQ]$ is odd. Fix an irreducible cuspidal automorphic representation $\Pi$ of $\GL_2(\bA_F)$ of parallel weight $2$ defined over a number field $\bE$. Let $B$, $G$ be as in the previous section; and let $K$ be a neat open compact subgroup of $G(\bA^\infty)$. Then we have the Shimura variety $\Sh(G,K)$ defined over $\bQ$. Let $\tR$ be a finite set of places of $F$ away from which $\Pi$ is unramified and $K$ is hyperspecial maximal.

Let $p$ be a rational prime inert in $F$ such that the unique prime $\fp$ of $F$ above $p$ is not in $\tR$. Then $K=K^pK_p$ and $\Sh(G,K)$ has a canonical integral model $\bfSh(G,K^p)$ over $\bZ_{(p)}$ as in the previous section. We also choose a prime $\lambda$ of $\bE$ and put $k_\lambda\coloneqq\cO_\bE/\lambda$.

Let $\bZ[\dT^\tR]$ (resp.\ $\bZ[\dT^{\tR\cup\{\fp\}}]$) be the (abstract) spherical Hecke algebra of $\GL_{2,F}$ away from $\tR$ (resp.\ $\tR\cup\{\fp\}$). Then $\Pi$ induces a homomorphism
\[
\phi_{\Pi,\lambda}\colon\bZ[\dT^\tR]\to\cO_\bE\to k_\lambda
\]
via Hecke eigenvalues. Put $\fm\coloneqq\Ker(\phi_{\Pi,\lambda}\res_{\bZ[\dT^{\tR\cup\{\fp\}}]})$.

The Hecke algebra $\bZ[\dT^{\tR\cup\{\fp\}}]$ acts on the (\'{e}tale) cohomology group $\rH^\bullet(\bfSh(G,K^p)\otimes\bF_p^\ac,k_\lambda)$. Let $\Gamma(\fB\times\bfSh(G',K^p)(\bF_p^\ac),*)$ be the abelian group of $*$-valued functions on $\fB\times\bfSh(G',K^p)(\bF_p^\ac)$, which admits the Hecke action of $\bZ[\dT^{\tR\cup\{\fp\}}]$ via the second factor. We have a Chow cycle class map
\[
\Gamma(\fB\times\bfSh(G',K^{p})(\bF_p^\ac),\bZ)\to\CH^{(g+1)/2}(\bfSh(G,K^p)_{\FF_p^\ac})
\]
sending a function $f$ on $\fB\times\bfSh(G',K^{p})(\bF_p^\ac)$ to the Chow class of $\sum_{\fa,s}f(\fa,s)\pi_\fa^{-1}(s)$, which is $\bZ[\dT^{\tR\cup\{\fp\}}]$-equivariant. We will show that under certain ``large image'' assumption on the mod-$\lambda$ Galois representation attached to $\Pi$, the above Chow cycle class map (eventually) induces the following Abel--Jacobi map
\begin{align}\label{E:AJ3}
\Gamma(\fB\times\bfSh(G',K^p)(\bF_p^\ac),k_\lambda)/\fm
\to\rH^1(\bF_{p^{2g}},\rH^g(\bfSh(G,K^p)_{\FF^{\ac}_p},k_\lambda((g+1)/2))/\fm).
\end{align}
See Section \ref{ss:statement} for more details. The following theorem is what we call \emph{arithmetic level raising}.

\begin{theorem}[Theorem \ref{T:level_raising}]\label{T:main3}
Suppose that $p$ is a $\lambda$-level raising prime in the sense of Definition \ref{D:level-raising-prime}. In particular, we have the following equalities in $k_\lambda$:
\[
\phi_{\Pi,\lambda}(\rT_\fp)^2=(p^{g}+1)^2,\qquad
\phi_{\Pi,\lambda}(\rS_\fp)=1,
\]
where $\rT_\fp$ (resp.\ $\rS_\fp$) is the (spherical) Hecke operator at $\fp$ represented by $\big(\begin{smallmatrix}p&0\\ 0 &1\end{smallmatrix}\big)\in\GL_2(F_\fp)$ (resp.\ $\big(\begin{smallmatrix}p&0\\ 0 &p\end{smallmatrix}\big)\in\GL_2(F_\fp)$). Then the map \eqref{E:AJ3} is surjective.
\end{theorem}

As we will point out in Remarks \ref{R:assumption-wp} and \ref{R:level-raising-prime}, if there exist rational primes inert in $F$, and $\Pi$ is not dihedral and not isomorphic to a twist by a character of any of its internal conjugates, then for all but finitely many prime $\lambda$, there are infinitely many (with positive density) rational primes $p$ that are $\lambda$-level raising primes.

Suppose that the Jacquet--Langlands transfer of $\Pi$ to $B$ exists, say $\Pi_B$. If $(\Pi_B^{\infty,p})^{K^p}$ has dimension $1$ and there is no other automorphic representation of $B^\times(\bA_F)$ (of parallel weight $2$, unramified at $\fp$, and with nontrivial $K^p$-invariant vectors) congruent to $\Pi_B$ modulo $\lambda$, then the target of \eqref{E:AJ3} has dimension $\binom{g}{(g-1)/2}$ over $k_\lambda$.

\begin{remark}
In principle, our method can be applied to prove a theorem similar to Theorem \ref{T:main3} when $B$ is not necessarily totally indefinite but the ``supersingular locus'', defined in an \emph{ad hoc} way if $B$ is not totally indefinite, still appears in the near middle dimension. In fact, the proof of Theorem \ref{T:main3} will be reduced to the case where $B$ is indefinite at only one archimedean places (that is, the corresponding Shimura variety $\Sh(B)$ is a curve). However, we decide not to pursue the most general scenario as that would make the exposition much more complicated and technical. On the other hand, we would like to point out that arithmetic level raising when $1<\dim\Sh(B)<[F:\bQ]$ has arithmetic application as well, for example, to bound the triple product Selmer group (see the next section) with respect to the cubic extension $F/F^\flat$ of totally real number fields with $F^\flat\neq\bQ$.
\end{remark}

Let us explain the meaning of Theorem \ref{T:main3}. Suppose that $\Pi$ admits Jacquet--Langlands transfer, say $\Pi_B$, to $B^\times$ such that $\Pi_B^K\neq\{0\}$. Then the right-hand side of \eqref{E:AJ3} is \emph{nonzero}. In particular, under the assumption of Theorem \ref{T:main3}, the left-hand side of \eqref{E:AJ3} is nonzero as well. One can then deduce that there is an (algebraic) automorphic representation $\Pi'$ of $G'(\bA)=B'^\times(\bA_F)$ (trivial at $\infty$) such that the associated Galois representations of $\Pi'$ and $\Pi$ with coefficient $\cO_\bE/\lambda$ are isomorphic. However, it is obvious that $\Pi'$ cannot be the Jacquet--Langlands transfer of $\Pi$, as $B'$ is ramified at $\fp$ while $\Pi$ is unramified at $\fp$. In this sense, Theorem \ref{T:main3} reveals certain level raising phenomenon. Moreover, this theorem does not only prove the existence of level raising, but also provides an explicit way to realize the congruence relation behind the level raising through the Abel--Jacobi map \eqref{E:AJ3}. As this process involves cycle classes and local Galois cohomology, we prefer to call Theorem \ref{T:main3} \emph{arithmetic level raising}. This is crucial for our later arithmetic application. Namely, we will use this arithmetic level raising theorem to bound certain Selmer groups, as we will explain in the next section.

\subsection{Selmer group of triple product motive}
\label{S:main3}

In this section, we assume that $g=[F:\bQ]=3$; in other words, $F$ is a totally real cubic number field.

Let $E$ be an elliptic curve over $F$. We have the $\bQ$-motive $\otimes\Ind^F_\bQ\sfh^1(E)$ (with coefficient $\bQ$) of rank $8$, which is the multiplicative induction of the $F$-motive $\sfh^1(E)$ to $\bQ$. The \emph{cubic-triple product motive} of $E$ is defined to be
\begin{align*}
\sfM(E)\coloneqq\(\otimes\Ind^F_\bQ\sfh^1(E)\)(2).
\end{align*}
It is canonically polarized. For every prime $p$, the $p$-adic realization of $\sfM(E)$, denoted by $\sfM(E)_p$, is a Galois representation of $\bQ$ of dimension $8$ with $\bQ_p$-coefficients. In fact, up to a twist, it is the multiplicative induction from $F$ to $\bQ$ of the rational $p$-adic Tate module of $E$.

Now we assume that $E$ is modular. Then it gives rise to an irreducible cuspidal automorphic representation $\Pi_E$ of $(\Res_{F/\bQ}\GL_{2,F})(\bA)=\GL_2(\bA_F)$ with trivial central character. Denote by $\tau\colon\pres{L}{G}\to\GL_8(\bC)$ the triple product $L$-homomorphism \cite{PSR87}*{Section 0}, and $L(s,\Pi_E,\tau)$ the triple product $L$-function, which has a meromorphic extension to the complex plane by \cites{Gar87,PSR87}. Moreover, we have a functional equation
\[L(s,\Pi_E,\tau)=\epsilon(\Pi_E,\tau)C(\Pi_E,\tau)^{1/2-s}L(1-s,\Pi_E,\tau)\]
for some $\epsilon(\Pi_E,\tau)\in\{\pm 1\}$ and positive integer $C(\Pi_E,\tau)$. The global root number $\epsilon(\Pi_E,\tau)$ is the product of local ones: $\epsilon(\Pi_E,\tau)=\prod_v\epsilon(\Pi_{E,v},\tau)$, where $v$ runs over all places of $\bQ$. Here, we have $\epsilon(\Pi_{E,v},\tau)\in\{\pm 1\}$ and that it equals $1$ for all but finitely many $v$. Put
\[
\Sigma(\Pi_E,\tau)\coloneqq\{v\res \epsilon(\Pi_{E,v},\tau)=-1\}.
\]
In particular, the set $\Sigma(\Pi_E,\tau)$ contains $\infty$. We have $L(s,\sfM(E))=L(s+1/2,\Pi_E,\tau)$.

Now we assume that $E$ satisfies Assumption \ref{as:elliptic_curve}. In particular, $\Sigma(\Pi_E,\tau)$ has odd cardinality. Let $B^\flat$ be the indefinite quaternion algebra over $\bQ$ with the ramification set $\Sigma(\Pi_E,\tau)-\{\infty\}$, and put $B\coloneqq B^\flat\otimes_\bQ F$. Put $G\coloneqq\Res_{F/\bQ}B^\times$ as before. We will define neat open compact subgroups $K_\fr\subseteq G(\bA)$, indexed by certain integral ideals $\fr$ of $F$. We have the Shimura threefold $\Sh(G,K_\fr)$ over $\bQ$. Put $G^\flat\coloneqq(B^\flat)^\times$ and let $K_\fr^\flat\subseteq G^\flat(\bA)$ be induced from $K_\fr$. Then we have the Shimura curve $\Sh(G^\flat,K_\fr^\flat)$ over $\bQ$ with a canonical finite morphism to $\Sh(G,K_\fr)$. Using this $1$-cycle, we obtain, under certain conditions, a cohomology class
\[
\Theta_{p,\fr}\in\rH^1_f(\bQ,\sfM(E)_p)^{\oplus a(\fr)},
\]
where $\rH^1_f(\bQ,\sfM(E)_p)$ is the \emph{Bloch--Kato Selmer group} (Definition \ref{de:selmer}) of the Galois representation $\sfM(E)_p$ (with coefficient $\bQ_p$), and $a(\fr)>0$ is some integer depending on $\fr$. See Section \ref{ss:selmer} for more details of this construction. We have the following theorem on bounding the Bloch--Kato Selmer group using the class $\Theta_{p,\fr}$.

\begin{theorem}[Theorem \ref{th:selmer}]\label{T:main4}
Let $E$ be a modular elliptic curve over $F$ satisfying Assumption \ref{as:elliptic_curve}. For a rational prime $p$, if there exists a perfect pair $(p,\fr)$ such that $\Theta_{p,\fr}\neq 0$, then we have
\[\dim_{\bQ_p}\rH^1_f(\bQ,\sfM(E)_p)=1.\]
See Definition \ref{de:perfect_pair} for the meaning of perfect pairs, and also Remark \ref{re:perfect}.
\end{theorem}

The above theorem is closely related to the Bloch--Kato conjecture. We refer readers to the Introduction of \cite{Liu1} for the background of this conjecture, especially how Theorem \ref{T:main4} can be compared to the seminal work of Kolyvagin \cite{Kol90} and the parallel result \cite{Liu1}*{Theorem~1.5} for another triple product case. In particular, we would like to point out that under the (conjectural) triple product version of the Gross--Zagier formula and the Beilinson--Bloch conjecture on the injectivity of the Abel--Jacobi map, the following two statements should be equivalent:
\begin{itemize}
  \item $L'(0,\sfM(E))\neq 0$ (note that $L(0,\sfM(E))=0$); and
  \item there exists some $\fr_0$ such that for every other $\fr$ contained in $\fr_0$, we have $\Theta_{p,\fr}\neq 0$ as long as $(p,\fr)$ is a perfect pair.
\end{itemize}
Assuming this, then Theorem \ref{T:main4} implies that if $L'(0,\sfM(E))\neq 0$, that is, $\ord_{s=0}L(s,\sfM(E))=1$, then $\dim_{\bQ_p}\rH^1_f(\bQ,\sfM(E)_p)=1$ for all but finitely many $p$. This is certainly evidence toward the Bloch--Kato conjecture for the motive $\sfM(E)$.

At this point, it is not clear how the arithmetic level raising, Theorem \ref{T:main3}, is related to Theorem \ref{T:main4}. We will briefly explain this in the next section.

\subsection{Structure and strategies}

There are four chapters in the main part. In short words, Section \ref{ss:2} is responsible for the basics on Shimura varieties that we will use later; Section \ref{ss:3} is responsible for Theorems \ref{T:main1} and \ref{T:main2}; Section \ref{ss:4} is responsible for Theorem \ref{T:main3}; and Section \ref{ss:5} is responsible for Theorem \ref{T:main4}.

In Section \ref{ss:2}, we study certain Shimura varieties and their integral models attached to both unitary groups of rank $2$ and quaternion algebras, and compare them through Deligne's recipe of connected Shimura varieties. The reason we have to study unitary Shimura varieties is the following: In the proof of Theorems \ref{T:main1}, \ref{T:main2} and \ref{T:main3}, we have to use induction process to go through certain quaternionic Shimura varieties associated to $B$ that are \emph{not} totally indefinite. Those Shimura varieties are not (coarse) moduli spaces but we still want to carry the information from moduli interpretation through the induction process. Therefore, we use the technique of changing Shimura data by studying closely related unitary Shimura varieties, which are of PEL-type. Such argument is coherent with \cite{TX1} in which the authors study Goren--Oort stratification on quaternionic Shimura varieties.

In Section \ref{ss:3}, we first construct candidates for the supersingular locus in Theorem \ref{T:main1} via Goren--Oort strata, which are studied in \cite{TX1}, and then prove that they exactly form the entire supersingular locus, both through an induction argument. As we mentioned previously, during the induction process, we need to compare quaternionic Shimura varieties to unitary ones. At last, we identify and prove certain properties for the superspecial locus, in some special cases.

In Section \ref{ss:4}, we state and prove the arithmetic level raising result. Using the non-degeneracy of certain intersection matrix proved in \cite{TX2}, we can reduce  Theorem \ref{T:main3} to establishing a similar isomorphism on certain quaternionic Shimura curves. Then we use the well-known argument of Ribet together with Ihara's lemma in this context to establish such isomorphism on curves.

In Section \ref{ss:5}, we focus on the number theoretical application of the arithmetic level raising established in the previous chapter. The basic strategy to bound the Selmer group follows the same line as in \cites{Kol90,Liu1,Liu2}. Namely, we construct a family of cohomology classes $\Theta^\nu_{p,\fr,\underline\ell}$ to serve as annihilators of the Selmer group after quotient by the candidate class $\Theta_{p,\fr}$ in rank $1$ case. In the case considered here, those cohomology classes are indexed by an integer $\nu$ as the depth of congruence, and a pair of rational primes $\underline\ell=(\ell,\ell')$ that are ``$p^\nu$-level raising primes'' (see Definition \ref{de:admissible} for the precise terminology and meaning). The key ideal is to connect $\Theta_{p,\fr}$ and various $\Theta^\nu_{p,\fr,\underline\ell}$ through some objects in the middle, that is, some mod-$p^\nu$ modular forms on certain Shimura set. Following past literature, the link between $\Theta_{p,\fr}$ and those mod-$p^\nu$ modular forms is called the \emph{second explicit reciprocity law}; while the link between $\Theta^\nu_{p,\fr,\underline\ell}$ and those mod-$p^\nu$ modular forms is called the \emph{first explicit reciprocity law}. The first law in this context has already been established by one of us in \cite{Liu2}. To establish the second law, we use Theorem \ref{T:main3}; namely, we have to compute the corresponding element in the left-hand side in the isomorphism of Theorem \ref{T:main3} of the image of $\Theta_{p,\fr}$ in the right-hand side.


\subsection{Notation and conventions}
\label{ss:notation}

The following list contains basic notation and conventions we fix throughout the article. We will usually not recall them when we use, as most of them are common.

\begin{itemize}
  \item Let $\Lambda$ be an abelian group and $\cS$ a finite set. We denote by $\Gamma(\cS,\Lambda)$ the abelian group of $\Lambda$-valued functions on $\cS$.

  \item For a finite set $\tS$, we denote by $|S|$ its cardinality.

  \item If a base is not specified in the tensor operation $\otimes$, then it is $\bZ$. For an abelian group $A$, put $\widehat{A}\coloneqq A\otimes(\varprojlim_n\bZ/n)$. In particular, we have $\widehat\bZ=\prod_{l}\ZZ_l$, where $l$ runs over all rational primes. For a fixed rational prime $p$, we put $\widehat\ZZ^{(p)}\coloneqq\prod_{l\neq p}\ZZ_l$.

  \item We denote by $\bA$ the ring of ad\`eles over $\QQ$. For a set $\Box$ of places of $\bQ$, we denote by $\bA^\Box$ the ring of ad\`eles away from $\Box$. For a number field $F$, we put $\bA^\Box_F\coloneqq\bA^\Box\otimes_\bQ F$. If $\Box=\{v_1,\dots,v_n\}$ is a finite set, we will also write $\bA^{v_1,\dots,v_n}$ for $\bA^\Box$.

  \item For a field $K$, denote by $K^\ac$ the algebraic closure of $K$ and put $\rG_K\coloneqq\Gal(K^\ac/K)$. Denote by $\bQ^\ac$ the algebraic closure of $\bQ$ in $\bC$. When $K$ is a subfield of $\bQ^\ac$, we take $\rG_K$ to be $\Gal(\bQ^\ac/K)$ hence a subgroup of $\rG_\bQ$.

  \item For a number field $K$, we denote by $\cO_K$ the ring of integers in $K$. For every finite place $v$ of $\cO_K$, we denote by $\cO_{K,v}$ the ring of integers of the completion of $K$ at $v$.

  \item If $K$ is a local field, then we denote by $\cO_K$ its ring of integers, $\rI_K\subseteq\rG_K$ the inertia subgroup. If $v$ is a rational prime, then we simply write $\rG_v$ for $\rG_{\bQ_v}$ and $\rI_v$ for $\rI_{\bQ_v}$.

  \item Let $K$ be a local field, $\Lambda$ a ring, and $N$ a $\Lambda[\rG_K]$-module. We have an exact sequence of $\Lambda$-modules
      \begin{align*}
      0\to\rH^1_\unr(K,N)\to\rH^1(K,N)\xrightarrow{\partial}\rH^1_\sing(K,N)\to0,
      \end{align*}
      where $\rH^1_\unr(K,N)$ is the submodule of unramified classes.

  \item Let $\Lambda$ be a ring and $N$ a $\Lambda[\rG_\bQ]$-module. For each prime power $v$, we have the localization map $\loc_v\colon\rH^1(\bQ,N)\to\rH^1(\bQ_v,N)$ of $\Lambda$-modules.

  \item Denote by $\dP^1$ the projective line scheme over $\bZ$, and $\dG_m=\Spec\bZ[T,T^{-1}]$ the multiplicative group scheme.

  \item Let $X$ be a scheme. The cohomology group $\rH^\bullet(X,-)$ will always be computed on the \'{e}tale site of $X$. If $X$ is of finite type over a subfield of $\bC$, then $\rH^\bullet(X(\bC),-)$ will be understood as the Betti cohomology of the associated complex analytic space $X(\bC)$.
\end{itemize}

\subsubsection*{Acknowledgements}

The authors would like to thank Liang~Xiao for his collaboration in some previous work with Y.~T., which plays an important role in the current one. They also thank Mladen~Dimitrov for helpful discussion concerning his work of cohomology of Hilbert modular varieties. Y.~L. would like to thank the hospitality of Universit\"{a}t Bonn during his visit to Y.~T. where part of the work was done. Y.~L. is partially supported by NSF grant DMS--1702019 and a Sloan Research Fellowship.

\section{Shimura varieties and moduli interpretations}
\label{ss:2}

In this chapter, we study certain Shimura varieties and their integral models attached to both unitary groups of rank $2$ and quaternion algebras, and compare them through Deligne's recipe of connected Shimura varieties.

Let $F$ be a totally real number field, and $p\geq3$ a rational prime unramified in $F$. Denote by $\Sigma_{\infty}=\Hom_{\QQ}(F,\CC)$ the set of archimedean places of $F$, and $\Sigma_p$ the set of $p$-adic places  of $F$ above $p$. We fix throughout Section \ref{ss:2} and Section \ref{ss:3} an isomorphism $\iota_p\colon\CC\xrightarrow\sim\QQ_{p}^{\ac}$. Via $\iota_p$, we identify $\Sigma_{\infty}$ with the set of $p$-adic embeddings of $F$ via $\iota_p$. For each $\fp\in\Sigma_p$, we put $g_{\fp}\coloneqq[F_{\fp}:\QQ_p]$ and denote by $\Sigma_{\infty/\fp}$ the subset of $p$-adic embeddings that induce $\fp$, so that we have
\[
\Sigma_{\infty}=\coprod_{\fp\in\Sigma_p}\Sigma_{\infty/\fp}.
\]
Since $p$ is unramified in $F$, the Frobenius, denoted by $\sigma$, acts as a cyclic permutation on each $\Sigma_{\infty/\fp}$.

We fix also a totally indefinite quaternion algebra $B$ over $F$ such that $B$ splits at all places of $F$ above $p$.

\subsection{Quaternionic Shimura varieties}
\label{S:quaternion-shimura}

Let $\tS$ be a subset of $\Sigma_{\infty}\cup\Sigma_{p}$ of even cardinality, and put $\tS_{\infty}\coloneqq\tS\cap \Sigma_{\infty}$. For each $\fp\in\Sigma_p$, we put $\tS_{\fp}\coloneqq\tS\cap(\Sigma_{\infty/\fp}\cup \{\fp\})$ and $\tS_{\infty/\fp}=\tS\cap\Sigma_{\infty/\fp}$. We suppose that $\tS_{\fp}$ satisfies the following assumptions.

\begin{assumption}\label{A:assumption-S}
Take $\fp\in\Sigma_p$.
\begin{enumerate}
  \item If $\fp\in\tS$, then $g_{\fp}$ is odd and $\tS_{\fp}=\Sigma_{\infty/\fp}\cup\{\fp\}$.

  \item If $\fp\notin \tS$, then $\tS_{\infty/\fp}$ is a disjoint union of chains of even cardinality under the  Frobenius action on $\Sigma_{\infty/\fp}$, that is, either $\tS_{\fp}=\Sigma_{\infty/\fp}$ has even cardinality or there exist $\tau_1,\dots, \tau_r\in\Sigma_{\infty/\fp}$ and integers $m_1,\dots, m_r\geq 1$ such that
      \begin{equation}\label{E:decomposition-S}
      \tS_{\fp}=\coprod_{i=1}^r\{\tau_i,\sigma^{-1}\tau_i, \dots, \sigma^{-2m_i+1}\tau_i\}
      \end{equation}
      and $\sigma\tau_i,\sigma^{-2m_i}\tau_i\not\in\tS_{\fp}$.
\end{enumerate}
\end{assumption}

Let $B_{\tS}$ denote the quaternion algebra over $F$ whose ramification set is the union of $\tS$ with the ramification set of $B$. We put $G_{\tS}\coloneqq\Res_{F/\bQ}(B^{\times}_{\tS})$. For $\tS=\emptyset$, we usually write $G=G_{\emptyset}$. Then $G_{\tS}$ is isomorphic to $G$ over $F_{v}$ for every place $v\notin \tS$, and we fix an isomorphism
\[
G_{\tS}(\bA^{\infty,p})\cong G(\bA^{\infty,p}).
\]
Let $\tT$ be a subset of $\tS_{\infty}$, and  $\tT_{\fp}=\tS_{\infty/\fp}\cap \tT$ for each $\fp\in \Sigma_p$. Throughout this paper, we will always assume that $|\tT_{\fp}|=\#\tS_{\fp}/2$.
Consider the Deligne homomorphism
\[
\xymatrix@R=0pt{
h_{\tS,\tT}\colon \dS(\bR)=\bC^{\times}\ar[rr] && G_{\tS}(\bR)\cong \GL_2(\bR)^{\Sigma_{\infty}-\tS_{\infty}}\times (\bH^{\times})^{\tT}\times (\bH^{\times})^{\tS_{\infty}-\tT}\\
x+\sqrt{-1}y\ar@{|->}[rr] &&
{\bigg(\big(\begin{smallmatrix}x&y\\ -y &x\end{smallmatrix}\big)^{\Sigma_{\infty}-\tS_{\infty}}, (x^2+y^2)^{ \tT}, 1^{\tS_{\infty}-\tT}\bigg)}
}
\]
where $\bH$ denotes the Hamiltonian algebra over $\bR$. Then $G_{\tS,\tT}\coloneqq(G_{\tS},h_{\tS,\tT})$ is a Shimura datum, whose reflex field $F_{\tS,\tT}$  is the subfield of the Galois closure of $F$ in $\CC$ fixed by the  subgroup stabilizing both $\tS_{\infty}$ and $\tT$. For instance, if $\tS_\infty=\emptyset$, then $\tT=\emptyset$ and  $F_{\tS}=\QQ$. Let $\wp$ denote the $p$-adic place of $F_{\tS,\tT}$ via the embedding $F_{\tS,\tT}\hookrightarrow\CC\xra{\sim}\QQ_p^{\ac}$.

In this article, we fix an open compact subgroup $K_p=\prod_{\fp\in\Sigma_p}K_{\fp}\subseteq G_{\tS}(\QQ_p)=\prod_{\fp\in\Sigma_p}(B_{\tS}\otimes_{F}F_{\fp})^{\times}$, where
\begin{itemize}
  \item $K_{\fp}$ is a hyperspecial subgroup if $\fp\notin \tS$, and

  \item $K_{\fp}=\cO^{\times}_{B_{\fp}}$ is the unique maximal open compact subgroup of $(B_{\tS}\otimes_{F}F_{\fp})^{\times}$ if $\fp\in\tS$.
\end{itemize}
For a sufficiently small open compact subgroup $K^p\subseteq G(\bA^{\infty,p})\cong G_{\tS}(\bA^{\infty,p})$, we have the Shimura variety $\Sh(G_{\tS,\tT},K^p)$ defined over $F_{\tS}$ whose $\CC$-points are given by
\[
\Sh(G_{\tS,\tT}, K^p)(\CC)=G_{\tS}(\QQ)\backslash(\fH^{\pm})^{\Sigma_{\infty}-\tS_{\infty}}\times G_{\tS}(\bA^{\infty})/K^p K_p
\]
where $K=K^pK_p\subseteq G(\bA^{\infty})$, and $\fH^{\pm}=\dP^1(\CC)-\dP^1(\RR)$ is the union of upper and lower half-planes. Note that the  algebraic variety $\Sh(G_{\tS,\tT}, K^p)_{ \QQ^{\ac}}$ over $\QQ^{\ac}$ is independent of $\tT$, but different choices of $\tT$ will give rise to different actions of $\Gal( \QQ^{\ac}/F_{\tS,\tT})$ on $\Sh(G_{\tS,\tT}, K^p)_{ \QQ^{\ac}}$.

When $\tS_{\infty}=\Sigma_{\infty}$, $\Sh(G_{\tS,\tT}, K^p)(\QQ^{\ac})$ is a discrete set and the action of $\Gamma_{F_{\tS,\tT}}\coloneqq\Gal(\QQ^{\ac}/F_{\tS,\tT})$ is given as follows. Note that the Deligne homomorphism $h_{\tS,\tT}$ factors through the center $T_{F}=\Res_{F/\QQ}(\dG_m)\subseteq G_{\tS}$, and the action of $\Gamma_{F_{\tS,\tT}}$  factors thus through its maximal abelian quotient $\Gamma_{F_{\tS,\tT}}^{\mathrm{ab}}$. Let $\mu\colon\dG_{m,F_{\tS,\tT}}\to T_{F}\otimes_{\QQ}F_{\tS,\tT}$ be the Hodge cocharacter (defined over the reflex field $F_{\tS,\tT}$) associated with $h_{\tS,\tT}$. Let  $\mathrm{Art}\colon\bA_{F_{\tS,\tT}}^{\infty,\times}\to \Gamma^{\mathrm{ab}}_{F_{\tS,\tT}}$ denote the  Artin reciprocity map that sends uniformizers to geometric Frobenii. Then the action of  $\Art(g)$ on $\Sh(G_{\tS,\tT}, K^p)(\QQ^{\ac})$ is given by the multiplication by the image of $g$ under the composite map
\[
\bA_{F_{\tS,\tT}}^{\infty,\times}\xra{\mu} T_F(\bA_{F_{\tS,\tT}}^{\infty})=(F\otimes_{\QQ} \bA_{F_{\tS,\tT}}^{\infty})^{\times}\xra{\rN_{F_{\tS,\tT}/\QQ}}\bA_{F}^{\infty,\times}\subseteq G_{\tS}(\bA^{\infty}).
\]
If $\widetilde F$ denotes the Galois closure of $F$ in $\CC$, then the restriction of the action of $\Gamma_{F_{\tS,\tT}}$ to $\Gamma_{\widetilde{F}}$ depends only on $\#\tT$.

We put $\Sh(G_{\tS,\tT})\coloneqq\varprojlim_{K^p}\Sh(G_{\tS,\tT},K^p)$. Let $\Sh(G_{\tS,\tT})^{\circ}$ be the neutral geometric connected component of $\Sh(G_{\tS,\tT})\otimes_{F_\tS}\QQ^{\ac}$, that is, the one containing the image of point
\[
(i^{\Sigma_{\infty}-\tS_{\infty}},1)\in(\fH^{\pm})^{\Sigma_{\infty}-\tS_{\infty}}\times G_{\tS}(\bA^{\infty}).
\]
Then $\Sh(G_{\tS,\tT})^{\circ}\otimes_{\QQ^{\ac},\iota_p}{\QQ_p^{\ac}}$ descends to $\QQ_p^\ur$, the maximal unramified extension of $\QQ_p$ in $\bQ_p^\ac$. Moreover, by Deligne's construction \cite{De79}, $\Sh_{K_p}(G_{\tS,\tT})$ can be recovered from the connected Shimura variety $\Sh(G_{\tS,\tT})^{\circ}$ together with its Galois and Hecke actions (see \cite{TX1}*{2.11} for details in our particular case).

\subsection{An auxiliary CM extension}
\label{S:signature-numbers}

Choose a CM extension $E/F$ such that
\begin{itemize}
  \item $E/F$ is inert at every  place of $F$ where $B$ is ramified,

  \item For $\fp\in\Sigma_p$, $E/F$ is split (resp.\ inert) at $\fp$ if $g_{\fp}$ is even (resp.\ if $g_{\fp}$ is odd).
\end{itemize}
Let $\Sigma_{E,\infty}$ denote the set of complex embeddings of $E$, identified also with the set of $p$-embeddings of $E$ by composing with $\iota_p$. For $\tilde\tau\in\Sigma_{E,\infty}$, we denote by $\tilde\tau^c$ the complex conjugation of $\tilde\tau$. For $\fp\in\Sigma_p$, we denote by $\Sigma_{E,\infty/\fp}$ the subset of $p$-adic embeddings of $E$ inducing $\fp$. Similarly, for a $p$-adic place $\fq$ of $E$, we have the subset $\Sigma_{E,\infty/\fq}\subseteq \Sigma_{E,\infty}$ consisting of $p$-adic embeddings that induce $\fq$.

\begin{assumption}\label{A:assumption-S-tilde}
Consider a subset $\tilde\tS_{\infty}\subseteq \Sigma_{E,\infty}$ satisfying the following
\begin{enumerate}
  \item For each $\fp\in\Sigma_p$, the natural restriction map $\Sigma_{E,\infty/\fp}\to \Sigma_{\infty/\fp}$ induces a bijection $\tilde\tS_{\infty/\fp}\xra{\sim}\tS_{\infty/\fp}$, where $\tilde\tS_{\infty/\fp}=\tilde\tS_{\infty}\cap\Sigma_{E,\infty/\fp}$.

  \item For each $p$-adic place $\fq$ of $E$ above a $p$-adic place $\fp$ of $F$, the cardinality of $\tilde\tS_{\infty/\fq}$ is half of the cardinality of the preimage of $\tS_{\infty/\fp}$ in $\Sigma_{E,\infty/\fq}$.
\end{enumerate}
\end{assumption}
For instance, if $\fp$ splits in $E$ into two places $\fq$ and $\fq^c$ and $\tS_{\fp}$ is given by \eqref{E:decomposition-S}, then the subset
\[
\tilde\tS_{\infty/\fp}=\coprod_{i=1}^r\{\tilde\tau_{i}, \sigma^{-1}\tilde\tau_{i}^c, \dots, \sigma^{-2m_i+2}\tilde\tau_{i},\sigma^{-2m_i+1}\tilde\tau_{i}^c\}
\]
satisfies the requirement. Here, $\tilde\tau_i\in\Sigma_{E,\infty/\fp}$ denotes the lift of $\tau_i$ inducing the $p$-adic place $\fq$. The choice of such a $\tilde\tS_{\infty}$ determines a collection of numbers $s_{\tilde\tau}\in\{0,1,2\}$ for $\tilde\tau\in\Sigma_{E,\infty}$ by the following rules:
\[
s_{\tilde\tau}=\begin{cases}0 &\text{if }\tilde\tau\in\tilde\tS_{\infty},\\
2 &\text{if }\tilde\tau^c\in\tilde\tS_{\infty},\\
1 &\text{otherwise.}\end{cases}
\]
Our assumption on $\tilde\tS_{\infty}$ implies that, for every prime $\fq$ of $E$ above $p$, the set $\{\tilde\tau\in\Sigma_{E,\infty/\fq}\res s_{\tilde\tau}=0\}$ has the same cardinality as $\{\tilde\tau\in\Sigma_{E,\infty/\fq}\res s_{\tilde\tau}=2\}$.

Put $\tilde\tS\coloneqq(\tS,\tilde\tS_{\infty})$ and $T_{E}\coloneqq\Res_{E/\QQ}(\dG_m)$. Consider the
Deligne homomorphism
\[
\xymatrix@R=0pt{
h_{E,\tilde\tS, \tT}\colon  \dS(\RR) = \CC^\times \ar[rr] &&
T_E(\RR) = \prod_{\tau\in \Sigma_{\infty}}(E\otimes_{F,\tau}\RR)^{\times}\cong (\CC^{\times})^{\tS_{\infty}-\tT}\times (\CC^{\times})^{\tT}\times (\CC^{\times})^{\tS_{\infty}^c}
\\
z=x+\sqrt{-1}y \ar@{|->}[rr] && \bigg( (\bar z,\dots, \bar z),(z^{-1},\dots, z^{-1}), (1,\dots, 1)\bigg).}
\]
where, for each $\tau\in\tS_{\infty}$, we identify $E\otimes_{\tau,F}\RR$ with $\CC$ via the embedding $\tilde\tau\colon E\hookrightarrow\CC$ with $\tilde\tau \in\tilde\tS_{\infty}$ lifting $\tau$.
We write $T_{E,\tilde\tS,  \tT}=(T_E, h_{E,\tilde \tS,\tT})$ and put $K_{E,p}\coloneqq(\cO_{E}\otimes\ZZ_p)^{\times}\subseteq T_{E}(\QQ_p)$, the unique maximal open compact subgroup of $T_{E}(\QQ_p)$. For each open compact subgroup $K_{E}^p\subseteq T_{E}(\bA^{\infty,p})$, we have the zero-dimensional  Shimura variety $\Sh(T_{E,\tilde\tS,\tT},K_E)$ whose $\QQ^{\ac}$-points are given by
\[
\Sh(T_{E,\tilde\tS, \tT}, K_E)(\QQ^{\ac})=E^{\times}\backslash T_{E}(\bA^{\infty})/K_E^pK_{E,p}.
\]

\subsection{Unitary Shimura varieties}
\label{S:unitary-Sh-var}

Put $T_F\coloneqq\Res_{F/\bQ}(\dG_{m,F})$. Then the reduced norm on $B_{\tS}$ induces a morphism of $\bQ$-algebraic groups
\[
\nu_{{\tS}}\colon G_{\tS}\to T_{F}.
\]
Note that the center of $G_{\tS}$ is isomorphic to $T_{F}$. Let $G''_{\tilde\tS,\tT}$ denote the quotient of $G_{\tS}\times T_{E}$ by $T_{F}$ via the embedding
\[
T_F\hookrightarrow G_{\tS}\times T_{E},\quad z\mapsto (z, z^{-1}),
\]
and let $G'_{\tilde\tS}$ be the inverse image of $\dG_m\subseteq T_F$ under the norm map
\[
\Nm\colon G''_{\tilde\tS}=(G_{\tS}\times T_{E})/T_F\to T_F,\quad (g,t)\mapsto \nu_{\tS}(g)\Nm_{E/F}(t).
\]
Here, the subscript $\tilde \tS$ is to emphasize that we will take the Deligne homomorphism $h_{\tilde\tS}''\colon \CC^{\times}\to G_{\tilde\tS}''(\RR)$ induced by $h_{\tS,\tT}\times h_{E,\tilde\tS,\tT}$, which is independent of $\tT$.  Note that the image of $h''_{\tilde\tS}$ lies in $G'_{\tilde\tS}(\RR)$, and we denote by $h'_{\tilde\tS}\colon \CC^{\times}\to G'_{\tilde\tS}(\RR)$ the induced map.

As for the quaternionic case, we fix the level at $p$ of the Shimura varieties for $G''_{\tilde\tS}$ and $G'_{\tilde\tS}$ as follows. Let $K_p''\subseteq G''_{\tilde\tS}(\QQ_p)$ be the image of $K_{p}\times K_{E,p}$, and put $K_{p}'\coloneqq K_p''\cap G'_{\tilde\tS}(\QQ_p)$. Note that $K''_p$ (resp.\ $K'_p$) is not a maximal open compact subgroup of $G''_{\tilde\tS}(\QQ_p)$ (resp.\ $G'_{\tilde\tS}(\QQ_p)$), if $\tS$ contains some $p$-adic place $\fp\in\Sigma_p$. For sufficiently open compact subgroups $K''^p\subseteq G''_{\tilde\tS}(\bA^{\infty,p})$ and $K'^p\subseteq G'_{\tilde\tS}(\bA^{\infty,p})$, we get Shimura varieties with $\CC$-points given by
\begin{gather*}
\Sh(G''_{\tilde\tS}, K''^p)(\CC)=G_{\tilde\tS}''(\QQ)\backslash (\fH^{\pm})^{\Sigma_{\infty}-\tS_{\infty}}\times G''_{\tilde\tS}(\bA^{\infty})/K''^pK''_p,\\
\Sh(G'_{\tilde\tS}, K'^p)(\CC)=G_{\tilde\tS}'(\QQ)\backslash(\fH^{\pm})^{\Sigma_{\infty}-\tS_{\infty}}\times G'_{\tilde\tS}(\bA^{\infty})/K'^p K'_p.
\end{gather*}
We put
\[
\Sh(G''_{\tilde\tS})\coloneqq\varprojlim_{K''^p}\Sh(G''_{\tilde\tS},K''^p),\quad \Sh(G'_{\tilde\tS})=\varprojlim_{K'^p}\Sh(G'_{\tilde\tS,\tT}, K'^p).
\]
The common reflex field $E_{\tilde\tS}$ of $\Sh(G_{\tilde\tS}')$ and $\Sh(G''_{\tilde\tS})$ is a subfield of the Galois closure of $E$ in $\CC$. The isomorphism $\iota_p\colon \CC\xrightarrow\sim\QQ_p^{\ac}$ defines a $p$-adic embedding of $E_{\tilde\tS}\hookrightarrow \QQ_{p}^{\ac}$, and hence a $p$-adic place $\tilde\wp$ of $E_{\tilde\tS}$. Then $E_{\tilde\tS}$ is unramified at $\tilde\wp$. Let $\Sh(G''_{\tilde\tS})^{\circ}$ (resp.\ $\Sh(G'_{\tilde\tS})^{\circ}$) denote the neutral geometric connected component of $\Sh(G''_{\tilde\tS})\otimes_{E_{\tilde\tS}}\bQ^\ac$ (resp.\ $\Sh(G'_{\tilde\tS})\otimes_{E_{\tilde\tS}}\bQ^\ac$). Then both $\Sh(G''_{\tilde\tS})^{\circ}\otimes_{\bQ^\ac,\iota_p}\bQ_p^\ac$ and $\Sh(G'_{\tilde\tS})^{\circ}\otimes_{\bQ^\ac,\iota_p}\bQ_p^\ac$ can be descended to $\QQ_p^\ur$.

In summary, we have a diagram of morphisms of algebraic groups
\[
G_{\tS}\leftarrow G_{\tS}\times T_{E}\rightarrow G''_{\tilde\tS}=(G_{\tS}\times T_{E})/T_F\leftarrow G'_{\tilde\tS}
\]
compatible with Deligne homomorphisms, such that the induced morphisms on the derived and adjoint groups are isomorphisms. By Deligne's theory of connected Shimura varieties (see \cite{TX1}*{Corollary~2.17}), such a diagram induces canonical isomorphisms between the neutral geometric connected components of the associated Shimura varieties:
\begin{equation}\label{E:isom-conn-Sh}
\Sh(G_{\tS,\tT})^{\circ}\xleftarrow{\sim} \Sh(G''_{\tilde\tS})^{\circ}\xrightarrow{\sim}\Sh(G'_{\tilde\tS})^{\circ}.
\end{equation}
Since a Shimura variety can be recovered from its neutral connected component together with its Hecke and Galois actions, one can transfer integral models of $\Sh(G'_{\tilde\tS})$ to integral models of $\Sh(G_{\tS,\tT})$ (see \cite{TX1}*{Corollary~2.17}).

\subsection{Moduli interpretation for unitary Shimura varieties}
\label{S:unitary-moduli}

Note that $\Sh(G'_{\tilde\tS},K'^p)$ is a Shimura variety of PEL-type. To simplify notation, let $\cO_{\tilde\wp}$ be the ring of integers of the completion of  $E_{\tilde\tS}$ at $\tilde\wp$. We recall the integral model of $\Sh(G'_{\tilde\tS},K'^p)$ over $\cO_{\tilde\wp}$ defined in \cite{TX1} as follows.


Let $K'^p\subseteq G'_{\tilde\tS}(\bA^{\infty,p})$ be an open compact subgroup such that $K'^pK'_p$ is neat (for PEL-type Shimura data). We put $D_{\tS}\coloneqq B_{\tS}\otimes_{F}E$, which is isomorphic to $\Mat_2(E)$ by assumption on $E$. Denote by $b\mapsto \bar b$ the involution on $D_{\tS}$ given by the product of the canonical involution on $B_{\tS}$ and the complex conjugation on $E/F$. Write $E=F(\sqrt{\fd})$ for some totally negative element $\fd\in F$ that is a $\fp$-adic unit for every $\fp\in\Sigma_p$. We choose also an element $\delta\in D_{\tS}^{\times}$ such that $\bar\delta=\delta$ as in \cite{TX1}*{Lemma~3.8}. Then the conjugation by $\delta^{-1}$ defines a new involution $b\mapsto b^*=\delta^{-1}\bar b\delta$. Consider $W=D_{\tS}$ as a free left $D_{\tS}$-module of rank $1$, equipped with an $*$-hermitian alternating pairing
\begin{equation}\label{E:unitary-alternating}
\psi\colon W\times W\to \QQ, \quad \psi(x,y)=\Tr_{E/\QQ}(\Tr_{D_{\tS}/E}^{\circ}(\sqrt{\fd}x\bar{y}\delta)),
\end{equation}
where $\Tr_{D_{\tS}/E}^{\circ}$ denotes the reduced trace of $D_{\tS}/E$. Then $G'_{\tilde\tS,\tT}$ can be identified with the unitary similitude group of $(W,\psi)$.

We choose an order $\cO_{D_{\tS}}\subseteq D_{\tS}$ that is stable under $*$ and maximal at $p$, and an $\cO_{D_{\tS}}$-lattice $L\subseteq W$ such that $\psi(L,L)\subseteq \ZZ$ and $L\otimes \ZZ_p$ is self-dual under $\psi$. Assume that $K'^p$ is a sufficiently small open compact subgroup of $G'_{\tilde\tS,\tT}(\bA^{\infty,p})$ which stabilizes $L\otimes \widehat{\ZZ}^{(p)}$.


Consider the moduli problem $\underline{\bfSh}(G'_{\tilde\tS},K'^p)$ that associates to each locally noetherian $\cO_{\tilde\wp}$-scheme $S$ the set of isomorphism classes of tuples $(A,\iota,\lambda,\bar\alpha_{K'^p})$, where
\begin{itemize}
  \item $A$ is an abelian scheme over $S$ of dimension $4[F:\QQ]$;

  \item $\iota\colon \cO_{D_{\tS}}\hookrightarrow \End_S(A)$ is an embedding such that the induced action of $\iota(b)$ for $b\in\cO_{E}$ on $\Lie(A/S)$ has characteristic polynomial
      \[
      \det(T-\iota(b)|\Lie(A/S))=\prod_{\tilde\tau\in\Sigma_{E,\infty}}(x-\tilde\tau(b))^{2s_{\tilde\tau}};
      \]

  \item $\lambda\colon A\to A^{\vee}$ is a polarization of $A$ such that
      \begin{itemize}
        \item the Rosati involution defined by $\lambda$ on $\End_S(A)$ induces the involution $b\mapsto b^*$ on $\cO_{D_{\tS}}$,
        \item if $\fp\notin \tS$, $\lambda$ induces an isomorphism of $p$-divisible groups $A[\fp^{\infty}]\xra{\sim} A^\vee[\fp^{\infty}]$, and
        \item if $\fp\in\tS$, then $(\ker \lambda)[\fp^{\infty}]$  is  a finite flat group scheme contained in $A[\fp]$ of rank $p^{4g_{\fp}}$ and the cokernel of induced morphism $\lambda_*\colon\rH^{\deR}_1(A/S)\to\rH^{\deR}_1(A^\vee/S)$ is a locally free module of rank two over $\cO_{S}\otimes_{\ZZ_p}\cO_E/\fp$. Here, $\rH^{\dr}_1(-/S)$ denotes the relative de Rham homology;
      \end{itemize}

  \item $\bar\alpha_{K'^p}$ is a $K'^p$ level structure on $A$, that is, a $K'^p$-orbit of $\cO_{D_{\tS}}$-linear isomorphisms of \'etale sheaves $\alpha\colon L\otimes\widehat{\ZZ}^{(p)}\xra{\sim}\widehat{T}^p(A)$ such that the alternating pairing $\psi\colon L\otimes\widehat{\ZZ}^{(p)}\times L\otimes\widehat{\ZZ}^{(p)}\to \widehat{\ZZ}^{(p)}$ is compatible with the $\lambda$-Weil pairing on $\widehat{T}^p(A)$ via some isomorphism $\widehat{\ZZ}^{(p)}\cong \widehat{\ZZ}^{(p)}(1)$. Here, $\widehat{T}^p(A)=\prod_{l\neq p}T_l(A)$ denotes the product of prime-to-$p$ Tate modules.
\end{itemize}

\begin{remark}\label{R:rational-moduli-unitary}
Sometimes it is convenient to  formulate the moduli problem $\underline{\bfSh}(G'_{\tilde\tS}, K'^p)$ in terms of isogeny classes of abelian varieties: one associates to each locally noetherian $\cO_{\tilde\wp}$-scheme $S$ the equivalence classes of tuples $(A,\iota,\lambda,\bar\alpha^{\rat}_{K'^p})$, where
  \begin{itemize}
    \item $(A,\iota)$ is an abelian scheme \emph{up to prime-to-$p$ isogenies} of dimension $4[F:\QQ]$ equipped with an action $\cO_{D_{\tS}}$ satisfying the determinant conditions as above;

    \item $\lambda$ is a polarization on $A$ satisfying the condition as above;

    \item $\bar\alpha^{\rat}_{K'^p}$ is a rational $K'^p$-level structure on $A$, that is, a $K'^p$-orbit of $\cO_{D_{\tS}}\otimes\bA^{\infty,p}$-linear isomorphisms of \'etale sheaves on $S$:
        \[
        \alpha\colon W\otimes_{\QQ} \bA^{\infty,p}\xra{\sim }\widehat V^p(A)\coloneqq\widehat{T}^p(A)\otimes\bQ
        \]
        such that the pairing $\psi$ on $W\otimes_{\QQ} \bA^{\infty,p}$ is compatible with the $\lambda$-Weil pairing on $\widehat{V}^p(A)$  \emph{up to a scalar in $\bA^{\infty,p,\times}$}.
\end{itemize}
For the equivalence of these two definitions, see \cite{Lan13}.
\end{remark}


\begin{theorem}[\cite{TX1}*{3.14, 3.19}]
The moduli problem $\underline{\bfSh}(G'_{\tilde\tS}, K'^p)$ is representable by a quasi-projective and smooth scheme $\bfSh(G'_{\tilde\tS}, K'^p)$ over $\cO_{\tilde\wp}$ such that
\[
\bfSh(G'_{\tilde\tS}, K'^p)\times_{\cO_{\tilde\wp}}E_{\tilde\tS,\tilde\wp}\cong \Sh(G'_{\tilde\tS}, K'^p)\times_{E_{\tilde\tS}}E_{\tilde\tS,\tT, \tilde\wp}.
\]
Moreover, the projective limit $\bfSh(G'_{\tilde\tS})\coloneqq\varprojlim_{K'^p}\bfSh(G'_{\tilde\tS}, K'^p)$ is an integral canonical model of $\Sh(G'_{\tilde\tS})$ over $\cO_{\tilde\wp}$ in the sense that $\bfSh(G'_{\tilde\tS})$ satisfies the following extension property over $\cO_{\tilde\wp}$: if $S$ is a smooth scheme over $S$, any morphism $S\otimes_{\cO_{\tilde\wp}}E_{\tilde\tS,\tilde\wp}\to \bfSh(G'_{\tilde\tS})$ extends uniquely to a morphism $S\to \bfSh(G'_{\tilde\tS})$.
\end{theorem}

Let $\ZZ_p^{\ur}$ be the ring of integers of $\QQ_p^{\ur}$. The closure of $\Sh(G_{\tilde\tS}')^{\circ}$ in $\bfSh(G'_{\tilde\tS})\otimes_{\cO_{\tilde\wp}}{\ZZ^\ur_p}$, denote by $\bfSh(G'_{\tilde\tS})^{\circ}_{\ZZ_p^{\ur}}$, is a smooth integral canonical model of $\Sh(G'_{\tilde\tS})^{\circ}$ over $\ZZ_{p}^\ur$. By \eqref{E:isom-conn-Sh}, this can also be regarded as an integral canonical model of $\Sh(G_{\tS,\tT})^{\circ}$ over $\ZZ_p^{\ur}$. This induces a smooth integral canonical model $\bfSh(G_{\tS,\tT})$ of $\Sh(G_{\tS,\tT})$ over $\cO_{F_{\tS,\tT},\wp}$ by Deligne's recipe (See \cite{TX1}*{Corollary~2.17}). For any open compact subgroup $K^p\subseteq G_{\tS}(\bA^{\infty,p})$, we define $\bfSh(G_{\tS,\tT}, K^p)$ as the quotient of $\bfSh(G_{\tS,\tT})$ by $K^p$. Then if $K^p$ is sufficiently small, $\bfSh(G_{\tS,\tT}, K^p)$ is a quasi-projective smooth scheme over $\cO_{F_{\tS,\tT},\wp}$, and it is an integral model for $\Sh(G_{\tS,\tT},K^p)$.

\subsection{Moduli interpretation for totally indefinite quaternionic Shimura varieties}
\label{S:quaternion-PEL}

When $\tS=\emptyset$, then $\tT=\emptyset$ and the Shimura variety $\Sh(G,K^p)$ has another moduli interpretation in terms of abelian varieties with real multiplication by $\cO_{B}$. Using this moduli interpretation, one can also construct another integral model of $\Sh(G, K^p)$. The aim of this part is to compare this integral canonical model of $\Sh(G, K^p)$ with $\bfSh(G, K^p)$ constructed in the previous subsection using unitary Shimura varieties.

We choose an element $\gamma\in B^{\times}$ such that
\begin{itemize}
  \item $\bar \gamma =-\gamma$;

  \item $b\mapsto b^*\coloneqq\gamma^{-1}\bar{b}\gamma$ is a positive involution;

  \item $\nu(\gamma)$ is a $\fp$-adic unit for every $p$-adic place $\fp$ of $F$, where $\nu\colon B^{\times}\to F^{\times}$ is the reduced norm map.
\end{itemize}
Put $V\coloneqq B$ viewed as a free left $B$-module of rank $1$, and consider the alternating pairing
\[
\langle\_,\_ \rangle_{F}\colon  V\times  V\to F, \quad \langle x,y\rangle_F=\Tr^{\circ}_{B/F}(x\bar y\gamma),
\]
where $\Tr^{\circ}_{B/F}$ is the reduced trace of $B$. Note that $\langle bx,y\rangle_F=\langle x,b^*y\rangle_F$ for $x,y\in V$ and $b\in B$. We let $G=B^{\times}$ act on $V$ via $g\cdot v= v g^{-1}$ for $g\in G$ and $v\in V$. One has an isomorphism
\[
G\cong \Aut_{B}(V).
\]
Fix an order $\cO_B\subseteq B$ such that
\begin{itemize}
  \item $\cO_B$ contains $\cO_F$, and it is stable under $*$;

  \item $\cO_B\otimes\ZZ_p$ is a maximal order of $B\otimes_{\QQ}\QQ_p\cong \GL_2(F\otimes_{\QQ}\QQ_p)$.

\end{itemize}

Let $K^p\subseteq G(\bA^{\infty,p})$ be an open compact subgroup. 
Consider the moduli problem $\underline{\bfSh}(G,{K^{p}})$ that associates to every $\ZZ_{(p)}$-scheme $T$ the equivalence classes of tuples $(A,\iota,\bar\lambda,\bar\alpha_{K^p})$ where
\begin{itemize}
  \item $A$ is a projective  abelian scheme over $T$ up to prime-to-$p$ isogenies;

  \item $\iota$ is a  real multiplication by $\cO_B$ on $A$, that is, a ring homomorphism $\iota\colon \cO_B\to \End(A)$ such that 
      \[
      \det(T-\iota(b)|\Lie(A))=\rN_{F/\QQ}(\rN^{\circ}_{B/F}(T-b)), \quad b\in \cO_B,
      \]
      where $\rN^{\circ}_{B/F}$ is the reduced norm of $B/F$;

  \item $\bar \lambda$ is an ${F}_+^{p,\times}$-orbit of $\cO_F$-linear prime-to-$p$ polarizations $\lambda\colon A\to A^{\vee}$ such that $\iota(b)^{\vee}\circ\lambda=\lambda\circ\iota(b^*)$ for all $b\in\cO_B$, where $F_{+}^{p,\times}\subseteq F^{\times}$ is the subgroup of totally positive elements that are $\fp$-adic units for all $\fp\in \Sigma_p$;

  \item $\bar\alpha_{K^p}$ is a $K^p$-level structure on $(A,\iota)$, that is, $\bar\alpha_{K^p}$ is a $K^p$-orbit of $B\otimes_\bQ\bA^{\infty,p}$-linear isomorphisms of \'etale sheaves on $T$:
      \[
      \alpha\colon V\otimes_{\QQ}\bA^{\infty,p}\xra{\sim }\widehat{V}^{p}(A).
      \]
 \end{itemize}

 \begin{remark}\label{R:polarization-free}
 By \cite{Zin82}*{Lemma 3.8}, there exists exactly one $F^{p,\times}_+$ orbit of prime-to-$p$ polarizations on $A$ that induces the given positive involution $*$ on $B$. Hence, one may omit $\bar \lambda$ from the definition of the moduli problem $\underline{\bfSh}(G,K^p)$. This is the point of view in \cite{Liu2}. Here, we choose to keep $\bar\lambda$ in order to compare it with unitary Shimura varieties.
 \end{remark}

By \cite{Zin82}*{page~27}, one has a bijection
\[
\underline{\bfSh}(G,K^p)(\CC)\cong G(\QQ)\backslash (\fH^{\pm})^{\Sigma_\infty}\times G(\bA^{\infty})/K^pK_p=  \Sh(G,K^p)(\CC).
\]
Note that an object  $(A,\iota,\bar\lambda,\bar\alpha_{K^p})\in\underline{\bfSh}(G,K^p)(T)$ admits automorphisms $\cO_{F}^{\times}\cap K^p$, which is always non-trivial if $F\neq \QQ$. Here, $\cO_{F}^{\times}$ is considered as a subgroup of $G(\bA^{\infty,p})$ via the diagonal embedding. Thus, the moduli problem $\underline {\bfSh}(G,K^p)$ can not be representable. However, Zink shows in \cite{Zin82}*{Satz~1.7} that $\underline{\bfSh}(G,K^p)$ admits a coarse moduli space $\bfSh(G,K^p)$, which is a projective scheme  over $\ZZ_{(p)}$. This gives an integral model of the Shimura variety $\Sh(G,K^p)$ over $\ZZ_{(p)}$.

We recall briefly Zink's construction of $\bfSh(G,K^p)$. Take $(A,\iota,\bar\lambda,\bar\alpha_{K^p})\in\underline{\bfSh}(G,K^p)(T)$ for some $\ZZ_{(p)}$-scheme $T$. Choose a polarization $\lambda\in \bar\lambda$, and an isomorphism $\alpha\in \bar\alpha_{K^p}$. Then $\lambda$ induces a Weil pairing
\[
\widehat\Psi^{\lambda}\colon \widehat V^{p}(A)\times \widehat V^p(A)\to \bA^{\infty,p}(1),
\]
and there exists a unique $F$-linear alternating pairing
\[
\widehat\Psi^{\lambda}_{F}\colon\widehat{V}^p(A)\times\widehat{V}^p(A)\to \bA_F^{\infty,p}(1)
\]
such that $\widehat\Psi^{\lambda}=\Tr_{F/\QQ}\circ \widehat\Psi^{\lambda}_F$. We fix an isomorphism $\ZZ\cong \ZZ(1)$, and view $\ang{\_,\_}$ as a pairing with values in $F(1)$. Then by \cite{Zin82}*{1.2}, there exists an element $c\in \bA_F^{\infty,p,\times}$ such that
\[
\widehat\Psi^\lambda_{F}(\alpha(x),\alpha(y))=c\langle x, y\rangle_{F},\quad x,y\in V\otimes_{\QQ}\bA^{\infty,p}.
\]
The class of $c$ in $\bA^{\infty,p,\times}_F/\nu(K^p)$, denoted by $c(A,\iota,\lambda,\bar\alpha_{K^p})$, is independent of the choice of $\alpha\in\bar\alpha_{K^p}$.
If $F^{\times}_+\subseteq F^{\times}$ is the subgroup of totally positive elements, then the image of $c(A,\iota,\lambda, \bar\alpha_{K^p})$ in
\[
\bA^{\infty,p,\times }_F /F^{p,\times }_+\nu(K^p)\cong \bA^{\infty,\times}_F/F^{\times}_+\nu(K)
\]
is independent of the choices of both $\lambda$ and $\alpha$.

We choose representatives $c_1,\dots, c_r\in \bA^{\infty,p,\times }_F /\nu(K^p)$ of the finite quotient $\bA^{\infty,p,\times}_F/F^{p,\times}_+\nu(K^p)$, and consider the moduli problem  $\underline{\widetilde\bfSh}(G,K^p)$ that associates to every $\ZZ_{p}$-scheme $T$ equivalence classes of tuple $(A,\iota,\lambda,\bar\alpha_{K^p})$, where
\begin{itemize}
  \item $(A,\iota)$ is an abelian scheme over $T$ up to prime-to-$p$ isogenies equipped with  real multiplication by $\cO_B$;

  \item $\lambda\colon  A\to A^{\vee}$ is a prime-to-$p$ polarization such that $\iota(b)^{\vee}\circ\lambda=\lambda\circ\iota(b^*)$ for all $b\in\cO_B$;

  \item $\bar \alpha_{K^p}$ is a $K^p$-level structure on $A$ such that $c(A,\iota,\lambda, \bar\alpha_{K^p})=c_i$ for some $i=1,\dots,r$.
\end{itemize}

To study the representability of $\underline {\widetilde \bfSh}(G,K^p)$, we need the following notion of neat subgroups.

\begin{definition}\label{D:neat-subgroup}
Let $\tR$ be the ramification set of $B$. For every $g_v\in (B\otimes_{F}F_v)^{\times}$ with $v\notin \tR$, let $\Gamma_{g_{v}}$ denote the subgroup of $F_v^{\ac,\times}$ generated by the eigenvalues of $g_v$. Choose an embedding $\QQ^{\ac}\hookrightarrow F_{v}^{\ac}$. Then $(\Gamma_{g_v}\cap \QQ^{\ac})^{\tor}$ is the subgroup of $\Gamma_{g_v}$ consisting of roots of unity, and it is independent of the embedding $\QQ^{\ac}\hookrightarrow F_{v}^{\ac}$.

Let $\Box$ be a finite set of places of $\bQ$ containing the archimedean place,  and let $\Box_F$ be  the set of places of $F$ above $\Box$. An element $g\in G(\bA^\Box)=(B\otimes_\bQ\bA^\Box)^\times$ is called \emph{neat} if $\bigcap_{v\in\Box_F-\tR}(\Gamma_{g_v}\cap\QQ^{\ac})^{\tor}=\{1\}$. We say  a subgroup $U\subseteq G(\bA^\Box)$ is \emph{neat} if every element $g=g^{\tR}g_{\tR}\in U$ with $\nu(g^{\tR})=1$ is neat. Here, $g^{\tR}\in (B\otimes_F \bA_{F}^{\Box_F\cup\tR})^\times$ (resp. $g_{\tR}\in \prod_{v\in\tR-\Box_F}(B\otimes_FF_v)^{\times}$) is the prime-to-$\tR$ component (resp. $\tR$-component) of $g$.
\end{definition}

Assume from now on that $K^p\subseteq G(\bA^{\infty,p})$ is neat. It is easy to see that each object of $\underline{\widetilde\bfSh}(G,K^p)$ has no non-trivial automorphisms. By a well-known result of Mumford, $\underline{\widetilde\bfSh}(G,K^p)$ is representable by a quasi-projective smooth scheme $\widetilde\bfSh(G,K^p)$ over $\ZZ_{(p)}$. If $B$ is a division algebra, then $\widetilde\bfSh(G,K^p)$ is even projective over $\ZZ_{(p)}$ (see \cite{Zin82}*{Lemma~1.8}).

Let $\cO_{F,+}^{\times}$ be the group of totally positive units of $F$. There is a natural action by $\cO_{F,+}^{\times}\cap \nu(K^p)$ on ${\widetilde{\bfSh}}(G,K^p)$ given by $\xi\cdot(A,\iota, \lambda, \bar\alpha_{K^p})=(A,\iota,\xi\cdot \lambda, \bar\alpha_{K^p})$ for $\xi\in\cO_{F,+}^{\times}$, and the quotient is the moduli problem $\underline{{\bfSh}}(G,K^p)$.
Note that the subgroup $(\cO_{F}^{\times}\cap K^p)^2$ acts trivially on ${\widetilde{\bfSh}}(G,K^p)$. Here, $\cO_{F}^{\times}$ is considered as  a subgroup in the center of $G(\bA^{\infty, p})$. Indeed, if $\xi=\eta^2$ with $\eta\in \cO_{F}^{\times }\cap K^p$, then the multiplication by $\eta$ on $A$ defines an isomorphism $(A,\iota, \lambda,\bar\alpha_{K^p})\xra{\sim}(A,\iota,\xi\cdot\lambda,\bar\alpha_{K^p})$. Put
\[
\Delta_{K^p}\coloneqq (\cO_{F,+}^{\times}\cap \nu(K^p))/(\cO_{F}^{\times}\cap K^p)^2.
\]


\begin{lem}
Assume that $K^p$ is neat. Let $(A,\iota,\bar\lambda,\bar\alpha_{K^p})$ be a $T$-valued point of $\underline{\bfSh}(G,K^p)$. Then the group of automorphisms of $(A,\iota, \bar\lambda, \bar\alpha_{K^p})$ is $\cO_F^{\times}\cap K^p$. Here, $\cO_{F}^{\times}$ is viewed as a subgroup of $G(\bA^{\infty,p})$ via the diagonal embedding.
\end{lem}

\begin{proof}
This is a slight generalization of \cite{Zin82}*{Korollar~3.3}. Take $\eta\in\End_{\cO_B}(A)_{\QQ}$ that preserves $\bar\lambda$ and $\bar\alpha_{K^p}$. Then there exists $\xi\in {F_+}^{\times}$ such that $\eta\hat\eta=\xi$, where $\hat{\eta}$ is the Rosati involution of $\eta$ induced by $\bar\lambda$. By \cite{Zin82}*{Satz~3.2}, it is enough to show that $\hat\eta=\eta$. Choose $\alpha\in \bar\alpha_{K^p}$, which induces an embedding
\[
(\End_{\cO_B}(A)\otimes{\QQ})^\times\to(\End_B(V)\otimes_{\QQ}\bA^{\infty, p})^\times\cong G(\bA^{\infty,p}).
\]
Then the image of $\eta$ under this embedding lies in $K^p$. Consider the endomorphism $\eta^2\xi^{-1}\in\End_{\cO_B}(A)\otimes{\QQ}$. Its image in $G(\bA^{\infty,p})$ lies in $K^p$ and has reduced norm equal to $1$. Since $K^p$ is neat,  all the eigenvalues of $\eta^2\xi^{-1}$ are $1$. So $\eta^2\xi^{-1}$ must be trivial, and hence $\eta=\hat{\eta}$.
\end{proof}

\begin{corollary}
Assume that $K^p$ is neat. Then the action of $\Delta_{K^p}$ on ${\widetilde{\bfSh}}(G,K^p)$ is free.
\end{corollary}

\begin{proof}
The same argument as \cite{Zin82}*{Korollar~3.4} works.
\end{proof}

We put
\begin{equation}\label{E:moduli-quaternion}
\bfSh(G,K^p)\coloneqq \widetilde {\bfSh}(G,K^p)/\Delta_{K^p},
\end{equation}
which exists as a quasi-projective smooth over $\ZZ_{(p)}$ by \cite{SGA1}*{Expos\'e~VIII,~Corollaire~7.7}. Then $\bfSh(G,K^p)$ is the coarse moduli space of the moduli problem $\underline { \bfSh}(G,K^p)$, and $\widetilde \bfSh(G,K^p)$ is a finite \'etale cover of $\bfSh(G,K^p)$ with Galois group $\Delta_{K^p}$. For each $i=1,\dots, r$, we denote by $\widetilde{\bfSh}^{c_i}(G,K^p)$ the subscheme of $\widetilde{\bfSh}(G,K^p)$ consisting the tuples $(A,\iota, \lambda,\bar\alpha_{K^p})$ with $c(A,\iota, \lambda, \bar\alpha_{K^p})=c_i$. It is clear that each $\widetilde{\bfSh}^{c_i}(G,K^p)$ is stable under the action of $\Delta_{K^p}$. Let $\bfSh^{c_i}(G,K^p) \subseteq \bfSh(G,K^p)$ be the image of $\widetilde{\bfSh}^{c_i}(G,K^p)$ under the morphism \eqref{E:moduli-quaternion}. Note that each $\bfSh^{c_i}(G,K^p) $ is not necessarily defined over $\ZZ_{(p)}$. Actually, using the strong approximation theorem, one sees easily that $\bfSh^{c_i}(G,K^p)(\CC)$ is a connected component of $\bfSh(G,K^p)(\CC)$.

\begin{remark}\label{R:universal-AV}
Assume that $K^p$ is neat.
\begin{enumerate}
  \item Let $(\widetilde\cA,\widetilde\iota)$ be the universal abelian scheme with real multiplication by $\cO_B$ over $\widetilde{\bfSh}(G,K^p)$. Then $\widetilde\cA$ is equipped with a natural descent data relative to the projection $\widetilde{\bfSh}(G,K^p)\to \bfSh(G,K^p)$, since the action of $\Delta_{K^p}$ modifies only the polarization. By \cite{SGA1}*{Expos\'e~VIII,~Corollaire~7.7}, the descent data on  $\widetilde\cA$ is effective. This means that, even though $\bfSh(G,K^p)$ is not a fine moduli space, there exists still a universal family $\cA$ over $\bfSh(G,K^p)$. Moreover, by \'{e}tale descent, $\widetilde\iota$ descends to a real multiplication $\iota$ by $\cO_B$ on the universal family $\cA$ over $\bfSh(G,K^p)$.

  \item In general, $\Delta_{K^p}$ is non-trivial. However,  for any open compact subgroup $K^p\subseteq G(\bA^{\infty,p})$, there exists a smaller open compact subgroup $K'^p\subseteq K^p$ such that $\Delta_{K'^p}$ is trivial.
\end{enumerate}
\end{remark}

We give an interpretation of $\widetilde{\bfSh}(G,K^p)$ in terms of Shimura varieties. Let $G^{\star}\subseteq G$ be the preimage of $\dG_{m,\QQ}\subseteq T_{F}=\Res_{F/\QQ}(\dG_{m,F})$ via the reduced norm map $\nu\colon G\to T_{F}$. The Deligne homomorphism $h_{\emptyset}\colon\dS(\RR)=\CC^{\times}\to G(\RR)$ factors through $G^{\star}(\RR)$, hence induces a map
\[
h_{G^{\star}}\colon \dS(\RR)\to G^{\star}(\RR).
\]
We put $K_{p}^{\star}\coloneqq G^{\star}(\QQ_p)\cap K_p$, which will be the fixed level at $p$ for Shimura varieties attached to $G^{\star}$. For a sufficiently small open compact subgroup $K^{\star p}\subseteq G^{\star}(\bA^{\infty,p})$, we have the associated Shimura variety $\Sh(G^{\star}, K^{\star p})$ defined over $\QQ$, whose $\CC$-points are given by
\[
\Sh(G^{\star}, K^{\star p})(\CC)=G^{\star}(\QQ)\backslash \big((\fH^\pm)^{\Sigma_\infty}\times G^{\star}(\bA^{\infty})/K^{\star p}K^\star_p\big).
\]
Put $\Sh(G^{\star})\coloneqq\varprojlim_{K^{\star p}}\Sh (G^{\star},K^{\star p })$ as usual.

There is a natural action of $\bA^{\infty, p,\times}$ on $\bA_{F}^{\infty,p,\times}/F^{p,\times}_{+}\nu(K^p)$ by multiplication. 
Let $\fc_1,\dots, \fc_h$ denote the equivalence classes modulo $F^{p,\times}_+\bA^{\infty, p,\times}$ of the chosen set  $\{c_1,\dots,c_r\}\subseteq \bA_{F}^{\infty,p,\times}/\nu(K^p)$. We may and do assume that all the $c_i$'s in one equivalence class differ from each other by elements in $\bA^{\infty,p,\times}$. For each $\fc\in \{\fc_1,\dots, \fc_h\}$, we put
\[
\widetilde{\bfSh}^{\fc}(G,K^p)\coloneqq \coprod_{c_i\in \fc} \widetilde{\bfSh}^{c_i}(G,K^p)
\]
and similarly $\bfSh^{\fc}(G,K^p)=\coprod_{c_i\in\fc}\bfSh^{c_i}(G,K^p)$.

\begin{proposition}\label{P:coarse-fine-moduli}
Suppose that $K^p\subseteq G(\bA^{\infty,p})$ is a neat open compact subgroup. For every $\fc\in \{\fc_1,\dots,\fc_h\}$, there exists an element $g^p\in G(\bA^{\infty,p})$ such that if $K^{\star, p}_{\fc}\coloneqq G^{\star}\cap g^{p}K^pg^{p,-1}$, then we have an isomorphism of algebraic varieties over $\QQ$
\[
\widetilde{\bfSh}^{\fc}(G, K^p)\otimes_{\ZZ_{(p)}}\QQ\xra{\sim} \Sh(G^{\star}, K^{\star,p}_{\fc}).
\]
\end{proposition}

\begin{proof}
Let $X\cong (\fH^{\pm})^{\Sigma_\infty}$ denote the set of conjugacy classes of $h_{G^{\star}}\colon\dS(\RR)\to G^{\star}(\RR)$. 
We fix a base point $(A_0,\iota_0,\lambda_0, \bar\alpha_{K^p,0})\in \widetilde{\bfSh}^{\fc}(G,K^p)(\CC)$. Put $V_{\QQ}(A_0)\coloneqq\rH_1(A_0(\bC),\QQ)$. We fix an isomorphism  $\eta_0\colon V_{\QQ}(A_0)\xra{\sim}V$ of left $B$-modules and a choice of $\alpha_0\in \bar\alpha_{K^p}$. Then the composite map
\[
(\eta_0\otimes1)\circ \alpha_0\colon V\otimes_{\QQ}\bA^{\infty,p}\to\widehat{V}^p(A_0)\cong V_{\QQ}(A_0)\otimes_{\QQ}\bA^{\infty,p}\to V\otimes_{\QQ}\bA^{\infty,p}
\]
defines an element $g^p\in G(\bA^{\infty,p})$.
Now let $(A,\iota,\lambda,\bar\alpha_{K^p})\in \widetilde{\bfSh}^{\fc_i}(G,K^p)(\CC)$ be another point. There exists also an isomorphism $\eta\colon V_\bQ(A)\xra{\sim}V$ as $B$-modules, and the Hodge structure on $V_\bQ(A)\otimes_{\QQ}\RR=\rH_1(A(\bC),\RR)$ defines an element $x_{\infty}\in X$.
By the definition of $\bfSh^{\fc}(G,K^p)$, there exists an element $\alpha\in\bar\alpha_{K^p}$ such that the isomorphism
\[
h^p\coloneqq (\eta\otimes 1)\circ\alpha\circ\alpha_{0}^{-1}(\eta_0\otimes 1)^{-1}\in G(\bA^{\infty,p})
\]
preserves the alternating pairing $\ang{\_,\_}_F$ on $V\otimes_{\QQ}\bA^{\infty,p}$ up to a scalar in $\bA^{\infty,p,\times}$. Such an element $\alpha$ is unique up to right multiplication by elements in $K^p$, and it follows that $h^p$ is well defined up to right multiplication by elements of $K^{\star,p}_{\fc}\coloneqq g^{p}K^pg^{p,-1}\cap G^{\star}(\bA^{\infty,p})$. Viewing $h^p$ as an element of $G^{\star}(\bA^{\infty})$ with $p$-component equal to $1$, then $(A,\iota, \lambda, \bar\alpha_{K^p})\mapsto [x_{\infty}, h^p ]$ defines a map
\[
f\colon\widetilde{\bfSh}^{\fc}(G, K^p)(\CC)\to\Sh(G^{\star},K^{\star,p})(\CC)\cong G^{\star}(\QQ)\backslash\big(X\times G^{\star}(\bA^{\infty})/K^{\star, p}_{\fc}K^{\star}_p\big).
\]
By the complex uniformization of abelian varieties, it is easy to see that $f$ is bijective, and $f$ descends to an isomorphism of algebraic varieties over $\QQ$ by the theory of canonical models.
\end{proof}

\begin{remark}\label{R:int-model-G-star}
In general, there is no canonical choice for $g^p$ in the above proposition. Different choices of $g^p$ will result in different $K^{\star,p}_{\fc}$, which are conjugate to each other in $G^{\star}(\bA^{\infty,p})$. Consequently, the corresponding $\bfSh(G^{\star},K^{\star,p}_{\fc})$ are isomorphic to each other by the Hecke action of some  elements in $G^{\star}(\bA^{\infty,p})$. However, if $\fc=\fc^{\mathrm{tri}}$ is the trivial equivalence class, $g^p$ has a canonical choice, namely $g^p=1$. In the sequel, we will always take $g^p=1$ if $\fc=\fc^{\tri}$. Applying Proposition~\ref{P:coarse-fine-moduli} to this case, one obtains a moduli interpretation of $\Sh(G^{\star},K^{\star,p})$ as well as an integral model $\bfSh(G^{\star},K^{\star, p})$ over $\ZZ_{(p)}$ of $\Sh(G^{\star},K^{\star, p})$. Explicitly, the integral model $\bfSh(G^{\star}, K^{\star,p})$ parameterizes equivalence classes of tuples $(A,\iota,\lambda,\bar\alpha_{K^{\star,p}})$, where $(A,\iota,\lambda)$ is the same data as in $\widetilde{\bfSh}(G,K^p)$, and $\alpha_{K^{\star,p}}$ is a $K^{\star,p}$-level structure on $A$, that is, an $K^{\star, p}$-orbit of isomorphisms $\alpha\colon V\otimes\bA^{\infty,p}\xra{\sim}\widehat{V}^p(A)$ such that $\ang{\_,\_}_F$ is compatible with $\widehat{\Psi}^{\lambda}_F$ up to a scalar in $\bA^{\infty,p,\times}$.
\end{remark}

\begin{example}\label{ex:level}
Fix a lattice $\Lambda\subseteq V$ stable under $\cO_B$ such that $\langle\Lambda,\Lambda\rangle_F\subseteq\fd_F^{-1}$, where $\fd_F$ is the different of $F/\QQ$, and that $\Lambda\otimes\ZZ_p$ is self-dual under $\langle\_,\_\rangle_F$.

Let $\fM,\fN$ be two ideals of $\cO_F$ such that they are mutually coprime, both prime to $p$ and the ramification set $\tR$ of $B$, and that $\fN$ is contained in $N\cO_F$ for some integer $N\geq 4$. Let $K_{0,1}(\fM,\fN)^p$ be a subgroup of $\gamma\in G(\bA^{\infty,p})$ such that there exists $v\in \Lambda$ with $\gamma v\in(\cO_Fv+\fM\Lambda)\cap(v+\fN\Lambda)$; put $K_{0,1}(\fM,\fN)\coloneqq K_{0,1}(\fM,\fN)^pK_p$. Then $K_{0,1}(\fM,\fN)^p$ is neat and  $\nu(K_{0,1}(\fM,\fN))=\widehat{\cO}_F^{\times}$. We have thus isomorphisms
\[
\bA^{\infty,p,\times }_F /F^{p,\times}_{+}\nu(K_{0,1}(\fM,\fN)^p)
\cong \bA^{\infty,\times}_F/F^{\times}_+ \widehat{\cO}_F^{\times}\cong \Cl^+(F),
\]
where $\Cl^+(F)$ is the strict ideal class group of $F$; and the action of $\bA^{\infty,\times}$ on $\Cl^{+}(F)$ is trivial. We choose prime-to-$p$ fractional ideals  $\fc_1,\dots,\fc_h$ that form a set of representatives  of $\Cl^+(F)$. Then for each $\fc\in\{\fc_1,\dots,\fc_h\}$, the moduli scheme $\widetilde{\bfSh}^{\fc}(G,K_{0,1}(\fM,\fN)^p)$ classifies tuples $(A,\iota,\lambda,C_\fM,\alpha_\fN)$, where
\begin{itemize}
  \item $(A,\iota)$ is a projective abelian scheme equipped with real multiplication by $\cO_B$;

  \item $\lambda\colon A\to A^{\vee}$ is an $\cO_F$-linear polarization such that $\iota(b)^\vee\circ\lambda=\lambda\circ\iota(b^*)$ for $b\in\cO_B$, and the induced map of abelian fppf-sheaves
      \[
      A^{\vee}\xra{\sim}A\otimes_{\cO_F}\fc
      \]
      is an isomorphism;

  \item $C_\fM$ is a finite flat subgroup scheme of $A[\fM]$ that is $\cO_B$-cyclic of order $(\Nm\fM)^2$;

  \item $\alpha_\fN\colon (\cO_F/\fN)^{\oplus 2}\hookrightarrow A[\fN]$ is an embedding of finite \'etale group schemes equivariant under the action of $\cO_{B}\otimes_{\cO_F}\cO_F/\fN\cong \GL_2(\cO_F/\fN)$.
\end{itemize}

Let $g^p_{\fc}\in G(\bA^{\infty,p})$ be such that the fractional ideal attached to the id\`ele $\nu(g^p_{\fc})\in \bA^{\infty,p,\times}_F$ represents the strict ideal class  $\fc$. Put
\[
K^{\star,p}_{\fc_i}\coloneqq g_{\fc}^p K_{0,1}(\fM,\fN)^pg_{\fc}^{p,-1}\cap G^{\star}(\bA^{\infty,p}).
\]
Then we have
\[
\widetilde{\bfSh}^{\fc}(G,K_{0,1}(\fM,\fN)^p)\otimes\QQ\cong \Sh(G^{\star}, K^{\star,p}_{\fc_i}).
\]
More explicitly, if $\Gamma^{\fc}_{0,1}(\fM,\fN)\coloneqq G^{\star}(\QQ)_+\cap K^{\star,p}_{\fc}$, where $G^{\star}(\QQ)_+\subseteq G^{\star}(\QQ)$ is the subgroup of elements with totally positive reduced norms, then
\[
\widetilde{\bfSh}^{\fc}(G,K_{0,1}(\fM,\fN)^p)(\CC)\cong \Sh(G^{\star}, K_{\fc}^{\star, p})(\CC)\cong \Gamma^{\fc}_{0,1}(\fM,\fN)\backslash (\fH^+)^{\Sigma_\infty}.
\]
In particular, $\widetilde \bfSh^{\fc}(G,K_{0,1}(\fM,\fN)^p)\otimes{\QQ}$ is geometrically connected for every $\fc$. In this case, one has $\Delta_{K_{0,1}(\fM,\fN)^p}=\cO_{F,+}^{\times}/\cO_{F,\fN}^{\times, 2}$, where $\cO_{F,\fN}^{\times}$ denotes the subgroup of $\xi\in\cO_F^{\times}$ with $\xi\equiv 1\mod \fN$. It is clear that the action of $\Delta_{K_{0,1}(\fM,\fN)^p}$ preserves $\widetilde\bfSh^{\fc}(G,K_{0,1}(\fM,\fN)^p)$, and one obtains an isomorphism
\[
\bfSh(G,K_{0,1}(\fM,\fN)^p)\cong \coprod_{i=1}^{h} \bfSh^{\fc_i}(G,K_{0,1}(\fM,\fN)^p)
\]
with $\bfSh^{\fc_i}(G,K_{0,1}(\fM,\fN)^p)=\widetilde{\bfSh}^{\fc_i}(G,K_{0,1}(\fM,\fN)^p)/\Delta_{K_{0,1}(\fM,\fN)^p}$. Since  $\Delta_{K_{0,1}(\fM,\fN)^p}$ acts freely on $\widetilde{\bfSh}(G,K_{0,1}(\fM,\fN)^p)$, each $\bfSh^{\fc_i}(G,K_{0,1}(\fM,\fN)^p)$ is a smooth quasi-projective scheme over $\ZZ_{(p)}$.
\end{example}

\if false

\begin{example}
Fix a lattice $\Lambda\subseteq V$ stable under $\cO_B$ such that $\langle\Lambda,\Lambda\rangle_F\subseteq \fd_F^{-1}$, where $\fd_F$ is the different of $F/\QQ$, and that $\Lambda\otimes\ZZ_p$ is self-dual under $\langle\_,\_\rangle_F$. Let $N\geq 4$ be an integer prime to $p$ and  the ramification set $\tR$ of $B$. Let $K_1(N)$ be the subgroup of $\gamma\in G(\bA^{\infty})$ such that there exists $v\in\Lambda$ with $\gamma v\equiv v\mod N\Lambda$. Then $K_1(N)$ is neat and $\nu(K_1(N))=\widehat{\cO}_F^{\times}$. We have thus an isomorphism
\[
\bA^{\infty,p,\times }_F /F^{p,\times}_{+}\nu(K_1(N)^p)\cong \bA^{\infty,\times}_F/F^{\times}_+\widehat{\cO}_F^{\times}\cong \Cl^+(F)
\]
where $\Cl^+(F)$ is the strict ideal class group  of $F$, and the action of $\bA^{\infty,\times}$ on $\Cl^{+}(F)$ is trivial. We choose prime-to-$p$ fractional ideals  $\fc_1,\dots,\fc_h$ that form a set of representatives  of $\Cl^+(F)$. Then for each $\fc\in\{\fc_1,\dots,\fc_h\}$, the moduli scheme   $\widetilde{\bfSh}^{\fc}(G,K_1(N)^p)$ classifies tuples $(A,\iota, \lambda, \alpha_{N})$, where
\begin{enumerate}
  \item $(A,\iota)$ is a projective abelian scheme equipped with real multiplication by $\cO_B$,

  \item $\lambda\colon A\to A^{\vee}$ is an $\cO_F$-linear polarization such that $\iota(b)^\vee\circ \lambda=\lambda\circ\iota(b^*)$ for $b\in\cO_B$, and the induced map of abelian fppf-sheaves
      \[A^{\vee}\xra{\sim}A\otimes_{\cO_F}\fc\]
      is an isomorphism,

  \item $\alpha_N\colon (\mu_N\otimes \fd_F^{-1})^{\oplus 2}\hookrightarrow A[N]$ is an embedding of finite \'etale group schemes equivariant under the action of $\cO_{B}\otimes \ZZ/N\ZZ\cong \GL_2(\cO_F/N\cO_F)$.
\end{enumerate}

Let $g^p_{\fc}\in G(\bA^{\infty,p})$ be such that the fractional ideal attached to the id\`ele $\nu(g^p_{\fc})\in \bA^{\infty,p,\times}_F$ represents the strict ideal class $\fc$. Put
\[
K^{\star,p}_{\fc_i}\coloneqq g_{\fc}^p K_1(N)^pg_{\fc}^{p,-1}\cap G^{\star}(\bA^{\infty,p}).
\]
Then we have
\[
\widetilde{\bfSh}^{\fc}(G,K_1(N)^p)\times_{\ZZ}\QQ\cong \Sh(G^{\star}, K^{\star,p}_{\fc_i}).
\]
More explicitly, if $\Gamma^{\fc}_1(N)\coloneqq G^{\star}(\QQ)_+\cap K^{\star,p}_{\fc}$, where
$G^{\star}(\QQ)_+\subseteq G^{\star}(\QQ)$ is the subgroup of elements with totally positive reduced norms,  then
\[
\widetilde{\bfSh}^{\fc}(G,K_1(N)^p)(\CC)\cong \Sh(G^{\star}, K_{\fc}^{\star, p})(\CC)\cong \Gamma^{\fc}_1(N)\backslash \fH^{[F:\QQ]}.
\]
In particular, each  $\widetilde \bfSh^{\fc}(G,K^p)\otimes_{\ZZ}{\QQ}$  is geometrically connected.  In this case,  one has $\Delta_{K_1(N)^p}=\cO_{F,+}^{\times}/\cO_{F,N}^{\times, 2}$, where $\cO_{F,N}^{\times}$ denotes the subgroup of $\xi\in \cO_F^{\times}$ with $\xi\equiv 1\mod N\cO_F$. It is clear that the action of $\Delta_{K_1(N)^p}$  preserves $\widetilde\bfSh^{\fc}(G,K^p)$, and one obtains
\[
\bfSh(G,K^p)\cong \coprod_{i=1}^{h} \bfSh^{\fc_i}(G,K^p), \quad \text{with }\bfSh^{\fc_i}(G,K^p)=\widetilde{\bfSh}^{\fc_i}(G,K^p)/\Delta_{K_1(N)^p}.
\]
 Since  $\Delta_{K_1(N)^p}$ acts freely on $\widetilde {\bfSh}(G,K^p)$, each $\bfSh^{\fc_i}(G,K^p)$ is a smooth quasi-projective scheme over $\ZZ_{(p)}$.
\end{example}

\fi

\subsection{Comparison of quaternionic and unitary moduli problems}
\label{S:comparison-moduli}

We now compare the integral model $\bfSh(G,K^p)$ defined in \eqref{E:moduli-quaternion} and the one constructed using the unitary Shimura variety $\bfSh(G'_{\tilde\tS}, K'^p)$ with $\tS=\emptyset$. Note that when $\tS=\emptyset$, there is only one choice for $\tilde\tS$, so we write simply $G'$ for $G'_{\tilde\tS}$. By the universal extension property of $\bfSh(G)\coloneqq \varprojlim_{K^p}\bfSh(G,K^p)$, these two integral canonical models are necessarily isomorphic. However, for later applications to the supersingular locus of $\bfSh(G, K^p)_{\FF_p}$, one needs a more explicit comparison between the universal family of abelian varieties over $\bfSh(G)$ (as in Remark \ref{R:universal-AV}(1)) with that over $\bfSh(G')$. It suffices to compare the universal objects over the the neutral connected components via the isomorphism
\[
\bfSh(G)_{\ZZ_p^{\ur}}^{\circ} \xra{\sim }\bfSh(G')^{\circ}_{\ZZ_p^{\ur}}
\]
induced by \eqref{E:isom-conn-Sh}. Here, $\bfSh(G)_{\ZZ_p^{\ur}}^{\circ}$ is defined similarly as $\bfSh(G')^{\circ}_{\ZZ_p^{\ur}}$; in other words, it is the closure of $\Sh(G)^{\circ}$ in $\bfSh(G)\otimes{\ZZ^\ur_p}$.

The natural inclusion $G^{\star}\hra G$ induces also an isomorphism of derived and adjoint groups, and is compatible with Deligne homomorphisms. By Deligne's theory of connected Shimura varieties, it induces an isomorphism of neutral connected components $\bfSh(G^{\star})^{\circ}\cong \bfSh(G)^{\circ}$. Therefore, we are reduced to comparing the universal family over $\bfSh(G^{\star})$ and $\bfSh(G')$.

Recall that we have chosen an element $\gamma\in B^{\times}$ to define the pairing $\ang{\_,\_}_F$ on $V=B$. We take the symmetric element $\delta\in D_{\tS}^{\times}$ in Section \ref{S:unitary-moduli} to be $\delta=\frac{\gamma}{2\sqrt{\fd}}$. One has $W=V\otimes_{F}E$, and
\[
\psi(x\otimes1, y\otimes 1)=\langle x,y\rangle
\]
for any $x,y\in V$. Put $\ang{\_,\_}\coloneqq\Tr_{F/\QQ}\circ \ang{\_,\_}_F$. Then $G^{\star}$ (resp.\ $G'$) can be viewed as the similitude group of $(V,\langle\_,\_\rangle)$ (resp.\ $(W,\psi)$ \eqref{E:unitary-alternating}); and there exists a natural injection $G^{\star}\hookrightarrow G$ compatible with Deligne homomorphisms that induces isomorphisms on the associated derived and adjoint groups. 

We take $\cO_{D_{\emptyset}}=\cO_B\otimes_{\cO_F}\cO_E$. Let $K^{\star p}\subseteq G^{\star}(\bA^{\infty,p})$ and $K'^p\subseteq G'(\bA^{\infty,p})$ be sufficiently small open compact subgroups with $K^{\star p}\subseteq K'^p$. To each point $(A,\iota,\lambda,\bar \alpha_{K^{\star, p}})$ of $\bfSh(G^{\star}, K^{\star, p})$ with values in a $\ZZ_{p}$-scheme $S$, we attach the tuple $(A',\iota',\lambda',\bar\alpha^{\rat}_{K'^p})$, where
\begin{itemize}
  \item $A'=A\otimes_{\cO_F}\cO_E$;

  \item $\iota'\colon \cO_{D_{\emptyset}}\to\End_{S}(A')$ is the action induced by $\iota$;

  \item $\lambda'\colon A'\to A'^{\vee}$ is the prime-to-$p$ polarization given by
    \[
    A'=A\otimes_{\cO_F}\cO_E\xra{\lambda\otimes 1} A^{\vee}\otimes_{\cO_F}\cO_E\xra{1\otimes i}  A^\vee\otimes_{\cO_F}\fd_{E/F}^{-1}\cong A'^\vee,
    \]
    where $\fd_{E/F}^{-1}$ is the inverse of the relative different of $E/F$ and $i\colon\cO_E\to \fd_{E/F}^{-1}$ is the natural inclusion;

  \item $\bar\alpha_{K'^p}^{\rat}$ is a rational $K'^p$-level structure on $A'$ induced by $\bar\alpha_{K^{\star, p}}$ by the compatibility of alternating forms $(V,\langle\_,\_\rangle)$ and $(W,\psi)$. Here, we use the moduli interpretation of $\bfSh(G',K'^p)$ in terms of isogeny classes of abelian varieties (See Remark \ref{R:rational-moduli-unitary}).
\end{itemize}
This defines a morphism
\[
\bfSh(G^{\star},K^{\star p})\to \bfSh(G', K'^p)
\]
over $\ZZ_p$ extending the morphism $\Sh(G^{\star},K'^{\star p})\otimes_\bQ{\QQ_p}\to\Sh(G', K'^p)\otimes_\bQ{\QQ_p}$. Taking the projective limit on the prime-to-$p$ levels, one gets a morphism of schemes over $\ZZ_p$:
\[
\bff\colon \bfSh(G^{\star})\to \bfSh(G')
\]
such that one has an isomorphism of abelian schemes:
\[
\bff^*\cA'\cong \cA\otimes_{\cO_F}\cO_E,
\]
where $\cA$ (resp.\ $\cA'$) is the universal abelian scheme over $\bfSh(G^{\star})$ (resp.\ over $\bfSh(G'_{\tilde\tS})$). By the extension property of the integral canonical model, the map $\bff$ induces an isomorphism
\[
\bff^{\circ}\colon \bfSh(G^{\star})^{\circ}\xra{\sim}\bfSh(G')^{\circ}
\]
which extends the isomorphism $\Sh(G^{\star})^{\circ}\xra{\sim} \Sh(G')^{\circ}$ induced by the morphism of Shimura data on the generic fibers.
Thus the two universal families over $\bfSh(G)^{\circ}$ induced from $\bfSh(G^{\star})$ and $\bfSh(G')$ respectively are related by the relation
\begin{equation}\label{E:comparison-univ-av}
\bff^{\circ,*}(\cA'\res_{\bfSh(G')^{\circ}})\cong \cA\res_{\bfSh(G)^{\circ}}\otimes_{\cO_F}\cO_E.
\end{equation}

\if false
Then
$G^{\star}$ is identified with the similitude group of $(V,\langle \_,\_\rangle)$, that is, for any $\QQ$-algebra $R$, we have
\[G^{\star}(R)=\{g\in\GL_{B\otimes R}(V\otimes R)\res\langle gx,gy\rangle=c\langle x,y\rangle \text{ for some $c\in R^{\times}$}\}.\]
Fix an order $\cO_B\subseteq B$ and an $\cO_B$-lattice $\Lambda\subseteq V$ such that
\begin{itemize}
  \item $\cO_B$ is stable under $*$,

  \item $\cO_B\otimes\ZZ_p$ is a maximal order of $B\otimes_{\QQ}\QQ_p$,

  \item and $(V,\langle\_,\_\rangle)$ induces a perfect pairing
     \[\Lambda\otimes\ZZ_p\times \Lambda\otimes \ZZ_p\to \ZZ_p.\]
\end{itemize}

The Shimura variety $\Sh(G^{\star}, K^{\star p})$ is of PEL-type, and it can be interpreted as a fine moduli space of certain polarized abelian varieties with additional structures. Denote by $b\mapsto \bar b$ the canonical involution on $B$. We choose an element $\gamma\in B^{\times}$ such that
\begin{itemize}
  \item $\bar \gamma =-\gamma$,
  \item $b\mapsto b^*\coloneqq \gamma^{-1} \bar b\gamma$ is a positive involution,
  \item $\nu(\gamma)$ is a $\fp$-adic unit for every $p$-adic place $\fp$ of $F$.
\end{itemize}
Let $V=B$ viewed as a free left $B$-module of rank $1$, and consider the alternating pairing
\[
\langle\_,\_ \rangle\colon  V\times  V\to \QQ, \quad \langle x,y\rangle =\Tr_{F/\QQ}\circ\Tr^{\circ}_{B/F}(x\bar y\gamma),
\]
where $\Tr^{\circ}_{B/F}$ is the reduced trace of $B$. Note that $\langle bx,y\rangle=\langle x, b^*y\rangle$ for $x,y\in V$ and $b\in B$, and
$G^{\star}$ is identified with the similitude group of $(V,\langle \_,\_\rangle)$, that is, for any $\QQ$-algebra $R$, we have
\[G^{\star}(R)=\{g\in\GL_{B\otimes R}(V\otimes R)\res\langle gx,gy\rangle=c\langle x,y\rangle \text{ for some $c\in R^{\times}$}\}.\]
Fix an order $\cO_B\subseteq B$ and an $\cO_B$-lattice $\Lambda\subseteq V$ such that
\begin{itemize}
  \item $\cO_B$ is stable under $*$,

  \item $\cO_B\otimes\ZZ_p$ is a maximal order of $B\otimes_{\QQ}\QQ_p$,

  \item and $(V,\langle\_,\_\rangle)$ induces a perfect pairing
     \[\Lambda\otimes\ZZ_p\times \Lambda\otimes \ZZ_p\to \ZZ_p.\]
\end{itemize}

Let $K^{\star p}\subseteq G^{\star}(\bA^{\infty,p})$ be a sufficiently small open compact subgroup stabilizing $\Lambda\otimes\widehat\ZZ^{(p)}$. We consider the functor that associates to every $\ZZ_{p}$-scheme $S$ the set of isomorphism classes $(A,\iota, \lambda, \eta)$ where
\begin{enumerate}
  \item $A$ is an abelian scheme over $S$ of dimension $2[F:\QQ]$;

  \item $\iota\colon \cO_B\to \End_S(A)$ is an action of $\cO_B$ on $A$ such that the characteristic polynomial of the induced action of $\iota(b)$ on $\Lie(A)$ for $b\in\cO_F$ is
      \[\det(T-\iota(b)|\Lie(A))=\prod_{\tau\in\Sigma_{\infty}}(T-\iota(b))^2;\]

  \item $\lambda\colon A\to A^\vee$ is an $\cO_B$-linear $\ZZ_{(p)}^{\times}$-polarization such that the Rosati involution on $\End_S(A)$ is compatible the involution $*$ on $\cO_B$;

  \item $\bar \eta$ is a $K^{\star p}$-level structure on $A$, that is, a $K^{\star p}$-orbit of $\cO_B\otimes \widehat\ZZ^{(p)}$-linear isomorphisms of \'etale sheaves
      \[
      \eta\colon \Lambda\otimes \widehat\ZZ^{(p)}\xra{\sim}\widehat{T}^p(A)
      \]
      such that the pairing $\langle\_,\_\rangle$ on $\Lambda\otimes\widehat\ZZ^{(p)}$ is compatible with $\lambda$-Weil pairing on $\widehat{T}^p(A)$ via some isomorphism $\widehat{\ZZ}^{(p)}\xra{\sim}\widehat\ZZ^{(p)}(1)$.
\end{enumerate}
Then this functor is representable by a quasi-projective (and projective if $B$ is a division algebra) and smooth scheme $\bfSh(G^{\star},K^{\star p})$ over $\ZZ_{p}$ with generic fiber isomorphic to $\Sh(G^{\star}, K^{\star p})$. Passing to limit on $K^{\star p}$, we get an integral canonical model $\bfSh(G^{\star})$ of $\Sh(G^{\star})$ over $\ZZ_p$.

To pass from the canonical model of the Shimura variety for $G^{\star}$ to  that for $G$, we can proceed in two equivalent ways. The first way uses the following Lemma.

\begin{lem}\label{L:comparison-quaternion}
For each open compact subgroup $K\subseteq G(\bA^{\infty})$, then there exist a finite set $I$ and an open compact subgroup $K_i^{\star}\subseteq G(\bA^{\infty})$ for each $i\in I$  and finite $2$-groups $\Delta_i $ acting freely on $\Sh_{K^{\star}_i}(G^{\star})$ for each $i\in I$ such that
\[
\Sh_K(G)=\bigsqcup_{i\in I} \Sh_{K_i^{\star}}(G^{\star})/\Delta_i.
\]
\end{lem}

\begin{proof}
It suffices to prove the statement at level of complex points. We have
\[
\Sh_K(G)(\CC)=G(\QQ)\backslash (\fH^{\pm})^{\Sigma_{\infty}}\times G(\bA^{\infty})/K=G(\QQ)_{+}\backslash \big(\fH^{\Sigma_{\infty}}\times G(\bA^{\infty}) /K\big),
\]
where $G(\QQ)_+\subseteq G(\QQ)$ is subgroup of elements with totally positive reduced norm. By the strong approximation theorem, there exist finitely many $g_i\in G(\bA^{\infty})$ such that
\[
G(\bA^{\infty})=\bigsqcup_{i\in I}G(\QQ)_+G^{\star}(\bA^{\infty})g_iK.
\]
Put $K^{\star}_i\coloneqq g_iKg_i^{-1}\cap G^{\star}(\bA^{\infty})$. Then by the strong approximation theorem again, there exist, for each $i\in I$, elements $h_{j}\in G^{\star}(\bA^{\infty})$ indexed by a finite set $J(i)$ such that
\[G^{\star}(\bA^{\infty})=\bigsqcup_{j\in J(i)}G^{\star}(\QQ)_+h_jK^{\star}_i,\]
where $G^{\star}(\QQ)_+=G(\QQ)_+\cap G^{\star}(\QQ)$. Then one has
\begin{align*}
\Sh_{K}(G)(\CC)&=\bigsqcup_{i\in I}\bigsqcup_{j\in J(i)} G(\QQ)_+\backslash\big( \fH^{\Sigma_{\infty}}\times G(\QQ)_+h_jg_iK\big/K)\\
&=\bigsqcup_{i\in I}\bigsqcup_{j\in J(i)}\Gamma_{h_jg_i}\backslash \fH^{\Sigma_{\infty}},
\end{align*}
where $\Gamma_{h_jg_i}\coloneqq G(\QQ)_+\cap h_jg_iKg_i^{-1}h_j^{-1}$. Similarly, one has for each $i\in I$
\begin{align*}
\Sh_{K^{\star}_i}(G^{\star})&=G^{\star}(\QQ)_+\backslash\big(\fH^{\Sigma_{\infty}}\times G^{\star}(\bA^{\infty})/K_i\big)\\
&=\bigsqcup_{j\in J(i)}\Gamma^{\star}_{h_jg_i} \backslash \fH^{\Sigma_{\infty}},
\end{align*}
where $\Gamma^{\star}_{h_jg_i}=G^{\star}(\QQ)_+\cap h_jK^{\star}_ih_j^{-1}=G^{\star}(\QQ)\cap\Gamma_{h_jg_i}$. Let $G^{\mathrm{ad}}$ be the adjoint group of $G$. Denote by $\overline \Gamma_{h_jg_i}$ and $\overline \Gamma^{\star}_{h_jg_i}$ the image of $\Gamma_{h_jg_i}$ and $\Gamma^{\star}_{h_jg_i}$ in $G^{\mathrm{ad}}(\QQ)$ respectively. Then one has
\[\Gamma_{h_jg_i}\backslash \fH^{\Sigma_{\infty}}=\Delta_{i,j}\backslash(\overline \Gamma^{\star}_{h_jg_i}\backslash\fH^{\Sigma_{\infty}})\]
with $\Delta_{i,j}\coloneqq\overline \Gamma_{h_jg_i}/\overline \Gamma^{\star}_{h_jg_i}$. Since $G(\QQ)_{+}/Z(\QQ)G^{\star}_{+}(\QQ)\cong F_{+}^{\times}/F^{\times}\QQ_{+}^{\times}$ is a finite copies of $\ZZ/2\ZZ$,  so is $\Delta_{i,j}$.

\end{proof}

Since the natural inclusion $G^{\star}\to G$ induces isomorphisms on the associated derived and adjoint groups, there exists an isomorphism of neutral geometric connected component
\[
\Sh(G^{\star})^{\circ}\cong \Sh(G)^{\circ},
\]
which can be defined over $\QQ_p^{\ur}$. Thus the canonical integral model $\bfSh(G^{\star})$ induces an integral canonical model $\bfSh(G)^{\circ}_{\ZZ_p^{\ur}}$ over $\ZZ_p^{\ur}$ of the connected Shimura variety $\Sh(G)^{\circ}$. As explained in \cite{TX1}*{Corollary~2.17}, since $\Sh(G)$ can be recovered from $\Sh(G)^{\circ}$ together with its Galois and Hecke action, $\bfSh(G)^{\circ}_{\ZZ_p^{\ur}}$ induces an integral canonical model over $\ZZ_p$ of $\Sh(G)$. By the universal extension property of integral canonical models, such an integral model is necessarily isomorphic to $\bfSh(G)$ that we constructed in the previous subsection. However, for later applications to the supersingular locus of $\bfSh(G)$, we need to understand how the universal family $\cA$ of abelian varieties on $\bfSh(G^{\star})$ is related to the universal family of abelian varieties over $\bfSh(G'_{\tilde\tS})$ with $\tS=\emptyset$.

Note that if $\tS=\emptyset$, the choice of $\tilde\tS_{\infty}$ is trivial, so we write $G'=G'_{\tilde\tS}$ for simplicity. The reflex field of $\Sh(G', K'^p)$ is $\QQ$. We take the symmetric element $\delta\in D_{\tS}^{\times}$ in Section \ref{S:unitary-moduli} to be $\delta=\frac{\gamma}{2\sqrt{\fd}}$. Then one has $W=V\otimes_{F}E$, and
\[\psi(x\otimes1, y\otimes 1)=\langle x,y\rangle\]
for any $x,y\in V$. Since $G^{\star}$ (resp.\ $G'$) can be viewed as the similitude group of $(V,\langle\_,\_\rangle)$ (resp.\ $(W,\psi)$), there exists a natural injection $G^{\star}\hookrightarrow G$ compatible with Deligne homomorphisms which induces isomorphisms between the associated derived and adjoint groups. By Deligne's theory of connected Shimura varieties \cite{De79}, we get a morphism of Shimura varieties $\Sh_{K^{\star}_p}(G^{\star})\to \Sh_{K'_p}(G')$ that induces an isomorphism of the neutral geometric connected components $\Sh_{K^{\star}_p}(G^{\star})^{\circ}\xra{\sim} \Sh_{K'_p}(G')^{\circ}$ over $\QQ_p^{\ur}$.

We take $\cO_{D_{\emptyset}}=\cO_B\otimes_{\cO_F}\cO_E$ and $L=\Lambda\otimes_{\cO_F}\cO_E$ in the definition of $\bfSh(G', K'^p)$. Let $K^{\star p}\subseteq G^{\star}(\bA^{\infty,p})$ and $K'^p\subseteq G'(\bA^{\infty,p})$ be sufficiently small open compact subgroups with $K^{\star p}\subseteq K'^p$. To each point $(A,\iota,\lambda,\eta)$ of $\bfSh(G^{\star}, K'^p)$ with values in a $\ZZ_p$-scheme $S$, we attach a point $(A',\iota',\lambda',\alpha')$, where
\begin{itemize}
  \item $A'=A\otimes_{\cO_F}\cO_E$;

  \item $\iota'\colon \cO_{D_{\emptyset}}\to\End_{S}(A')$ is the action induced by $\iota$;

  \item $\lambda'\colon A'\to A'^{\vee}$ is the prime-to-$p$ polarization given by
    \[
    A'=A\otimes_{\cO_F}\cO_E\xra{\lambda\otimes 1} A^{\vee}\otimes_{\cO_F}\cO_E\xra{1\otimes i}  A^\vee\otimes_{\cO_F}\fd_{E/F}^{-1}\cong A'^\vee,
    \]
    where $\fd_{E/F}^{-1}$ is the inverse of the relative different of $E/F$ and $i\colon\cO_E\to \fd_{E/F}^{-1}$ is the natural inclusion;

  \item $\alpha'$ is a $K'^p$-level structure on $A'$ induced by $\eta$ by the compatibility of alternating forms $(V,\langle\_,\_\rangle)$ and $(W,\psi)$.
\end{itemize}
This defines a morphism $\bfSh(G^{\star},K^{\star p})\to \bfSh(G', K'^p)$ over $\ZZ_p$ extending the morphism $\Sh(G^{\star}, K'^{\star p})_{\QQ_p}\to\Sh(G', K'^p)_{\QQ_p}$. Taking the projective limit on the prime-to-$p$ levels, one gets a morphism of schemes over $\ZZ_p$:
\[\bff\colon \bfSh(G^{\star})\to \bfSh(G')\]
such that one has an isomorphism of abelian schemes:
\[
\bff^*\cA'\cong \cA\otimes_{\cO_F}\cO_E,
\]
where $\cA$ (resp.\ $\cA'$) is the universal abelian scheme over $\bfSh(G^{\star})$ (resp.\ over $\bfSh(G'_{\tilde\tS})$). By the extension property of the integral canonical model, the map $\bff$ induces an isomorphism
\[
\bff^{\circ}\colon \bfSh(G^{\star})^{\circ}\xra{\sim}\bfSh(G')^{\circ}
\]
which extends the isomorphism $\Sh(G^{\star})^{\circ}\xra{\sim} \Sh(G')^{\circ}$ induced by the morphism of Shimura data on the generic fibers.
Thus the two universal families over $\bfSh(G)^{\circ}$ induced from $\bfSh(G^{\star})$ and $\bfSh(G')$ respectively are related by the relation
\begin{equation}\label{E:comparison-univ-av}
\cA'\res_{\bfSh(G)^{\circ}}\cong \cA\res_{\bfSh(G)^{\circ}}\otimes_{\cO_F}\cO_E.
\end{equation}

\fi

\section{Goren--Oort cycles and supersingular locus}
\label{ss:3}

In this chapter, we study the supersingular locus and the superspecial locus of certain Shimura varieties established in the previous chapter.

\subsection{Notation and conventions}
\label{S:notation}

Let $k$ be a perfect field containing all the residue fields of the auxiliary field $E$ in Section \ref{S:signature-numbers} at $p$-adic places, and $W(k)$ be the ring of Witt vectors. Then $\Sigma_{E,\infty}$ is in natural bijection with $\Hom_{\ZZ}(\cO_E,W(k))$, and we have a canonical decomposition
\[
\cO_{D_{\tS}}\otimes_{\ZZ}W(k)\cong \Mat_2(\cO_{E}\otimes_{\ZZ} W(k))=\bigoplus_{\tilde\tau\in\Sigma_{E,\infty}} \rM(W(k)).
\]
Let $S$ be a $W(k)$-scheme, and $N$ a coherent $\cO_S\otimes\cO_{D_{\tS}}$-module. Then one has a canonical decomposition
\[
N=\bigoplus_{\tilde\tau\in\Sigma_{E,\infty}}N_{\tilde\tau},
\]
where $N_{\tilde\tau}$ is a left $\Mat_2(\cO_S)$-module on which $\cO_E$ acts via $\tilde\tau\colon \cO_E\xra{\tilde\tau} W(k)\to\cO_S$. We also denote by $N^{\circ}_{\tilde\tau}$ the direct summand $\fe\cdot N_{\tilde\tau}$ with $\fe=\big(\begin{smallmatrix}1&0\\0&0\end{smallmatrix}\big)\in\Mat_2(\cO_S)$, and we call $M^{\circ}_{\tilde\tau}$ the \emph{reduced $\tilde\tau$-component} of $M$.

Consider a quaternionic Shimura variety $\Sh(G_{\tS,\tT},K^p)$ of type considered in Section \ref{S:quaternion-shimura}, and let $\bfSh(G_{\tilde\tS}', K'^p)$ be the associated unitary Shimura variety  over $\cO_{\tilde\wp}$ as constructed in Section \ref{S:unitary-moduli} for a certain choice of auxiliary CM extension $E/F$.
Let $k_0$ be the smallest subfield of $\bF_p^\ac$ containing all the residue fields of characteristic $p$ of  $E$. Then we have  $k_0\cong\bF_{p^h}$ with $h$ equal to the least common multiple of  $\left\{(1+g_\fp-2\lfloor g_\fp/2\rfloor)g_\fp\res\fp\in\Sigma_p\right\}$. Put
\[
\bfSh(G'_{\tilde\tS},K'^p)_{k_0}\coloneqq\bfSh(G'_{\tilde\tS},K'^p)\otimes_{\cO_{\tilde\wp}}k_0.
\]
The universal abelian scheme over $\bfSh(G'_{\tilde\tS},K'^p)_{k_0}$ is usually denoted by $\cA'_{\tilde\tS}$.

\subsection{Hasse invariants}
\label{S:Hasse-invariants}

We recall first the definition of essential invariant on $\bfSh(G_{\tilde\tS}',K'^p)_{k_0}$ defined in \cite{TX1}*{Section 4.4}. Let $(A,\iota,\lambda,\bar\alpha_{K'^p})$ be an $S$-valued point of $\bfSh(G_{\tilde\tS}', K'^p)_{k_0}$ for some $k_0$-scheme $S$. Recall that $\rH^{\dr}_1(A/S)$ is the relative de Rham homology of $A$. Let $\omega_{A^\vee}$ be the module of invariant differential $1$-forms on $A^\vee$. Then for each $\tilde\tau\in\Sigma_{E,\infty}$, $\rH^{\dr}_1(A/S)_{\tilde\tau}$ is a locally free $\cO_S$-module on $S$ of rank $2$, and one has a Hodge filtration
\[
0\to \omega_{A^{\vee}, \tilde\tau}^{\circ}\to \rH^{\dr}_1(A/S)^{\circ}_{\tilde\tau}\to \Lie(A/S)_{\tilde\tau}^{\circ}\to 0.
\]
We defined, for each $\tilde\tau\in\Sigma_{E,\infty}$, the essential Verschiebung
\[
V_{\es,\tilde\tau}\colon  \rH^{\dr}_1(A/S)^{\circ}_{\tilde\tau}\to \rH^{\dr}_1(A^{(p)}/S)^{\circ}_{\tilde\tau}\cong \rH^{\dr}_1(A/S)^{\circ,(p)}_{\sigma^{-1}\tilde\tau},
\]
to be the usual Verschiebung map if $s_{\sigma^{-1}\tilde\tau}=0$ or $1$, and to be the inverse of Frobenius if $s_{\tilde\tau}=2$.
This is plausible since for $s_{\tilde\tau}=2$, the Frobenius map $F\colon \rH^{\dr}_1(A^{(p)}/S)^{\circ}_{\tilde\tau}\to\rH^{\dr}_1(A/S)_{\tilde\tau}^{\circ}$ is an isomorphism. For every integer $n\geq 1$, we denote by
\[
V_{\es}^{n}\colon \rH_1^{\dr}(A/S)^{\circ}_{\tilde\tau}\to \rH^{\dr}_1(A^{(p^n)}/S)_{\tilde\tau}^{\circ}\cong \rH^{\dr}_1(A/S)_{\sigma^{-n}\tilde\tau}^{\circ, (p^n)}
\]
the $n$-th iteration of the essential Verschiebung.

Similarly, if $S=\Spec k$ is the spectrum of a perfect field $k$ containing $k_0$, then one can define the essential Verschiebung
\[
V_{\es}\colon \tcD(A)^{\circ}_{\tilde\tau}\to \tcD(A)^{\circ}_{\sigma^{-1}\tilde\tau}\quad \text{for all }\tilde\tau\in\Sigma_{E,\infty},
\]
as the usual Verschiebung on Dieudonn\'e modules if $s_{\tilde\tau}=0,1$ and as the inverse of the usual Frobenius if $s_{\tilde\tau}=2$. Here $\tilde\cD(A)$ denote the covariant Dieudonn\'e module of $A[p^{\infty}]$. This is a $\sigma^{-1}$-semi-linear map of $W(k)$-modules. For any integer $n\geq 1$, we denote also by
\[V_{\es}^n\colon \tcD(A)^{\circ}_{\tilde\tau}\to \tcD(A)^{\circ}_{\sigma^{-n}\tilde\tau}\]
the $n$-th iteration of the essential Verschiebung.

Now return to  a general base $S$ over $k_0$. For $\tau\in\Sigma_{\infty}-\tS_{\infty}$, let $n_{\tau}=n_{\tau}(\tS)$ denote the smallest integer $n\geq 1$ such that $\sigma^{-n}\tau\in\Sigma_{\infty}-\tS_{\infty}$. Assumption \ref{A:assumption-S} implies that $n_{\tau}$ is odd. Then for each $\tilde\tau\in\Sigma_{E,\infty}$ with $s_{\tilde\tau}=1$, or equivalently each $\tilde\tau\in\Sigma_{E,\infty}$ lifting some $\tau\in\Sigma_{\infty}-\tS_{\infty}$, the restriction of $V_{\es}^{n_{\tau}}$ to $\omega_{A^\vee, \tilde\tau}^{\circ} $ defines a map
\[
h_{\tilde\tau}(A)\colon \omega^{\circ}_{A^\vee, \tilde\tau}\to \omega_{A^\vee, \sigma^{-n_\tau}\tilde\tau}^{\circ, (p^{n_{\tau}})}\cong(\omega_{A^\vee, \sigma^{-n_\tau}\tilde\tau}^{\circ})^{\otimes p^{n_{\tau}}}.
\]
Applying this construction to the universal object, one gets a global section
\begin{equation}\label{E:Hasse-invariant}
h_{\tilde\tau}\in\Gamma(\bfSh(G_{\tilde\tS}', K'^p)_{k_0}, (\omega^{\circ}_{\cA'^{\vee}_{\tilde\tS}, \sigma^{-n_{\tau}}\tilde\tau})^{\otimes p^{n_{\tau}}}\otimes (\omega^{\circ}_{\cA'^{\vee}_{\tilde\tS},\tilde\tau})^{\otimes -1}).
\end{equation}
called the \emph{$\tau$-th partial Hasse invariant}.

\begin{proposition}\label{P:ordinary}
Let $x=(A,\iota, \lambda,\bar  \alpha_{K'^p})$ be an $\bF_p^\ac$-point of $\bfSh(G_{\tilde\tS}', K'^p)_{k_0}$, and $\fp$ a $p$-adic place of $F$ such that $\tS_{\infty/\fp}\neq \Sigma_{\infty/\fp}$. Assume that $h_{\tilde\tau}(A)\neq 0$ for all $\tilde\tau\in\Sigma_{E,\infty/\fp}$ with $s_{\tilde\tau}=1$. Then the $p$-divisible group $A[\fp^{\infty}]$ is not supersingular.
\end{proposition}

\begin{proof}
The covariant Dieudonn\'e module $\tilde\cD(A)$ of $A[p^{\infty}]$ is a free $W(\bF_p^\ac)\otimes_{\ZZ}\cO_{D_{\tS}}$-module of rank $1$. Then the covariant Dieudonn\'e module of $A[\fp^{\infty}]$ is given by
\[
\tilde\cD(A[\fp^{\infty}])=\bigoplus_{\tilde\tau\in\Sigma_{E,\infty/\fp}}\tilde\cD(A)_{\tilde\tau}^{\circ, \oplus 2},
\]
and there exists a canonical isomorphism
\[
\tilde\cD(A)_{\tilde\tau}^{\circ}/p\tilde\cD(A)_{\tilde\tau}^{\circ}\cong \rH^{\dr}_1(A/\bF_p^\ac)^{\circ}_{\tilde\tau}.
\]


By assumption, for all $\tilde\tau\in\Sigma_{E,\infty/\fp}$ lifting some $\tau\in\Sigma_{\infty/\fp}-\tS_{\infty/\fp}$, the map
\[
h_{\tilde\tau}(A)\colon \omega_{A^\vee,\tilde\tau}^{\circ}\to \omega_{A^{\vee}, \sigma^{-n_{\tau}}\tilde\tau}^{\circ, (p^n)}
\]
is non-vanishing. Thus it is an isomorphism, as both the source and the target are one-dimensional $\bF_p^\ac$-vector spaces. For each $\tilde\tau\in\Sigma_{E,\infty/\fp}$ lifting some $\tau\in\Sigma_{\infty/\fp}-\tS_{\infty/\fp}$, choose a basis $e_{\tilde\tau}$ for $\omega_{A^{\vee}, \tilde\tau}^{\circ}$, and extend it to a basis $(e_{\tilde\tau},f_{\tilde\tau})$ of $\rH^{\dr}_1(A/\bF_p^\ac)^{\circ}_{\tilde\tau}$. If we consider $V_{\es}$ as a $\sigma^{-1}$-linear map on $\rH^{\dr}_1(A/\bF_p^\ac)^{\circ}_{\tilde\tau}$, then one has
\[
V_{\es}^{n_{\tau}}(e_{\tilde\tau}, f_{\tilde\tau})=(e_{\sigma^{-n_{\tau}}\tilde\tau},f_{\sigma^{-n_{\tau}}\tilde\tau})
\begin{pmatrix}u_{\tilde\tau} &0\\ 0& 0\end{pmatrix}
\]
with $u_{\tilde\tau}\in\bF_p^{\ac,\times}$.

Let $\fq$ be a $p$-adic place of $E$ above $\fp$. By our choice of $E$,  $g_{\fq}\coloneqq[E_{\fq}:\QQ_p]$  is always even  no matter whether $\fp$ is split or inert in $E$.   To prove the proposition, it suffices to show that the $p$-divisible group $A[\fq^{\infty}]$ is not supersingular. By composing the essential Verschiebung maps on all $\rH^{\dr}_1(A/S)^{\circ}_{\tilde\tau}$ with $\tilde\tau\in\Sigma_{E,\infty/\fq}$, we get
\[
V_{\es}^{g_{\fq}}(e_{\tilde\tau},f_{\tilde\tau})=(e_{\tilde\tau}, f_{\tilde\tau})\begin{pmatrix}
\bar a_{\tilde\tau} & 0 \\ 0& 0
\end{pmatrix}
\]
with $\bar a_{\tilde\tau}\in\bF_p^{\ac,\times}$ for all $\tilde\tau\in\Sigma_{E,\infty/\fq}$ with $s_{\tilde\tau}=1$. Now, note that  $V_{\es}^{g_{\fq}}$ on $\rH^{\dr}_1(A/\bF_p^\ac)^{\circ}_{\tilde\tau}$ is nothing but the reduction modulo $p$ of the $\sigma^{-g_{\fq}}$-linear map
\[
V^{g_{\fq}}/p^{m}\colon \tilde\cD(A)_{\tilde\tau}^{\circ}\to \tilde\cD(A)^{\circ}_{\tilde\tau},
\]
where $m$ is the number of $\tilde\tau\in\Sigma_{E,\infty/\fq}$ with $s_{\tilde\tau}=2$. If $(\tilde e_{\tilde\tau},\tilde f_{\tilde\tau})$ is a  lift of $(e_{\tilde\tau}, f_{\tilde\tau})$ to a basis of $\tilde\cD(A)^{\circ}_{\tilde\tau}$ over $W(\bF_p^\ac)$, then $V^{g_{\fq}}/p^m$ on $\tilde\cD(A)^{\circ}_{\tilde\tau}$ is given by
\[
\frac{V^{g_{\fq}}}{p^m}(\tilde e_{\tilde\tau},\tilde f_{\tilde\tau})=(\tilde e_{\tilde\tau}, \tilde f_{\tilde\tau})\begin{pmatrix}  a_{\tilde\tau} & pb_{\tilde\tau}\\ p c_{\tilde\tau} & p d_{\tilde\tau} \end{pmatrix}
\]
for some $a_{\tilde\tau}\in W(\bF_p^\ac)^{\times}$ lifting $\bar a_{\tilde\tau}$ and $b_{\tilde\tau }, c_{\tilde\tau}, d_{\tilde\tau}\in W(\bF_p^\ac)$. Put
\[
L\coloneqq\bigcap_{n\geq 1} \big(\frac{V^{g_{\fq}}}{p^m}\big)^n\tilde\cD(A)^{\circ}_{\tilde\tau}.
\]
It is easy to see that $L$ is a $W(\bF_p^\ac)$-direct summand of $\tilde\cD(A)^{\circ}_{\tilde\tau}$ of rank one, on which $V^{g_{\fp}}/p^m$ acts bijectively. It follows that $1-m/g_{\fq}$ is a slope of the $p$-divisible group $A[\fq^{\infty}]$. By our choice of the $s_{\tilde\tau}$'s  in Section \ref{S:signature-numbers}, the two sets  $\{\tilde\tau\in\Sigma_{E,\infty/\fq}\res s_{\tilde\tau}=2\}$ and $\{\tilde\tau\in\Sigma_{E,\infty/fq}\res s_{\tilde\tau}=0\}$ have the same cardinality, hence $2m<g_{\fq}$, i.e. $1-m/g_{\fq}>1/2$. Therefore, $A[\fq^{\infty}]$ hence  $A[\fp^{\infty}]$, are not supersingular.

\end{proof}

\subsection{Goren--Oort divisors}

For each $\tau\in\Sigma_{\infty}-\tS_{\infty}$, let $\bfSh(G_{\tilde\tS}', K'^p)_{k_0, \tau}$ be the closed subscheme of $\bfSh(G_{\tilde\tS}', K'^p)_{k_0}$ defined by the vanishing of $h_{\tilde\tau}$ for some $\tilde\tau\in\Sigma_{E,\infty}$ lifting $\tau$.
By \cite{TX1}*{Lemma~4.5}, $h_{\tilde\tau}$ vanishes at a point $x$ of $\bfSh(G_{\tilde\tS,\tT}', K'^p)_{k_0}$ if and only if $h_{\tilde\tau^c}$ vanishes at $x$. In particular, $\bfSh(G_{\tilde\tS}', K'^p)_{k_0,\tau}$ does not depend on the choice of $\tilde\tau$ lifting $\tau$. We call $\bfSh(G_{\tilde\tS}',K'^p)_{k_0,\tau}$ the $\tau$-th \emph{Goren--Oort divisor} of $\bfSh(G_{\tilde\tS}', K'^p)_{k_0}$. For a non-empty subset $\Delta\subseteq \Sigma_{\infty}-\tS_{\infty}$, we put
\begin{equation*}
\bfSh(G_{\tilde\tS}',K'^p)_{k_0,\Delta}\coloneqq\bigcap_{\tau\in\Delta}\bfSh(G_{\tilde\tS}',K'^p)_{k_0,\tau}.
\end{equation*}
According to \cite{TX1}*{Proposition~4.7}, $\bfSh(G_{\tilde\tS}', K'^p)_{k_0, \Delta}$ is a proper and smooth closed subvariety of $\bfSh(G_{\tilde\tS}',K'^p)_{k_0}$ of codimension $\#\Delta$; in other words, the union $\bigcup_{\tau\in\Sigma_{\infty}-\tS_{\infty}}\bfSh(G_{\tilde\tS}', K'^p)_{k_0, \tau}$
is a strict normal crossing divisor of $\bfSh(G_{\tilde\tS}', K'^p)_{k_0}$.

In \cite{TX1}, we gave an explicit description of $\bfSh(G_{\tilde\tS}', K'^p)_{k_0,\tau}$ in terms of another unitary Shimura variety of type in Section \ref{S:unitary-moduli}. To describe this, let $\fp\in\Sigma_p$ denote the $p$-adic place induced by $\tau$. Set
\begin{equation}\label{D:S-tau}
\tS_{\tau}=\begin{cases}\tS\cup \{\tau, \sigma^{-n_{\tau}}\tau\} & \text{if }\Sigma_{\infty/\fp}\neq \tS_{\infty/\fp}\cup \{\tau\},\\
\tS\cup \{\tau, \fp\} &\text{if }\Sigma_{\infty/\fp}=\tS_{\infty/\fp}\cup\{\tau\},
\end{cases}
\end{equation}
We fix a lifting $\tilde\tau\in\Sigma_{E,\infty}$ of $\tau$, and take $\tilde\tS_{\tau, \infty}$ to be $\tilde\tS_{\infty}\cup\{\tilde\tau,\sigma^{-n_{\tau}}\tilde\tau^c\}$ if $\Sigma_{\infty/\fp}\neq \tS_{\infty/\fp}\cup \{\tau\}$, and to be $\tilde\tS_{\infty}\cup \{\tilde\tau\}$ if $\Sigma_{\infty/\fp}=\tS_{\infty/\fp}\cup\{\tau\}$. This choice of $\tilde\tS_{\tau,\infty}$ satisfies Assumption \ref{A:assumption-S-tilde}. We note that both $D_{\tS}$ and $D_{\tS_{\tau}}$ are isomorphic to $\Mat_2(E)$. We fix an isomorphism $D_{\tS}\cong D_{\tS_{\tau}}$, and let $\cO_{D_{\tS_{\tau}}}$ denote the order of $D_{\tS_{\tau}}$ correspond to $\cO_{D_{\tS}}$ under this isomorphism.

\begin{proposition}[\cite{TX1}*{Theorem~5.2}]\label{P:GO-divisors}
Under the above notation, there exists a canonical projection
\[
\pi_{\tau}'\colon \bfSh(G_{\tilde\tS}', K'^p)_{k_0,\tau}\to \bfSh(G'_{\tilde\tS_{\tau}}, K'^p)_{k_0}
\]
where
\begin{enumerate}
  \item if $\Sigma_{\infty/\fp}\neq \tS_{\infty/\fp}\cup \{\tau\}$, then $\pi'_{\tau}$ is a $\dP^1$-fibration over $\bfSh(G'_{\tilde\tS_{\tau}},K'^p)_{k_0}$ such that the restriction of $\pi'_{\tau}$ to ${\bfSh(G_{\tilde\tS}',K'^p)_{k_0,\{\tau,\sigma^{-n_{\tau}}\tau\}}}$ is an isomorphism;

  \item if $\Sigma_{\infty/\fp}=\tS_{\infty/\fp}\cup\{\tau\}$, then $\pi'_{\tau}$ is an isomorphism.
\end{enumerate}
Moreover, $\pi'_{\tau}$ is equivariant under prime-to-$p$ Hecke correspondences when $K'^p$ varies, and there exists a $p$-quasi-isogeny
\[
\phi\colon\cA'_{\tilde\tS}\res_{\bfSh(G'_{\tilde\tS},K'^p)_{k_0,\tau}}\to\pi'^{*}_{\tau}\cA'_{\tilde\tS_{\tau}}
\]
that is compatible with polarizations and $K'^p$-level structures on both sides, and that induces an isomorphism of relative de Rham homology groups
\[
\phi_{*,\tau}\colon \rH^{\dr}_1(\cA'_{\tilde\tS}\res_{\bfSh(G'_{\tilde\tS}, K'^p)_{k_0,\tau}}/\bfSh(G'_{\tilde\tS},K'^p)_{k_0,\tau})^{\circ}_{\tilde\tau'}\cong  \rH^{\dr}_1(\cA'_{\tilde\tS_{\tau}}/ \bfSh(G'_{\tilde\tS_{\tau}}, K'^p))^{\circ}_{\tilde\tau'}
\]
for any $\tilde\tau'\in \Sigma_{E,\infty/\fp}$ lifting some $\tau'\in\Sigma_{\infty}- \tS_{\tau,\infty/\fp}$.

\end{proposition}

Here, we are content with explaining the map $\pi'_{\tau}$ and the quasi-isogeny $\phi$ on $\bF_p^\ac$-points. Take $x=(A,\iota_A,\lambda_A,\alpha_A)\in\bfSh(G'_{\tilde\tS,\tT},K'^p)_{k_0,\tau}(\bF_p^\ac)$. Denote by $\tcD(A)^{\circ}=\bigoplus_{\tilde\tau'\in\Sigma_{E,\infty}}\tcD(A)^{\circ}_{\tilde\tau'}$ the reduced covariant Dieudonn\'e module as usual. Consider a $W(\bF_p^\ac)$-lattice $M^{\circ}=\bigoplus_{\tilde\tau'\in\Sigma_{E,\infty}}M_{\tilde\tau'}$ of $\tcD(A)^{\circ}[1/p]$ such that
\[
M_{ \tilde\tau'}^{\circ}=
\begin{cases}
F_{\es}^{n_{\tau}-\ell}\tcD(A)_{\sigma^{-n_{\tau}}\tilde\tau}^{\circ}
& \text{if  }\tilde\tau'=\sigma^{-\ell} \tilde\tau \text{ with $0\leq \ell \leq n_{\tau}-1$,}\\
\frac{1}{p}F_{\es}^{n_{\tau}-\ell}\tcD(A)_{\sigma^{-n_{\tau}}\tilde\tau^c}^{\circ}
& \text{if }\tilde\tau'=\sigma^{-\ell}\tilde\tau^c \text{ with $0\leq \ell\leq n_{\tau}-1$ and $\Sigma_{\infty/\fp}\neq \tS_{\infty/\fp}\cup \{\tau\}$},\\
\tcD(A)^{\circ}_{\tilde\tau'} &\text{otherwise}.
\end{cases}
\]
Note that the condition that $h_{\tilde\tau}(A)=0$ is equivalent to $\tilde\omega_{A^\vee,\tilde\tau}^{\circ}=F_{\es}^{n_{\tau}}(\tcD(A)^{\circ}_{\sigma^{-n_{\tau}}\tilde\tau})$, where $\tilde\omega_{A^\vee,\tilde\tau}^{\circ}$ denotes the preimage of $\omega^{\circ}_{A^\vee,\tilde\tau}$ under the natural reduction map
\[
\tcD(A)^{\circ}_{\tilde\tau}\to \tcD(A)^{\circ}_{\tilde\tau}/p\tcD(A)^{\circ}_{\tilde\tau}\cong\rH^{\dr}_1(A/\bF_p^\ac)^{\circ}_{\tilde\tau}.
\]
Using this property, one checks easily that $M^{\circ}$ is a Diedonn\'e submodule of $\tcD(A)^{\circ}[1/p]$. Put $M\coloneqq M^{\circ,\oplus 2}$ equipped with the natural action of $\cO_{D_{\tS}}\otimes \ZZ_p\cong \Mat_2(\cO_{E}\otimes \ZZ_p)$. Then $M$ corresponds to a $p$-divisible group $G$ equipped with an $\cO_{D_{\tS}}$-action and an $\cO_{D_{\tS}}$-linear isogeny $\phi_p\colon A[p^{\infty}]\to G$. Thus there exists an abelian variety $B$ over $\bF_p^\ac$ with $B[p^{\infty}]=G$ and a $p$-quasi-isogeny $\phi\colon A\to B$ such that $\phi_p$ is obtained by taking the $p^{\infty}$-torsion of $\phi$. Moreover, by construction, it is easy to see that
\[
\dim \Lie(B)_{\tilde\tau'}^{\circ}
=\begin{cases}
\dim(\Lie(A)_{\tilde\tau'}^{\circ}) &\text{if $\tilde\tau'\neq\tilde\tau, \sigma^{-n_{\tau}}\tilde\tau$,}\\
0 &\text{if }\tilde\tau'=\tilde\tau, \sigma^{-n_{\tau}}\tilde\tau^{c},\\
2 & \text{if }\tilde\tau'=\tilde\tau^c, \sigma^{-n_{\tau}}\tilde\tau.
\end{cases}
\]
In other words, the $\cO_E$-action on $B$ satisfies Kottwitz' condition for $\bfSh(G'_{\tilde\tS_{\tau}}, K'^p)$. Moreover, $\lambda_A$ and $\alpha_A$ induce an $\cO_{D_{\tS_{\tau}}}$-linear prime-to-$p$ polarization $\lambda_B$ via the fixed isomorphism $\cO_{D_{\tS}}\simeq\cO_{D_{\tS_{\tau}}}$ and a $K'^p$-level structure $\alpha_B$ on $B$, respectively, such that $(B,\iota_{B},\lambda_B,\bar\alpha_B)$ is an $\bF_p^\ac$-point of $\bfSh(G_{\tilde\tS_{\tau}},K'^p)$. The resulting map
$(A,\iota_A,\lambda_A,\bar\alpha_A)\mapsto(B,\iota_{B},\lambda_B,\bar \alpha_B)$ is nothing but $\pi'_{\tau}$.

If $\Sigma_{\infty/\fp}\neq \tS_{\infty/\fp}\cup\{\tau\}$, then $\sigma^{-n_{\tau}}\tau\neq \tau$ and we have $\tcD(B)^{\circ}_{\sigma^{-n_{\tau}}\tilde\tau}=\tcD(A)_{\sigma^{-n_{\tau}}\tilde\tau}^{\circ}$ by construction. To recover $A$ from $B$, it suffices to ``remember'' the line $\omega^{\circ}_{A^\vee, \sigma^{-n_{\tau}}\tilde\tau}$ inside the two dimensional $\bF_p^\ac$-vector space
\[
\tcD(A)_{\sigma^{-n_{\tau}}\tilde\tau}^{\circ}/p\tcD(A)^{\circ}_{\sigma^{-n_{\tau}}\tilde\tau}
=\tcD(B)^{\circ}_{\sigma^{-n_{\tau}}\tilde\tau}/p\tcD(B)^{\circ}_{\sigma^{-n_{\tau}}\tilde\tau}.
\]
This means that the fiber of $\pi'_{\tau}$ over a point $(B,\iota_B,\lambda_B,\bar\alpha_B)\in\bfSh(G_{\tilde\tS_{\tau}}',K'^p)$ is isomorphic to $\dP^1$. On the other hand, if $\Sigma_{\infty/\fp}=\tS_{\infty/\fp}\cup\{\tau\}$ then $n_{\tau}=[F_{\fp}:\QQ_p]$ is odd, one can completely recover $A$ from $B$, and thus $\pi'_{\tau}$ induces a bijection on closed points\footnote{To show that $\pi'_{\tau}$ is indeed an isomorphism, one has to check also that $\pi'_{\tau}$ induces  isomorphisms of tangent spaces to each closed point. This is the most technical part of \cite{TX1}. For more details,  see \cite{TX1}*{Lemma 5.20}.}. The moreover part of the statement follows from the construction of $\pi'_{\tau}$.

\if false
The basic idea for the proof of this Proposition is to show that if $(A,\iota,\lambda, \alpha)$  is a point in $X'_{\tilde\tS,\tau}$, then $A$ is quasi-isogenous to an abelian variety parameterized by $\bfSh_{K'}(G'_{\tilde\tS,\tau})$ or $\bfSh_{K'^pK'_{\tau,p}}(G'_{\tilde\tS,\tau})$.

For later applications, we describe this in more detail. Consider the moduli space $Y'_{\tilde\tS, \tau}$ which associates to each $k_0$-scheme $S$ the isomorphism classes of tuples $(A,\iota_A,\lambda_A,\alpha_A,B,\iota_B, \lambda_B, \alpha_B; C, \iota_C; \phi_A,\phi_B)$ satisfying the following conditions:
\begin{enumerate}
  \item[(i)] $(A,\iota_A,\lambda_A,\alpha_{A})$ is an $S$-point of $X'_{\tilde\tS,\tau}$;

  \item[(ii)] $(B,\iota_B,\lambda_B,\alpha_B)$ is an $S$-point of $\bfSh_{K'}(G'_{\tilde\tS_{\tau}})$ (resp.\ of $\bfSh_{K'^pK'_{\tau,p}}(G'_{\tilde\tS_{\tau}})$) if $\Sigma_{\infty/\fp}\neq \tS_{\infty/\fp}\cup \{\tau\}$ (resp.\ if $\Sigma_{\infty/\fp}=\tS_{\infty/\fp}\cup \{\tau\}$);

  \item[(iii)] $C$ is an abelian scheme over $S$ of dimension $4[F:\QQ]$, and $\iota_C\colon \cO_{D_{\tS}}\to \End_{S}(C)$ is an action of $\cO_{D}$ on $C$;

  \item[(iv)] $\phi_A\colon A\to C$ and $\phi_B\colon B\to C$ are $\cO_D$-linear isogenies such that $\ker(\phi_A)$ and $\ker(\phi_B)$ are contained in the $p$-torsion of $A$ and $B$ respectively, and that the induced morphism on the de Rham homology groups satisfy certain conditions specified in \cite{TX1}*{Section 5.15};

  \item[(v)] We also require that the $p$-isogeny
      \[A\xrightarrow{\phi_A}C\xleftarrow{\phi_B}B\]
      is compatible with the polarizations and the $K'^p$-level structures on $A$ and $B$.

\end{enumerate}
We get thus a diagram of morphisms
\[
\xymatrix{& Y'_{\tilde\tS,\tau}\ar[rd]^{\pr_2}\ar[ld]_{\pr_1}^{\cong}\\
X'_{\tilde\tS,\tau} && \bfSh_{K'_{\tau}}(G'_{\tilde\tS_{\tau}}),}
\]
where $K'_{\tau}=K'$ if $\Sigma_{\infty/\fp}\neq \tS_{\infty/\fp}\cup \{\tau\}$ and $K'=K'^pK'_{\tau, p}$ if $\Sigma_{\infty/\fp}= \tS_{\infty/\fp}\cup \{\tau\}$, $\pr_1$ sends a tuple
$(A,\iota_A,\lambda_A,\alpha_A,B,\iota_B, \lambda_B, \alpha_B; C, \iota_C; \phi_A,\phi_B)$ to $(A,\iota_A,\lambda_A,\alpha_A)$ and $\pr_2$ sends such a tuple to $(B,\iota_B, \lambda_B, \alpha_B)$.

Then $\pr_1$ is an isomorphism by \cite{TX1}{Proposition~5.17}, and $\pr_2$ is a $\bP^1$-fibration over $\bfSh_{K'}(G'_{\tilde\tS_{\tau}})$ if  $\Sigma_{\infty/\fp}\neq \tS_{\infty/\fp}\cup \{\tau\}$ and it is an isomorphism if $\Sigma_{\infty/\fp}=\tS_{\infty/\fp}\cup \{\tau\}$ by \cite{TX1}*{Proposition~5.23}.

\fi

\subsection{Periodic semi-meanders}

Following \cite{TX2}, we iterate the construction of Goren--Oort divisors to produce some closed subvarieties called Goren--Oort cycles. To parameterize those cycles, one need to recall some combinatory data introduced in \cite{TX2}*{Section 3.1}.


For a prime $\fp\in\Sigma_p$, put $d_{\fp}(\tS)\coloneqq g_{\fp}-\#\tS_{\infty/\fp}$. If there is no confusion, then we write $d_{\fp}=d_{\fp}(\tS)$ for simplicity. Consider the cylinder $C\colon x^2+y^2=1$ in $3$-dimensional Euclidean space, and let $C_0$ be the section with $z=0$. We write $\Sigma_{\infty/\fp}=\{\tau_0,\dots, \tau_{g_{\fp}-1}\}$ such that $\tau_j=\sigma\tau_{j-1}$ for $j\in\ZZ/g_{\fp}\ZZ$. For $0\leq j\leq g_{\fp}-1$, we use $\tau_j$ to label the point $(\cos\frac{2\pi j}{f_{\fp}},\sin\frac{2\pi j}{g_{\fp}},0)$ on $C_0$. If $\tau_j\in\tS_{\infty/\fp}$, then we put a plus sign at $\tau_j$; otherwise, we put a node at $\tau_j$. We call such a picture \emph{the band} associated to $\tS_{\infty/\fp}$. We often draw the picture on the 2-dimensional $xy$-plane by thinking of $x$-axis modulo $g_{\fp}$. We put the points $\tau_0, \dots, \tau_{g_{\fp}-1}$ on the $x$-axis with coordinates $x=0,\dots, g_{\fp}-1$ respectively. For example, if $g_{\fp}=6$ and $\tS_{\infty/\fp}=\{\tau_1,\tau_3,\tau_4\}$, then we draw the band as
\[
\psset{unit=0.3}
\begin{pspicture}(-.5,-0.3)(5.5,0.3)
\psset{linecolor=black}
\psdots(0,0)(2,0)(5,0)
\psdots[dotstyle=+](1,0)(3,0)(4,0)
\end{pspicture}
.\]

A \emph{periodic semi-meander} for $\tS_{\infty/\fp}$ is a collection of curves (called \emph{arcs}) that link two nodes of the band for $\tS_{\infty/\fp}$, and straight lines (called \emph{semi-lines}) that links a node to the infinity (that is, the direction $y\to +\infty$ in the $2$-dimensional picture) subject to the following conditions:
\begin{enumerate}
  \item All the arcs and semi-lines lie on the cylinder above the band (that is to have positive $y$-coordinate in the $2$-dimensional picture).

  \item Every node of the band for $\tS_{\infty/\fp}$ is exactly one end point of an arc or a semi-line.

  \item There are no intersection points among these arcs and semi-lines.
\end{enumerate}
The number of arcs is denoted by $r$ (so $r \leq d_{\fp}/2$), and the number of semi-lines $d_{\fp}-2r$ is called the \emph{defect} of the periodic semi-meander. Two periodic semi-meanders are considered as the same if they can be continuously deformed into each other while keeping the above three properties in the process. We use $\fB(\tS_{\infty/\fp}, r)$ denote the set of semi-meanders for $\tS_{\infty/\fp}$ with $r$ arcs (up to continuous deformations). For example, if $g_{\fp}=7$, $r=2$, and $\tS_{\infty/\fp} =\{\tau_1,\tau_4\}$, then we have $d_{\fp}=5$ and
\begin{eqnarray*}
\fB(\tS_{\infty/\fp}, 2)=\
\Big\{
\psset{unit=0.6}
\begin{pspicture}(-0.25,-0.1)(3.25,0.75)
\psset{linewidth=1pt}
\psset{linecolor=red}
\psarc{-}(1.25,0){0.25}{0}{180}
\psarc{-}(2.75,0){0.25}{0}{180}
\psline{-}(0,0)(0,0.75)
\psset{linecolor=black}
\psdots(0,0)(1,0)(1.5,0)(2.5,0)(3.0)
\psdots[dotstyle=+](0.5,0)(2,0)
\end{pspicture},
\begin{pspicture}(-0.25,-0.1)(3.25,0.75)
\psset{linewidth=1pt}
\psset{linecolor=red}
\psarc{-}(-0.25,0){0.25}{0}{90}
\psarc{-}(3.25,0){0.25}{90}{180}
\psarc{-}(2,0){0.5}{0}{180}
\psline{-}(1,0)(1,0.75)
\psset{linecolor=black}
\psdots(0,0)(1,0)(1.5,0)(2.5,0)(3.0)
\psdots[dotstyle=+](0.5,0)(2,0)
\end{pspicture},\
\begin{pspicture}(-0.25,-0.1)(3.25,0.75)
\psset{linewidth=1pt}
\psset{linecolor=red}
\psarc{-}(0.5,0){0.5}{0}{180}
\psarc{-}(2.75,0){0.25}{0}{180}
\psline{-}(1.5,0)(1.5,0.75)
\psset{linecolor=black}
\psdots(0,0)(1,0)(1.5,0)(2.5,0)(3.0)
\psdots[dotstyle=+](0.5,0)(2,0)
\end{pspicture},\
\begin{pspicture}(-0.25,-0.1)(3.25,0.75)
\psset{linewidth=1pt}
\psset{linecolor=red}
\psarc{-}(1.25,0){0.25}{0}{180}
\psarc{-}(-0.25,0){0.25}{0}{90}
\psarc{-}(3.25,0){0.25}{90}{180}
\psline{-}(2.5,0)(2.5,0.75)
\psset{linecolor=black}
\psdots(0,0)(1,0)(1.5,0)(2.5,0)(3.0)
\psdots[dotstyle=+](0.5,0)(2,0)
\end{pspicture},\
\begin{pspicture}(-0.25,-0.1)(3.25,0.75)
\psset{linewidth=1pt}
\psset{linecolor=red}
\psarc{-}(0.5,0){0.5}{0}{180}
\psarc{-}(2,0){0.5}{0}{180}
\psline{-}(3,0)(3,0.75)
\psset{linecolor=black}
\psdots(0,0)(1,0)(1.5,0)(2.5,0)(3.0)
\psdots[dotstyle=+](0.5,0)(2,0)
\end{pspicture},\\
\nonumber
\psset{unit=0.6}
\begin{pspicture}(-0.25,-0.1)(3.25,0.75)
\psset{linewidth=1pt}
\psset{linecolor=red}
\psarc{-}(2,0){0.5}{0}{180}
\psbezier{-}(1,0)(1,1)(3,1)(3,0)
\psline{-}(0,0)(0,0.75)
\psset{linecolor=black}
\psdots(0,0)(1,0)(1.5,0)(2.5,0)(3,0)
\psdots[dotstyle=+](0.5,0)(2,0)
\end{pspicture},\
\begin{pspicture}(-0.25,-0.1)(3.25,0.75)
\psset{linewidth=1pt}
\psset{linecolor=red}
\psarc{-}(2.75,0){0.25}{0}{180}
\psbezier{-}(1.5,0)(1.5,1)(3.5,1)(3.5,0)
\psbezier{-}(-2,0)(-2,1)(0,1)(0,0)
\psline{-}(1,0)(1,0.75)
\psset{linecolor=black}
\psdots(0,0)(1,0)(1.5,0)(2.5,0)(3.0)
\psdots[dotstyle=+](0.5,0)(2,0)
\end{pspicture},\
\begin{pspicture}(-0.25,-0.1)(3.25,0.75)
\psset{linewidth=1pt}
\psset{linecolor=red}
\psarc{-}(-0.25,0){0.25}{0}{90}
\psarc{-}(3.25,0){0.25}{90}{180}
\psbezier{-}(-2,0)(-2,1)(1,1)(1,0)
\psbezier{-}(2.5,0)(2.5,1)(4.5,1)(4.5,0)
\psline{-}(1.5,0)(1.5,0.75)
\psset{linecolor=black}
\psdots(0,0)(1,0)(1.5,0)(2.5,0)(3.0)
\psdots[dotstyle=+](0.5,0)(2,0)
\end{pspicture},\
\begin{pspicture}(-0.25,-0.1)(3.25,0.75)
\psset{linewidth=1pt}
\psset{linecolor=red}
\psarc{-}(0.5,0){0.5}{0}{180}
\psbezier{-}(-.5,0)(-.5,1)(1.5,1)(1.5,0)
\psbezier{-}(3,0)(3,1)(5,1)(5,0)
\psline{-}(2.5,0)(2.5,0.75)
\psset{linecolor=black}
\psdots(0,0)(1,0)(1.5,0)(2.5,0)(3.0)
\psdots[dotstyle=+](0.5,0)(2,0)
\end{pspicture},\
\begin{pspicture}(-0.25,-0.1)(3.25,0.75)
\psset{linewidth=1pt}
\psset{linecolor=red}
\psarc{-}(1.25,0){0.25}{0}{180}
\psbezier{-}(0,0)(0,1)(2.5,1)(2.5,0)
\psline{-}(3,0)(3,0.75)
\psset{linecolor=black}
\psdots(0,0)(1,0)(1.5,0)(2.5,0)(3.0)
\psdots[dotstyle=+](0.5,0)(2,0)
\end{pspicture} \Big\}.
\end{eqnarray*}

It is easy to see that the cardinality of $\fB(\tS_{\infty/\fp}, r)$ is $\binom{d_{\fp}}{r}$. In fact, the map that associates to each element $\fa\in\fB(\tS_{\infty/\fp}, r)$ the set of right end points of arcs in $\fa$ establishes a bijection between $\fB(\tS_{\infty/\fp}, r)$ and the subsets with cardinality $r$ of the $d_{\fp}$-nodes in the band of $\tS_{\infty/\fp}$.

\subsection{Goren--Oort cycles and supersingular locus}

We fix a lifting $\tilde\tau\in\Sigma_{E,\infty/\fp}$ for each $\tau\in\Sigma_{\infty/\fp}-\tS_{\infty/\fp}$.

For a periodic semi-meander $\fa\in\fB(\tS_{\infty/\fp},r)$ with $r\geq 1$, we put
\begin{equation}\label{E:definition-Sa}
\tS_{\fa}\coloneqq\tS\cup\{\tau\in\Sigma_{\infty/\fp}\res \text{$\tau$ is an end point of some arc in $\fa$}\}.
\end{equation}
For an arc $\delta$ in $\fa$, we use $\tau^+_\delta$ and $\tau^-_\delta$ to denote its right and left end points respectively. We take
\[
\tilde\tS_{\fa,\infty}=\tilde\tS_{\infty}\cup \{\tilde\tau_\delta^+,\tilde\tau_\delta^{-,c}\res \delta \text{ is an arc of $\fa$}\}.
\]
Here, $\tilde\tau_\delta^+$ denotes the fixed lifting of $\tau_\delta^+$, and $\tilde\tau_\delta^{-,c}$ the conjugate of the fixed lifting $\tilde\tau_\delta^-$ of $\tau_\delta^-$. We fix an isomorphism $G'_{\tilde\tS_{\fa}}(\bA^{\infty})\cong G'_{\tilde\tS}(\bA^{\infty})$, and consider $K'^p$ as an open compact subgroup of $G'_{\tilde\tS_{\fa}}(\bA^{\infty,p})$. We may thus speak of the unitary Shimura variety $\bfSh(G'_{\tilde\tS_{\fa}},K'^p)$.

Following \cite{TX2}*{Section 3.7}, for every $\fa\in\fB(\tS_{\infty/\fp},r)$, we construct a closed subvariety $Z'_{\tilde\tS}(\fa)\subseteq\bfSh(G'_{\tilde\tS},K'^p)_{k_0}$ of codimension $r$, which is an $r$-th iterated $\dP^1$-fibration over $\bfSh(G'_{\tilde\tS_{\fa}}, K'^p)_{k_0}$. We make the induction on $r\geq 0$. When $r=0$, we put simply $Z'_{\tilde\tS}(\fa)\coloneqq\bfSh(G'_{\tilde\tS},K'^p)_{k_0}$. Assume now $r\geq1$. An arc in $\fa$ is called \emph{basic}, if it does not lie below any other arcs. Choose such a basic arc $\delta$, and put $\tau\coloneqq\tau^+_\delta$ and $\tau^-\coloneqq\tau_\delta^-$ for simplicity. We note that $\tau^-=\sigma^{-n_{\tau}}\tau$. Consider the Goren--Oort divisor $\bfSh(G'_{\tilde\tS},K'^p)_{k_0,\tau}$, and let $\pi'_{\tau}\colon\bfSh(G'_{\tilde\tS},K'^p)_{k_0,\tau}\to \bfSh(G'_{\tilde\tS_{\tau}},K'^p)_{k_0}$ be the $\dP^1$-fibration given by Proposition~\ref{P:GO-divisors}. Let $\fa_{\delta}\in\fB(\tS_{\fa, \infty/\fp}, r-1)$ be the periodic semi-meander for $\tS_{\fa}$ obtained from $\fa$ by replacing the nodes at $\tau,\tau^-$ with plus signs and removing the arc $\delta$.
For instance, if
\[
\fa=\psset{unit=0.6}
\begin{pspicture}(-0.25,-0.1)(3.25,0.75)
\psset{linewidth=1pt}
\psset{linecolor=red}
\psarc{-}(2,0){0.5}{0}{180}
\psbezier{-}(1,0)(1,1)(3,1)(3,0)
\psline{-}(0,0)(0,0.75)
\psset{linecolor=black}
\psdots(0,0)(1,0)(1.5,0)(2.5,0)(3,0)
\psdots[dotstyle=+](0.5,0)(2,0),
\end{pspicture}
\]
 then $\tS_{\fa}=\tS\cup\{\tau_2,\tau_3,\tau_5,\tau_6\}$, and the arc $\delta$ connecting $\tau_3$ and $\tau_5$ is the unique basic arc in $\fa$, and
\[
\fa_{\delta}=\psset{unit=0.6}
\begin{pspicture}(-0.25,-0.1)(3.25,0.75)
\psset{linewidth=1pt}
\psset{linecolor=red}
\psbezier{-}(1,0)(1,1)(3,1)(3,0)
\psline{-}(0,0)(0,0.75)
\psset{linecolor=black}
\psdots(0,0)(1,0)(3,0)
\psdots[dotstyle=+](1.5,0)(2.5,0)(0.5,0)(2,0)
\end{pspicture}.
\]
By induction hypothesis, we have constructed a closed subvariety $Z'_{\tilde\tS_{\tau}}(\fa_{\delta})\subseteq\bfSh(G'_{\tilde\tS_{\fa}},K'^p)_{k_0}$ of codimension $r-1$. Then we define $Z'_{\tilde\tS}(\fa)$ as the preimage of $Z'_{\tilde\tS_{\tau}}(\fa_{\delta})$ via $\pi'_{\tau}$. We denote by
\[
\pi'_{\fa}\colon Z'_{\tilde\tS}(\fa)\to \bfSh(G'_{\tilde\tS_{\fa}}, K'^p)_{k_0}
\]
the canonical projection. In summary, we have a diagram
\[
\xymatrix{Z'_{\tilde\tS}(\fa)\ar@{^(->}[r]\ar@/_2pc/[dd]_{\pi'_{\fa}}\ar[d] &\bfSh(G'_{\tilde\tS}, K'^p)_{k_0,\tau} \ar[d]^{\pi'_{\tau}}\ar@{^(->}[r] & \bfSh(G'_{\tilde\tS}, K'^p)_{k_0}\\
Z'_{\tilde\tS_{\tau}}(\fa_{\delta})\ar[d]^{\pi'_{\fa_{\delta}}}\ar@{^(->}[r]& \bfSh(G'_{\tilde\tS_{\tau}}, K'^p)_{k_0}\\
\bfSh(G'_{\tilde\tS_{\fa}}, K'^p)_{k_0}}
\]
where the square is cartesian. By the induction hypothesis, the morphism $\pi'_{\fa_{\delta}}$ is an $(r-1)$-th iterated $\dP^1$-fibration. It follows that $\pi'_{\fa}$ is an $r$-th iterated $\dP^1$-fibration.

We explain the relationship between Goren--Oort cycles and the $\fp$-supersingular locus of $\bfSh(G'_{\tilde\tS}, K'^p)_{k_0}$. Take $\fa\in\fB(\tS_{\infty/\fp},\lfloor d_{\fp}/2\rfloor)$. If $d_{\fp}$ is even, then we put $W'_{\tilde\tS}(\fa)\coloneqq Z'_{\tilde\tS}(\fa)$. If $d_{\fp}$ is odd, then we let $\tau(\fa)\in\Sigma_{\infty/\fp}$ denote the end point of the unique semi-line in $\fa$, and define $W'_{\tilde\tS}(\fa)$ by the following Cartesian diagram:
\[
\xymatrix{W'_{\tilde\tS}(\fa)\ar@{^(->}[r]\ar[d] &Z'_{\tilde\tS}(\fa)\ar[d] ^-{\pi'_{\fa}}\\
\bfSh(G'_{\tilde\tS_{\fa}},K'^p)_{k_0,\tau(\fa)}\ar@{^(->}[r] & \bfSh(G'_{\tilde\tS_{\fa}}, K'^p)_{k_0}.}
\]
We put
\begin{align}\label{E:tS-a}
\tilde\tS^*_{\fa}\coloneqq
\begin{cases}
\tilde\tS_{\fa}=(\tS_\fa,\tilde\tS_{\fa,\infty}) &\text{if  $d_{\fp}$ is even},\\
(\tS_{\fa}\cup\{\fp\},\tilde\tS_{\fa,\infty}\cup\{\tilde\tau(\fa)\}) &\text{if $d_{\fp}$ is odd}.
\end{cases}
\end{align}
Note that the underlying set $\tS^*_{\fa}$ of $\tilde\tS^*_{\fa}$ is independent of $\fa\in\fB(\tS_{\infty/\fp},\lfloor d_{\fp}/2\rfloor)$, namely all $\tS^*_\fa$ are equal to
\begin{equation}\label{E:S-a}
\tS(\fp)\coloneqq
\begin{cases}
\tS\cup \Sigma_{\infty/\fp} & \text{if $d_{\fp}$ is even,}\\
\tS\cup \Sigma_{\infty/\fp}\cup\{\fp\} &\text{if $d_{\fp}$ is odd}.
\end{cases}
\end{equation}
If $d_{\fp}$ is odd, then we have an isomorphism
\[
\bfSh(G'_{\tilde\tS_{\fa}},K'^p)_{k_0,\tau(\fa)}\cong\bfSh(G'_{\tilde\tS^*_{\fa}},K'^p)_{k_0}
\]
by Proposition~\ref{P:GO-divisors}. Thus, regardless of the parity of $d_{\fp}$, one has a $\lfloor d_{\fp}/2\rfloor$-th iterated $\dP^1$-fibration equivariant under prime-to-$p$ Hecke correspondences:
\[
\pi'_{\fa}\res_{W'_{\tilde\tS}(\fa)}\colon W'_{\tilde\tS}(\fa)\to \bfSh(G'_{\tilde\tS_{\fa}^*}, K'^p)_{k_0}.
\]

\begin{theorem}\label{T:p-supersingular-unitary}
Under the notation above, the union
\[
\bigcup_{\fa\in\fB(\tS_{\infty/\fp}, \lfloor d_{\fp}/2\rfloor)} W'_{\tilde\tS}(\fa)
\]
is exactly the $\fp$-supersingular locus of $\bfSh(G'_{\tilde\tS}, K'^p)_{k_0}$, that is, the maximal closed subset where the universal $\fp$-divisible group $\cA'_{\tilde\tS}[\fp^{\infty}]$ is supersingular.
\end{theorem}

\begin{proof}
We proceed by induction on $d_{\fp}\geq 0$. If $d_{\fp}=0$, then $\fB(\tS_{\infty/\fp}, 0)$ consists only of the trivial periodic semi-meander (that is, the one without any arcs or semi-lines). In this case, one has to show that the whole $\bfSh(G'_{\tilde\tS}, K'^p)_{k_0}$ is $\fp$-supersingular. First, we have $s_{\tilde\tau}\in\{0,2\}$ for all $\tilde\tau\in\Sigma_{E,\infty/\fp}$, and Assumption~\ref{A:assumption-S-tilde}(2) implies that the number of $\tilde\tau\in\Sigma_{E,\infty/\fp}$ with $s_{\tilde\tau}=2$ equals exactly to the number of $\tilde\tau\in\Sigma_{E,\infty/\fp}$ with $s_{\tilde\tau}=0$. Now consider a point $x=(A,\iota,\lambda,\alpha)\in\bfSh(G'_{\tilde\tS},K'^p)(\bF_p^\ac)$. Then, for every $\tilde\tau\in\Sigma_{E,\infty/\fp}$, the $2g_{\fp}$-th iterated essential Verschiebung
\[
V_{\es}^{2g_{\fp}}=\frac{V^{2g_{\fp}}}{p^{g_{\fp}}}\colon \tcD(A)^{\circ}_{\tilde\tau}\to\tcD(A)^{\circ}_{\sigma^{-2g_{\fp}}\tilde\tau}=\tcD(A)^{\circ}_{\tilde\tau}
\]
is bijective, no matter whether $\fp$ is split or inert in $E$. It follows immediately that $1/2$ is the only slope of the Dieudonn\'e module $\bigoplus_{\tilde\tau\in\Sigma_{E,\infty/\fp}}\tcD(A)_{\tilde\tau}=\tcD(A[\fp^{\infty}])$, so that $A[\fp^{\infty}]$ is supersingular.

Assume now $d_{\fp}\geq 1$. We prove first that the union $\bigcup_{\fa\in\fB(\tS_{\infty/\fp}, \lfloor d_{\fp}/2\rfloor)}W'_{\tilde\tS}(\fa)$ is contained in the $\fp$-supersingular locus of $\bfSh(G'_{\tilde\tS}, K'^p)_{k_0}$. Fix $\fa\in\fB(\tS_{\infty/\fp},\lfloor d_{\fp}/2\rfloor)$. Then one has a projection
\[
\pi'_{\fa}\res_{W'_{\tilde\tS}(\fa)}\colon W'_{\tilde\tS}(\fa)\to \bfSh(G_{\tilde\tS_{\fa}}, K'^p)_{k_0}
\]
and a $p$-quasi-isogeny
\[
\phi_{\fa}\colon \cA'_{\tilde\tS}\res_{W'_{\tilde\tS}(\fa)}\to \pi'^{*}_{\fa}\cA'_{\tilde\tS_{\fa}}
\]
by the construction of $\pi'_{\fa}$ and Proposition~\ref{P:GO-divisors}. Note that $d_{\fp}(\tS_{\fa})=0$, and by the discussion above, $\cA'_{\tilde\tS_{\fa}}[\fp^{\infty}]$ is supersingular over the whole $\bfSh(G_{\tilde\tS_{\fa}}, K'^p)_{k_0}$. It follows that $\cA'_{\tilde\tS}[\fp^{\infty}]$ is supersingular over $W_{\tilde\tS}(\fa)$.

To complete the proof, it remains to show that if $x\in\bfSh(G'_{\tilde\tS}, K'^p)(\bF_p^\ac)$ is a $\fp$-supersingular point, then $x\in W'_{\tilde\tS}(\fa)(\bF_p^\ac)$ for some $\fa\in\fB(\tS_{\infty/\fp},\lfloor d_{\fp}/2\rfloor)$. By Proposition~\ref{P:ordinary}, there exists $\tau\in\Sigma_{\infty/\fp}$ such that $x\in\bfSh(G'_{\tilde\tS},K'^p)_{k_0,\tau}(\bF_p^\ac)$. Consider the $\dP^1$-fibration $\pi'_{\tau}\colon\bfSh(G'_{\tilde\tS}, K'^p)_{k_0,\tau}\to\bfSh(G'_{\tilde\tS_{\tau}}, K'^p)_{k_0}$. Since $\cA'_{\tilde\tS,x}$ is $p$-quasi-isogenous to $\cA'_{\tilde\tS_{\tau},\pi'_{\tau}(x)}$, we see that $\pi'_{\tau}(x)$ lies in the $\fp$-supersingular locus of $\bfSh(G'_{\tilde\tS_{\tau}},K'^p)_{k_0}$. By the induction hypothesis, $\pi'_{\tau}(x)\in W'_{\tilde\tS_{\tau}}(\fb)(\bF_p^\ac)$ for some periodic semi-meander $\fb\in\fB(\tS_{\tau, \infty/\fp},\lfloor d_{\fp}/2-1\rfloor)$. Now let $\fa$ be the periodic semi-meander obtained from $\fb$ by adjoining an arc $\delta$ connecting $\sigma^{-n_{\tau}}\tau$ and $\tau$ so that $\tau$ is the right end point of $\delta$. Then $\fa\in\fB(\tS_{\infty/\fp},\lfloor d_{\fp}/2\rfloor)$, and $\delta$ is a basic arc of $\fa$ such that $\fb=\fa_{\delta}$. To finish the proof, it suffices to note that $W'_{\tilde\tS}(\fa)=\pi'^{-1}_{\tau}(W'_{\tilde\tS_{\tau}}(\fb))$ by definition.
\end{proof}

\begin{definition}\label{D:superspecial-unitary}
We put
\[
\bfSh(G'_{\tilde\tS},K'^p)^{\fp-\r{sp}}_{k_0}\coloneqq \bfSh(G'_{\tilde\tS}, K'^p)_{k_0, \Sigma_{\infty/\fp}},
\]
and call it the \emph{$\fp$-superspecial locus} of $\bfSh(G'_{\tilde\tS}, K'^p)_{k_0}$. 
\end{definition}

We have the proposition characterizing the $\fp$-superspecial locus.

\begin{proposition}\label{P:p-superspecial-unitary}
Let $\fp\in\Sigma_p$ be such that $d_{\fp}$ is odd, and take $\fa\in\fB(\tS_{\infty/\fp},(d_{\fp}-1)/2)$. Then $\bfSh(G'_{\tilde\tS},K'^p)^{\fp-\r{sp}}_{k_0}$ is contained in $W'_{\tilde\tS}(\fa)$, and the restriction of $\pi'_{\fa}$ to $\bfSh(G'_{\tilde\tS},K'^p)_{k_0}^{\fp-\r{sp}}$ induces an isomorphism
\[
\bfSh(G'_{\tilde\tS},K'^p)_{k_0}^{\fp-\r{sp}}\xra{\sim}\bfSh(G'_{\tilde\tS_{\fa}^*},K'^p)_{k_0},
\]
which is equivariant under prime-to-$p$ Hecke correspondences.
\end{proposition}

\begin{proof}
We proceed by induction on $d_{\fp}\geq 1$. If $d_{\fp}=1$, then all the $\fp$-supersingular locus is $\fp$-superspecial, and the $\fp$-supersingular locus consists of only one stratum $W'_{\tilde\tS}(\fa)$. So the statement is clear.

Assume now $d_{\fp}>1$. Choose a basic arc $\delta $ of $\fa$. Let $\tau$ (resp. $\tau^{-}$) be the right (resp. left) node of $\delta$, and $\fa_{\delta}$ be the semi-meander obtained from $\fa$ by removing the arc $\delta$. Then one has a commutative diagram
\begin{equation*}
\xymatrix{
W'_{\tilde\tS}(\fa)\ar[r]\ar[d]& Z'_{\tilde\tS}(\fa)\ar[r]\ar[d] & \bfSh(G'_{\tilde\tS}, K'^p)_{k_0, \tau}\ar[d]^{\pi'_{\tau}}\\
W'_{\tilde\tS_{\tau}}(\fa_{\delta}) \ar[d]\ar[r] & Z'_{\tilde\tS_{\tau}}(\fa_{\delta})\ar[d]^{\pi'_{\fa_{\delta}}}\ar[r] & \bfSh(G'_{\tilde{\tS}_{\tau}}, K'^p)_{k_0}\\
\bfSh(G'_{\tilde{\tS}_{\fa}}, K'^p)_{k_0, \tau(\fa)}\ar[r]\ar[d]^{\cong} & \bfSh(G'_{\tilde\tS_{\fa}}, K'^p)_{k_0}\\
\bfSh(G'_{\tilde{\tS}^*_{\fa}}, K'^p)_{k_0}
}
\end{equation*}
where all the squares are cartesian; all horizontal maps are closed immersions; and all vertical arrows are iterated $\dP^1$-bundles. By the induction hypothesis, the $\fp$-superspecial locus $\bfSh(G'_{\tilde\tS_{\tau}},K'^p)_{k_0}^{\fp-\r{sp}}$ is contained in $W'_{\tilde\tS_{\tau}}(\fa_{\delta})$ and the restriction of $\pi'_{\fa_{\delta}}$ induces an isomorphism
\begin{equation}\label{E:isom-superspecial}
\bfSh(G'_{\tilde\tS_{\tau}},K'^p)_{k_0}^{\fp-\r{sp}}\xra{\sim}\bfSh(G'_{\tilde\tS^*_{\fa}},K'^p)_{k_0}.
\end{equation}
Now by Proposition~\ref{P:GO-divisors}, the restriction of $\pi'_{\tau}$ induces an isomorphism
\[
\bfSh(G_{\tilde\tS},K'^p)_{k_0,\{\tau,\tau^-\}}\xra{\sim}\bfSh(G'_{\tilde\tS_{\tau}},K'^p)_{k_0}
\]
compatible with the construction of Goren--Oort divisors. Thus, $\pi'_{\tau}$ sends $\bfSh(G'_{\tilde\tS},K'^p)_{k_0}^{\fp-\r{sp}}$ isomorphically to $\bfSh(G'_{\tilde\tS_{\tau}},K'^p)_{k_0}^{\fp-\r{sp}}$. The statement now follows immediately by composing with the isomorphism \eqref{E:isom-superspecial}.
\end{proof}

\subsection{Total supersingular and superspecial loci}
We study now the total supersingular locus of $\bfSh(G'_{\tilde\tS}, K'^p)_{k_0}$, that is, the maximal closed subset where the universal $p$-divisible group $\cA'_{\tilde\tS}[p^{\infty}]$ is supersingular. Put
\[
\fB_{\tS}\coloneqq\{\underline\fa=(\fa_{\fp})_{\fp\in\Sigma_p}\res \fa_{\fp}\in\fB(\tS_{\infty/\fp}, \lfloor d_{\fp}/2\rfloor)\},
\]
and $r=\sum_{\fp\in\Sigma_p}\lfloor d_{\fp}/2\rfloor$. We attach to each $\underline\fa$ an $r$-dimensional closed subvariety $W'_{\tilde\tS}(\underline{\fa})\subseteq \bfSh_{K'}(G'_{\tilde\tS})_{k_0}$ as follows. We write $\Sigma_{p}=\{\fp_1,\dots,\fp_m\}$, that is, we choose an order for the elements of $\Sigma_p$. We put $\tS_1\coloneqq\tS_{\fa_{\fp_1}}$ and $\tilde\tS^*_1\coloneqq\tilde\tS_{\fa_{\fp_1}}^*$ (see \eqref{E:tS-a}); put inductively $\tS_{i+1}\coloneqq(\tS_i)_{\fa_{\fp_{i+1}}}$, $\tilde\tS^*_{i+1}=\tilde{(\tS_i)}{}_{\fa_{\fp_{i+1}}}^*$ for $1\leq i\leq m-1$; and finally put $\tS_{\underline\fa}\coloneqq\tS_m$ and $\tilde\tS_{\underline\fa}^*\coloneqq\tilde\tS^*_m$. For $\fa_{\fp_1}\in\fB(\tS_{\infty/\fp}, \lfloor d_{\fp_1}/2\rfloor)$, we have constructed a $\lfloor d_{\fp_1}/2\rfloor$-th iterated $\dP^1$-fibration
\[
\pi'_{\fa_{\fp_1}}\res_{W'_{\tilde\tS}(\fa_{\fp_1})}\colon W'_{\tilde\tS}(\fa_{\fp_1})\to\bfSh(G'_{\tilde\tS_1^*}, K'^p)_{k_0}.
\]
Now, applying the construction to $\fa_{\fp_2}\in\fB(\tS_{\infty/\fp_2},\lfloor d_{\fp_2}/2\rfloor) $ and $\bfSh(G'_{\tilde\tS_1^*},K'^p)_{k_0}$, we have a closed subvariety $W'_{\tilde\tS_1^*}(\fa_{\fp_2})\subseteq\bfSh(G'_{\tilde\tS_1^*},K'^p)_{k_0}$ of codimension $\lceil d_{\fp_2}/2\rceil$. We put
\[
W'_{\tilde\tS}(\fa_{\fp_1},\fa_{\fp_2})\coloneqq(\pi'_{\fa_{\fp_1}})^{-1}(W'_{\tilde\tS_1^*}(\fa_{\fp_2})).
\]
Then there exists a canonical projection
\[
\pi'_{\fa_{\fp_1},\fa_{\fp_2}}\colon W'_{\tilde\tS}(\fa_{\fp_1},\fa_{\fp_2})\xra{\pi'_{\fa_{\fp_1}}\res_{W'_{\tilde\tS}(\fa_{\fp_1},\fa_{\fp_2})}} W'_{\tilde\tS_1^*}(\fa_{\fp_2})\xra{\pi'_{\fa_{\fp_2}}\res_{W'_{\tilde\tS_1^*}(\fa_{\fp_2})}} \bfSh(G'_{\tilde\tS_2^*},K'^p)_{k_0}.
\]
Repeating this construction, we finally get a closed subvariety $W'_{\tilde\tS}(\underline\fa)\subseteq \bfSh(G'_{\tilde\tS},K'^p)_{k_0}$ of codimension $\sum_{\fp\in\Sigma}\lceil d_{\fp}/2\rceil$ together with a canonical projection
\[
\pi'_{\underline\fa}\colon W'_{\tilde\tS}(\underline\fa)\to \bfSh(G'_{\tilde\tS_{\underline\fa}^*},K'^p)_{k_0}.
\]
Note that the underlying set $\tS_{\underline\fa}^*$ of $\tilde\tS_{\underline\fa}^*$ is independent of $\underline\fa\in\fB_{\tS}$, namely all of them are equal to
\begin{equation}\label{E:S-max}
\tS_{\max}\coloneqq\Sigma_{\infty}\cup\{\fp\in\Sigma_p\res g_{\fp}\coloneqq[F_{\fp}:\QQ_p]\text{ is odd}\}.
\end{equation}
Thus $\bfSh(G'_{\tilde\tS_{\underline\fa}^*},K'^p)_{k_0}$ is a Shimura variety of dimension $0$, and $\pi'_{\underline\fa}$ is by construction an $r$-th iterated $\dP^1$-fibration over $\bfSh(G'_{\tilde\tS^*_{\underline\fa}},K'^p)_{k_0}$. We note that $W'_{\tilde\tS}(\underline\fa)$ does not depend on the order $\fp_1,\dots, \fp_m$ of the places of $F$ above $p$.

\begin{theorem}\label{T:supersingular-unitary}
The total supersingular locus of $\bfSh(G'_{\tilde\tS}, K'^p)_{k_0}$ is given by
\[
\bfSh(G'_{\tilde\tS},K'^p)_{k_0}^{\mathrm{ss}}\coloneqq\bigcup_{\underline\fa\in\fB_{\tS}} W'_{\tilde\tS}(\underline\fa),
\]
where each $W'_{\tilde\tS}(\underline\fa)$ is a $\sum_{\fp\in\Sigma_p}\lfloor d_{\fp}/2\rfloor$-th iterated $\dP^1$-fibration over some discrete Shimura variety $\bfSh(G'_{\tilde\tS^*_{\underline\fa}},K'^p)_{k_0}$. In particular, $\bfSh(G'_{\tilde\tS},K'^p)_{k_0}^{\mathrm{ss}}$ is proper and of equidimension $\sum_{\fp\in\Sigma_p}\lfloor d_{\fp}/2\rfloor$.
\end{theorem}

\begin{proof}
This follows immediately from Theorem~\ref{T:p-supersingular-unitary} by induction on the number of $p$-adic places $\fp\in\Sigma_p$ such that $d_{\fp}\neq 0$.
\end{proof}

\begin{remark}
It is clear that the total supersingular locus is the intersection of all $\fp$-supersingular loci for $\fp\in\Sigma_p$. It follows that
\[
W'_{\tilde\tS}(\underline\fa)=\bigcap_{\fp\in\Sigma_p}W'_{\tilde\tS}(\fa_{\fp}),
\]
and the intersection is transversal.
\end{remark}

Similarly to Definition~\ref{D:superspecial-unitary}, we can define the total superspecial locus of $\bfSh(G'_{\tilde{\tS}}, K'^p)_{k_0}$ as
\[
\bfSh(G'_{\tilde{\tS}}, K'^p)_{k_0}^{\r{sp}}\coloneqq \bfSh(G'_{\tilde{\tS}}, K'^p)_{k_0, \Sigma_{\infty}}.
\]
We have the following analogue of Proposition~\ref{P:p-superspecial-unitary}.
\begin{proposition}\label{P:superspecial-unitary}
For each $\underline{\fa}\in \fB_{\tS}$, $W'_{\tilde{\tS}}(\underline{\fa})$ contains $\bfSh(G'_{\tilde{\tS}}, K'^p)_{k_0}^{\r{sp}}$, and each geometric irreducible component of $W'_{\tilde{\tS}}(\underline \fa)$ contains exactly one point of $\bfSh(G'_{\tilde{\tS}}, K'^p)_{k_0}^{\r{sp}}$. In other words,  the restriction of $\pi'_{\underline \fa}$ induces an isomorphism
\[
\bfSh(G'_{\tilde{\tS}}, K'^p)_{k_0}^{\r{sp}}\xra{\sim} \bfSh(G'_{\tilde\tS^*_{\underline\fa}},K'^p)_{k_0}.
\]
\end{proposition}
\begin{proof}
	This follows immediately from Proposition~\ref{P:p-superspecial-unitary}.
	\end{proof}

\subsection{Applications to quaternionic Shimura varieties}
\label{S:application}


Denote by $\bfSh(G_{\tS,\tT}, K^p)$ the integral model of $\Sh(G_{\tS,\tT},K^p)$ over $\cO_{F_{\tS,\tT}, \wp}$ induced by $\bfSh(G'_{\tilde\tS}, K'^p)$. We assume that the residue field of $\cO_{F_{\tS,\tT},\wp}$ is contained in $k_0$ (e.g. $\tS=\tT=\emptyset$), and  put  $\bfSh(G_{\tS,\tT}, K^p)_{k_0}\coloneqq\bfSh(G_{\tS,\tT},K^p)\otimes_{\cO_{F_{\tS,\tT},\wp}}k_0$. As in \cites{TX1,TX2}, the construction of Goren--Oort divisors can be transferred to $\bfSh(G_{\tS,\tT},K^p)_{k_0}$ for a sufficiently small open compact subgroup $K^p\subseteq G_{\tS}(\bA^{\infty,p})$.


Consider first the connected Shimura variety $\bfSh(G_{\tS})^{\circ}_{\FF_p^{\ac}}\coloneqq\bfSh(G_{\tS})^{\circ}_{\ZZ_p^{\ur}}\otimes_{\ZZ_p^{\ur}}\FF_p^{\ac}$. For each $\tau\in\Sigma_{\infty}$, the Goren--Oort divisor $\bfSh(G'_{\tilde\tS})_{k_0,\tau}=\varprojlim_{K'^p}\bfSh(G'_{\tilde\tS},K'^p)_{k_0,\tau}$ induces a divisor $\bfSh(G'_{\tilde\tS})_{\FF_p^{\ac},\tau}^{\circ}$ on $\bfSh(G'_{\tilde\tS})^{\circ}_{\FF_p^{\ac}}$. By the canonical isomorphism
\[
\bfSh(G_{\tS,\tT})^{\circ}_{\FF_p^{\ac}}\cong \bfSh(G'_{\tilde\tS})^{\circ}_{\FF_p^{\ac}}
\]
from Section \ref{S:comparison-moduli} and Deligne's recipe of recovering $\bfSh(G_{\tS,\tT})_{\FF^{\ac}_p}$ from $\bfSh(G_{\tS,\tT})^{\circ}_{\FF^{\ac}_p}$ \cite{TX1}*{Corollary~2.13}, the divisor $\bfSh(G_{\tS,\tT})_{\FF_p^{\ac},\tau}^{\circ}$ induces a divisor $\bfSh(G_{\tS,\tT})_{\FF_p^{\ac},\tau}$ on $\bfSh(G_{\tS,\tT})_{\FF_p^{\ac}}$. By Galois descent, one gets a divisor $\bfSh(G_{\tS,\tT})_{k_0,\tau}$ on $\bfSh(G_{\tS,\tT})_{k_0}$, which is stable under prime-to-$p$ Hecke action. Finally, we define the Goren--Oort divisors on $\bfSh(G_{\tS,\tT},K^p)_{k_0}$ as the image of Goren--Oort divisors on $\bfSh(G_{\tS,\tT},K^p)_{k_0}$ via the natural projection $\bfSh(G_{\tS,\tT})_{k_0}\to \bfSh(G_{\tS,\tT},K^p)_{k_0}$.


\begin{proposition}\label{P:GO-divisor-quaternion}
Take $\tau\in\Sigma_{\infty/\fp}$ for some $\fp\in\Sigma_p$, and put $\tT_{\tau}\coloneqq\tT\cup \{\tau\}$. There exists a morphism of $\FF_p^{\ac}$-algebraic varieties
\[
\pi_{\tau}\colon \bfSh(G_{\tS, \tT},K^p)_{\FF_p^{\ac},\tau}\to \bfSh(G_{\tS_{\tau}, \tT_{\tau}},K^p)_{\FF_p^{\ac}},
\]
where $\tS_{\tau}$ was defined in \eqref{D:S-tau}, such that
\begin{enumerate}
 	\item it is an isomorphism if $\Sigma_{\infty/\fp}=\tS_{\infty/\fp}\cup \{\tau\}$;
 	
 	\item and it is a $\dP^1$-fibration that descends to  a morphism  $\bfSh(G_{\tS, \tT},K^p)_{k_0,\tau}\to \bfSh(G_{\tS_{\tau}, \tT_{\tau}},K^p)_{k_0}$ if
 $\Sigma_{\infty/\fp}\neq \tS_{\infty/\fp}\cup \{\tau\}$.
\end{enumerate}
 \end{proposition}

\begin{proof}
This follows immediately from Proposition~\ref{P:GO-divisors}.
\end{proof}

Now, the construction of Goren--Oort cycles can be transferred to the quaternionic Shimura variety $\bfSh(G_{\tS, \tT},K^p)_{k_0}$. For a periodic semi-meander $\fa\in\fB(\tS_{\infty/\fp},\lfloor d_{\fp}/2\rfloor)$, we construct inductively in the same way as $Z'_{\tilde\tS}(\fa)$  a closed $k_0$-subvariety $Z_{\tS,\tT}(\fa)\subseteq \bfSh(G_{\tS, \tT},K^p)_{k_0}$ such that there exists  a $\lfloor d_{\fp}/2\rfloor$-th iterated $\dP^1$-fibration
\[\pi_{\fa}\colon Z_{\tS,\tT}(\fa)\to \bfSh(G_{\tS_{\fa},\tT_{\fa}}, K^p)_{k_0}\]
according to Proposition~\ref{P:GO-divisor-quaternion}(2),
where $\tS_{\fa}$ is defined in \eqref{E:definition-Sa} and
\[\tT_{\fa}=\tT\cup\{\tau\in \Sigma_{\infty}|\text{$\tau$ is the right end point of an arc in $\fa$}.\}\]
We define similarly
\begin{equation}\label{E:defn-W-a}
W_{\tS,\tT}(\fa)
=\begin{cases}Z_{\tS,\tT}(\fa) &\text{if $d_{\fp}$ is even,}\\
\pi_{\fa}^{-1}(\bfSh(G_{\tS_{\fa},\tT_{\fa}}, K^p)_{k_0, \tau(\fa)}) &\text{if $d_{\fp}$ is odd,}
\end{cases}
\end{equation}
where $\tau(\fa)\in \Sigma_{\infty/\fp}$ is the end point of the unique semi-line of $\fa$. Then $\pi_{\fa}$ induces a $\lfloor d_{\fp}/2\rfloor$-th iterated $\dP^1$-fibration
\begin{equation*}
\pi_{\fa}\res_{W_{\tS,\tT}(\fa)_{\FF_p^{\ac}}}\colon W_{\tS,\tT}(\fa)_{\FF_p^{\ac}}\to \bfSh(G_{\tS(\fp),\tT_{\fa}^*}, K^p)_{\FF_p^{\ac}}
\end{equation*}
where $\tS(\fp)=\tS^*_{\fa}$ is defined in \eqref{E:S-a}, and
\[
\tT_{\fa}^*=
\begin{cases}
\tT_{\fa} &\text{if $d_{\fp}$ is even,}\\
\tT_{\fa}\cup\{\tau(a)\} &\text{if $d_{\fa}$ is odd}.
\end{cases}
\]
Of course, when $d_{\fp}$ is even, the morphism $\pi_{\fa}\res_{W_{\tS,\tT}(\fa)_{\FF_p^{\ac}}}$ is simply the base change to $\FF_p^{\ac}$ of $\pi_{\fa}$.

Similarly, for $\underline\fa=(\fa_{\fp})_{\fp\in\Sigma_p}\in\fB_{\tS}=\prod_{\fp\in \Sigma_p}\fB(\tS_{\infty/\fp}, \lfloor d_{\fp}/2\rfloor)$, we can define a closed subvariety $W_{\tS,\tT}(\underline\fa)\subseteq\bfSh(G_{\tS,\tT},K^p)_{k_0}$ of dimension $r=\sum_{\fp\in\Sigma_p}\lfloor d_{\fp}/2\rfloor$ together with an $r$-th iterated $\dP^1$-fibration
\[
\pi_{\underline\fa}\colon W_{\tS,\tT}(\underline\fa)_{\FF_p^{\ac}}\to \bfSh(G_{\tS_{\max}, \tT_{\underline\fa}^*}, K^p)_{\FF_p^{\ac}},
\]
where $\tS_{\max}$ was defined in \eqref{E:S-max}. Note that one has
\begin{align*}
\bfSh(G_{\tS_{\max}, \tT_{\underline\fa}^*}, K^p)(\FF_p^{\ac})\cong
B_{\tS_{\max}}^{\times}\backslash \widehat{B}^{\times}_{\tS_{\max}}/K^p\prod_{\fp\in\Sigma_p}K_{\fp}
\end{align*}
where $K_{\fp}\cong \GL_2(\cO_{F_{\fp}})$ if $[F_{\fp}:\QQ_p]$ is even, and $K_{\fp}$ is the unique maximal open compact subgroup of $(B_{\tS_{\max}}\otimes_FF_{\fp})^{\times}$ if $[F_{\fp}:\QQ_p]$ is odd.
 We will denote thus the target of $\pi_{\underline{\fa}}$ uniformly by $\bfSh(G_{\tS_{\max}}, K^p)_{\FF_p^{\ac}}$ for all $\underline\fa\in \fB_{\tS}$. In particular, the set of geometric irreducible components of $W_{\tS,\tT}(\fa)$ is in bijection with $\bfSh(G_{\tS_{\max}}, K^p)(\FF_p^{\ac})$.



\subsection{Totally indefinite quaternionic Shimura varieties}

We consider the case $\tS=\emptyset$ (and hence $\tT=\emptyset$), and we write $G=G_{\emptyset}$ and $G'=G'_{\tilde\emptyset}$ for simplicity as usual. Recall that $\bfSh(G,K^p)$ classifies tuples $(A,\iota,\bar\lambda,\bar\alpha_{K^p})$ as defined in Section \ref{S:quaternion-PEL}. Even though it is only a coarse moduli space, there still exists a universal abelian scheme $\cA$ over $\bfSh(G,K^p)$ (See Remark~\ref{R:universal-AV}(1)).

\begin{definition}
Put $\bfSh(G,K^p)_{\FF_p}\coloneqq\bfSh(G,K^p)\otimes{\FF_p}$.
\begin{enumerate}
  \item For each $\fp\in\Sigma_p$, we define the \emph{$\fp$-supersingular locus} of $\bfSh(G,K^p)_{\FF_p}$ as the maximal closed subvariety of $\bfSh(G,K^p)_{\FF_p}$ where the universal $\fp$-divisible group $\cA[\fp^{\infty}]$ is supersingular.

  \item We define the \emph{total supersingular locus} of $\bfSh(G,K^p)_{\FF_p}$ as the intersection of the $\fp$-supersingular locus for all $\fp\in\Sigma_p$.
\end{enumerate}
\end{definition}

 \if false

Let $G^{\star}\subseteq G$ be the subgroup considered in \eqref{S:quaternion-PEL}. The $\fp$-supersingular locus of $\bfSh(G,K^p)_{\FF_p}$ can be well defined. By the canonical isomorphisms of connected Shimura varieties
\[
\bfSh(G^{\star})^{\circ}_{\FF_p^{\ac}}\cong \bfSh(G)^{\circ}_{\FF_p^{\ac}}\cong \bfSh(G')^{\circ}_{\FF_p^{\ac}},
\]
we have two families of abelian varieties over $\bfSh(G)^{\circ}_{\FF^{\ac}_p}$, namely $\cA\res_{\bfSh(G)^{\circ}_{\FF_p^{\ac}}}$ and  $\cA'\res_{\bfSh(G)^{\circ}_{\FF_p^{\ac}}}$ induced by the universal families over $\bfSh(G^{\star})$ and $\bfSh(G')$ respectively. They are related by the relation \eqref{E:comparison-univ-av}. In particular, the $\fp$-supersingular loci for these two families of abelian schemes are the  same; we define it thus as the supersingular locus  of $\bfSh(G)_{\FF_p^{\ac}}^{\circ}$. By \cite{TX1}*{Corollary~2.13}, the $\fp$-supersingular locus of $\bfSh(G)_{\FF_p^{\ac}}^{\circ}$ induces a closed subvariety of $\bfSh(G)_{\FF_p^{\ac}}$, which we define to be the $\fp$-supersingular locus of $\bfSh(G)_{\FF_p^{\ac}}$. By Galois descent, $\fp$-supersingular locus of $\bfSh(G)_{\FF_p^{\ac}}$ descends to a closed subvariety of $\bfSh(G)$ defined over  $\FF_p$ by Galois descent. For any open compact subgroup $K^p\subseteq G(\bA^{\infty,p})$, we define the $\fp$-supersingular locus of $\bfSh(G,K^p)_{\FF_p}$ as the image of  the $\fp$-supersingular locus of $\bfSh(G)_{\FF_p}$ via the natural projection $\bfSh(G)_{\FF_p}\to \bfSh(G,K^p)_{\FF_p}$. Then similarly to Theorem~\ref{T:p-supersingular-unitary}, we have the following

\fi

\begin{theorem}\label{T:p-supersingular-quaternion}
For $\fp\in\Sigma_p$, put $g_{\fp}\coloneqq[F_{\fp}:\QQ_p]$. Then the $\fp$-supersingular locus of $\bfSh(G, K^p)_{\FF_p}$, after base change to $k_0$, is
\[
\bigcup_{\fa\in\fB(\emptyset_{\infty/\fp},\lfloor g_{\fp}/2\rfloor)} W_{\emptyset}(\fa),
\]
where $\fB(\emptyset_{\infty/\fp},\lfloor g_{\fp}/2\rfloor)$ is the set of periodic semi-meanders of $g_{\fp}$-nodes and $\lfloor g_{\fp}/2\rfloor$-arcs, and each $W_{\emptyset}(\fa)$ is defined in \eqref{E:defn-W-a} and $W_{\emptyset}(\fa)_{\FF_p^{\ac}}$ is a $\lfloor g_{\fp}/2\rfloor$-th iterated $\dP^1$-fibration over $\bfSh(G_{\emptyset(\fp)}, K^p)_{\FF_p^{\ac}}$.
\end{theorem}

\begin{proof}
According to the discussion of Section \ref{S:comparison-moduli}, the definition of the $\fp$-supersingular locus of $\bfSh(G,K^p)_{\FF_p}$ using the universal family $\cA$ coincides with the one induced the $\fp$-supersingular locus of the unitary Shimura variety $\bfSh(G',K'^p)_{\FF_p}$. The statement then follows from  Theorem~\ref{T:p-supersingular-unitary}.
\end{proof}

\begin{theorem}\label{T:supersingular-quaternion}
Denote by $\bfSh(G,K^p)_{\FF_p}^{\mathrm{ss}}$ the total supersingular locus of $\bfSh(G,K^p)_{\FF_p}$. Then we have
\[
\bfSh(G,K^p)_{\FF_p}^{\mathrm{ss}}\otimes k_0=\bigcup_{\underline\fa\in\fB_{\emptyset}}W_{\emptyset}(\underline\fa),
\]
where $\fB_{\emptyset}$ is the set of tuples $(\fa_{\fp})_{\fp\in\Sigma_{p}}$ with $\fa_{\fp}\in\fB(\emptyset_{\infty/\fp},\lfloor g_{\fp}/2\rfloor)$.
 The base change $W_{\emptyset}(\fa)_{\FF_p^{\ac}}$ of $W_{\emptyset}(\underline\fa)$ to $\FF_p^{\ac}$ is a $(\sum_{\fp\in\Sigma_p}\lfloor g_{\fp}/2\rfloor)$-th iterated $\dP^1$-fibration over $\bfSh(G_{\tS_{\max}},K^p)_{\FF_p^{\ac}}$, equivariant under prime-to-$p$ Hecke correspondences, where $\tS_{\max}$ was defined in \eqref{E:S-max}. In particular, $\bfSh(G,K^p)_{\FF_p}^{\mathrm{ss}}$ is proper and of equidimension $\sum_{\fp\in\Sigma_p}\lfloor g_{\fp}/2\rfloor$.
\end{theorem}

\begin{proof}
This follows from Theorem~\ref{T:p-supersingular-quaternion} by induction on the number of $p$-adic places $\fp\in \Sigma_p$.
\end{proof}

\begin{remark}
The above theorem is known in the following cases.
\begin{enumerate}
  \item If $p$ is inert in $F$ of degree $2$ and $B$ is the matrix algebra, then the theorem was first proved in \cite{BG99}.

  \item If $p$ is inert in $F$ of degree $4$ and $B$ is the matrix algebra, then the results was due to \cite{Yu03}.

  \item Assume that $p$ is inert in $F$ of even degree. Then the strata $W_{\emptyset}(\underline\fa)$ have already been constructed in \cite{TX2}, and the authors proved there that, under certain genericity conditions on the Satake parameters of a fixed automorphic cuspidal representation $\pi$, the cycles $W_{\emptyset}(\underline\fa)$ give all the $\pi$-isotypic Tate cycles on the quaternionic Shimura variety $\bfSh(G,K^p)_{\FF_p}$.
\end{enumerate}
\end{remark}

We define an action of $\rG_{\FF_p}=\Gal(\FF^\ac_p/\FF_p)$ on the set $\fB_{\emptyset}$ as follows. For each periodic semi-meander $\fa_{\fp}\in\fB(\emptyset_{\infty/\fp},\lfloor g_{\fp}/2\rfloor)$, let $\sigma(\fa_{\fp})$ be the Frobenius translate of $\fa_{\fp}$, that is, there is an arc in $\sigma(\fa_{\fp})$ linking two nodes $x,y$ if and only if there is an arc in $\fa_{\fp}$ linking $\sigma^{-1}(x),\sigma^{-1}(y)$. For $\underline\fa=(\fa_{\fp})_{\fp}$, we put $\sigma(\underline\fa)\coloneqq(\sigma(\fa_{\fp}))_{\fp\in \Sigma_{p}}$. It is clear that the subgroup $\Gal(\FF_p^{\ac}/k_0)$ of $\Gal(\FF_p^{\ac}/\FF_p)$ stabilizes each $\underline {\fa}\in \fB_{\emptyset}$. Then the action of $\Gal({\FF}^{\ac}_p/\FF_p)$ on $\bfSh(G,K^p)^{\r{ss}}_{\FF_p^\ac}$ sends the stratum $W_{\emptyset}(\underline\fa)$ to $W_{\emptyset}(\sigma(\underline \fa))$.

\begin{definition}\label{D:superspecial}
We define the \emph{superspecial locus} $\bfSh(G,K^p)^{\r{sp}}_{\FF_p}$ of $\bfSh(G,K^p)_{\bF_p}$ to be the maximal closed subset where the universal $p$-divisible group $\cA[p^\infty]$ is superspecial.
\end{definition}

Using the universal family of abelian varieties $\cA$ over $\bfSh(G,K^p)$, one can define, for each $\tau\in\Sigma_{\infty}$, a partial Hasse invariant $h_{\tau}$ on $\bfSh(G,K^p)_{k_0}$ similarly to \eqref{E:Hasse-invariant}. We can also define the Goren--Oort divisor $\bfSh(G,K^p)_{k_0,\tau}$ of $\bfSh(G,K^p)_{k_0}$ as being the vanishing locus of $h_{\tau}$. By the relation of universal abelian schemes \eqref{E:comparison-univ-av}, this definition coincides with the one defined by transferring to the unitary Shimura variety $\bfSh(G',K'^p)_{k_0}$. It is easy to see that
\[
\bfSh(G,K^p)^{\r{sp}}_{\FF_p}\otimes k_0=\bigcap_{\tau\in\Sigma_{\infty}}\bfSh(G,K^p)_{k_0,\tau}.
\]

\begin{theorem}\label{P:superspecial-quaternion}
Assume that $g_\fp$ is odd for every $\fp\in\Sigma_p$.
\begin{enumerate}
  \item For each $\underline\fa\in\fB_{\emptyset}$ as in Theorem \ref{T:supersingular-quaternion}, $W_{\emptyset}(\underline\fa)$ contains the superspecial locus $\bfSh(G,K^p)^{\r{sp}}_{\FF_p}\otimes k_0$, and the morphism $\pi_{\underline\fa}\colon W_{\emptyset}(\underline\fa)_{\FF_p^{\ac}}\to \bfSh(G_{\tS_{\max}},K^p)_{\FF_p^{\ac}}$ induces a bijection
      \begin{align*}
      \bfSh(G,K^p)^{\r{sp}}(\bF_p^{\ac})\xra{\sim}\bfSh(G_{\tS_{\max}},K^p)(\FF_p^{\ac})=B^{\times}_{\tS_{\max}}\backslash \widehat{B}^{\times}_{\max}/K^p\prod_{\fp\in \Sigma_p}K_{\fp}
      \end{align*}
      compatible with prime-to-$p$ Hecke correspondences. Here, $K_{\fp}$ is the unique  maximal open compact subgroups of $(B_{\max}\otimes_F F_{\fp})^{\times}$ for each $\fp\in \Sigma_p$.\footnote{Note that the assumption $g_{\fp}$ is odd implies that $B_{\tS_{\max}}$ ramifies at $\fp$.}

  \item Via this  bijection in (1), the action of the arithmetic Frobenius element $\sigma_{p}^2\in \Gal(\FF_p^{\ac}/\FF_{p^2})$ on $\bfSh(G,K^p)^{\r{sp}}(\FF_p^{\ac})$ is given by the  multiplication by the element $\underline{p}^{-1}\in \widehat{F}^{\times}\subseteq \widehat{B}_{\tS_{\max}}^{\times}$, which equals $p^{-1}$ at all $p$-adic places and equals to $1$ at other places.
\end{enumerate}
\end{theorem}

\begin{proof}
Statement (1) follows from Proposition~\ref{P:superspecial-unitary}. To prove (2), we take a superspecial point $x=(A,\iota, \bar\lambda,\bar \alpha_{K^p})\in \bfSh(G,K^p)^{\r{sp}}(\FF_p^{\ac})$ as in Section \ref{S:quaternion-PEL}. Then $A$ is  of the form $A=E\otimes_{\bZ}\cI$, where $E$ is a supersingular elliptic curve and $\cI$ is a (left) fractional ideal of $\cO_B$. It is well known that all supersingular elliptic curves are defined over $\FF_{p^2}$,  and the $p^2$-Frobenius endomorphism $F_{E}^2\colon E\to E^{(p^2)}\cong E$ induces is identified with the multiplication by $-p$. It follows that the effect of  $\sigma_{p^2}$ on $\bfSh(G,K^p)^{\r{sp}}(\FF_p^{\ac})$ coincides with the central Hecke action  on $B^{\times}_{\tS_{\max}}\backslash \widehat{B}^{\times}_{\max}/K^p\prod_{\fp\in \Sigma_p}K_{\fp}$  by the id\`ele $\underline{-p}^{-1}\in \widehat{F}^{\times}$ which equals to $-p^{-1}$ at all $p$-adic places and $1$ elsewhere. We conclude by remarking that the difference between $\underline{-p}^{-1}$ and $\underline{p}^{-1}$ lies in $K^p\prod_{\fp}K_{\fp}$.

\end{proof}

In what follows, we will identify $\bfSh(G,K^p)^{\r{sp}}_{\bF_p^{\ac}}$ and $\bfSh(G_{\tS_{\max}},K^p)_{\FF_p^{\ac}}$ through some fixed $\underline\fa$.
Via this identification,  $\bfSh(G_{\tS_{\max}},K^p)_{\FF_p^{\ac}}$ is equipped with a structures of $\bF_p$-scheme, denoted by $\bfSh(G_{\tS_{\max}},K^p)_{\bF_p}$. From Theorem~\ref{T:supersingular-quaternion}(2), the underlying $\bF_{p^2}$-structure on $\bfSh(G_{\tS_{\max}},K^p)_{\FF_p^{\ac}}$ does not depend on the choice of $\underline\fa$.

\begin{corollary}\label{C:Galois-equivariance}
Assume that $g_{\fp}$ is odd for every $\fp\in \Sigma_p$. For $\underline{\fa}\in \fB_{\emptyset}$, the morphism $\pi_{\underline{\fa}}\colon W_{\emptyset}(\underline {\fa})_{\FF_p^{\ac}}\to \bfSh(G_{\tS_{\max}}, K^p)_{\FF_p^{\ac}}$ is equivariant under $\Gal(\FF_p^{\ac}/k_0)$, and hence it  descends to a morphism of $k_0$-schemes:
\[
\pi_{\underline{\fa}}\colon W_{\emptyset}(\underline{\fa})\to \bfSh(G_{\tS_{\max}},K^p)_{k_0}.
\]
\end{corollary}
\begin{proof}
This follows from the definition of underlying $k_0$-structure on $\bfSh(G_{\tS_{\max}},K^p)_{\FF_p^{\ac}}$ and the fact that the inclusion $\bfSh(G, K^p)^{\r{sp}}_{\FF_p^{\ac}}\hookrightarrow W_{\emptyset}(\underline {\fa})_{\FF_p^{\ac}}$ is equivariant under $\Gal(\FF_p^{\ac}/k_0)$.
\end{proof}

\if false

\yifeng{In Section \ref{S:application}, how sufficiently small should $K^p$ be? Can I take any $K^p$ such that $K_pK^p$ is neat?}

{\yichao{If you want each irreducible component of the  supersingular locus to be isomorphic to certain iterated $\dP^1$-bundle, I think you need more restrictive conditions on $K^p$ other than being neat.  Being neat just that the Shimura variety $\bfSh(G, K^p)_{\FF_p}$  is a scheme. }

\yifeng{In the case $\tS=\emptyset$, we do need a direct moduli interpretation of $\Sh(G,K^p)$ for later argument. In particular, the comparison \eqref{E:comparison-univ-av} can be done directly over $\bfSh(G)$. Moreover, one can show directly that $\pi_{\underline\fa}\colon W_{\emptyset}(\underline\fa)\to\bfSh(G_{\tS_{\max}}, K^p)$ is defined over $k_0$. (It doesn't seem obvious to me that the hypersurface $\bfSh(G_{\tS})_{\FF_p^{\ac},\tau}$ is stable under $\Gal(\FF_p^{\ac}/k_0)$ in the non-PEL case.)

In \cite{Zin82}, Zink considers the case where $B$ is a totally indefinite quaternion algebra over $F$, and $p$ a prime inert in $F$ at which $B$ is \emph{ramified}. He shows in \cite{Zin82}*{Section 3} that $\Sh(G,K^p)$ has a moduli interpretation over $\bZ_{(p)}$. It seems to me that his proof works for every $p$ that is unramified in $F$ and such that $K_p$ is maximal.}

\yichao{ Zink shows that $\bfSh(G, K^p)$ classifies the isomorphism classes of  tuples $(A,\iota, \bar \lambda, \bar\alpha_{K^p})$, where
\begin{enumerate}
\item $(A,\iota)$ is an abelian scheme of dimension $4g$ equipped with an action by $\cO_B$,
\item  $\bar \lambda$ is an $\cO_{F,+}^{\times}$-orbit of $\cO_F$-linear prime-to-$p$   polarizations $\lambda$ on $A$,
\item $\bar\alpha_{K^p}$ is a $K^p$-level structure.
\end{enumerate}
He shows that $\bfSh(G,K^p)$ is a quotient by the finite group $(\cO_{F,+}^{\times}\cap \det(K))/(\cO_F^{\times}\cap K)^{2}$ of some fine moduli space $\tilde \cM_{C}$, and the action of $(\cO_{F,+}^{\times}\cap \det(K))/(\cO_F^{\times}\cap K)^{2}$ is free. But I don't understand why he says in \cite{Zin82}*{Korollar~3.4} that $\bfSh(G,K^p)$ is a fine moduli space  since every point $x=(A,\iota, \bar \lambda, \bar\alpha)\in \bfSh(G, K^p)$ has automorphism group $\cO_F^{\times}\cap K$. But it seems that the universal abelian scheme $\cA$ over $\tilde \cM_{C}$ descends to $\bfSh(G,K^p)$ since the action of $(\cO_{F,+}^{\times}\cap \det(K))/(\cO_F^{\times}\cap K)^{2}$ only affect the polarization on $\cA$.}

\yifeng{Let $K=\prod_{v}K_v$ be an open compact subgroup of $B^\times(\bA_F^\infty)$. Denote by $\Delta$ the ramification set of $B$. For every $g_v\in B^\times(F_v)$ with $v\not\in\Delta$, let $\Gamma_{g_v}$ be the subgroup of $(F_v^\ac)^\times$ generated by eigenvalues of $g_v$. The group $((\bQ^\ac)^\times\cap\Gamma_{g_v})_\tor$ does not depend on the embedding $\bQ^\ac\hookrightarrow F_v^\ac$. We say that $g\in B^\times(\bA_F^\infty)$ is neat if $\bigcap_{v\not\in\Delta}((\bQ^\ac)^\times\cap\Gamma_{g_v})_\tor=\{1\}$. We say that $K$ is \emph{neat} if every element $g=g_\Delta g^\Delta\in K$ with $\det g^\Delta=1$ is neat. In fact, if $K$ is contained in $K_1(N)$ for some $N\geq 4$ (and coprime to $\Delta$), then it must be neat.

Now suppose that $p$ is unramified in $F$ such that $K_\fp$ is hyperspecial maximal for every $\fp$ above $p$. Then \cite{Zin82}*{Korollary~3.3} holds for neat $K$, and we can apply the argument of \cite{Zin82}*{Korollary~3.4} to obtain that $\bfSh(G,K^p)$ is the quotient of a fine moduli scheme $\widetilde{\bfSh}(G,K^p)$ under a free action of the finite group $Q_K\coloneqq (\cO_{F,+}^{\times}\cap \det(K))/(\cO_F^{\times}\cap K)^{2}$. Moreover, for every $\tS$, $Q_K$ acts freely on $\widetilde{\bfSh}(G_\tS,K^p)_{k_0}$. Thus all strata of $\bfSh(G,K^p)_{k_0}$ all still $\dP^1$-fibrations.}

\fi

\section{Arithmetic level raising}
\label{ss:4}

In this chapter, we state and prove the arithmetic level raising result.
We suppose that $g=[F:\bQ]$ is odd. Fix an irreducible cuspidal automorphic representation $\Pi$ of $\GL_2(\bA_F)$ of parallel weight $2$ defined over a number field $\bE$.

\subsection{Statement of arithmetic level raising}
\label{ss:statement}

Let $B$ be a totally indefinite quaternion algebra over $F$, and put $G\coloneqq\Res_{F/\bQ}B^\times$. Let $K$ be a neat open compact subgroup of $G(\bA^{\infty})$ (Definition~\ref{D:neat-subgroup}).   Then we have the Shimura variety $\Sh(G,K)$ defined over $\bQ$ whose $\CC$-points are given by
\[
\Sh(G,K)(\CC)=G(\QQ)\backslash(\fH^{\pm})^{\Sigma_{\infty}}\times G(\bA^{\infty})/K.
\]

Let $\tR$ be a finite set of places\footnote{The meaning of $\tR$ changes from here; in particular, it contains the ramification set of $B$, which it previously stands for.} of $F$ away from which $\Pi$ is unramified and $K$ is hyperspecial maximal. Let $\dT^\tR$ be the Hecke monoid away from $\tR$ \cite{Liu2}*{Notation~3.1} (that is, the commutative monoid generated by $\rT_\fq,\rS_\fq,\rS^{-1}_\fq$ with the relation $\rS_\fq\rS_\fq^{-1}=1$ for all primes $\fq\not\in\tR$). Then $\Pi$ induces a homomorphism
\[
\phi^\tR_\Pi\colon\bZ[\dT^\tR]\to\cO_\bE
\]
by its Hecke eigenvalues. For every prime $\lambda$ of $\bE$, we have an attached  Galois representation
\begin{align}\label{eq:galois}
\rho_{\Pi,\lambda}\colon\rG_F=\Gal(F^\ac/F)\to\GL_2(\cO_{\bE_\lambda})
\end{align}
which is unramified outside $\tR\cup\tR_{\lambda}$, where $\tR_{\lambda}$ denotes the subset of  all places of $F$ with the same residue characteristic as $\lambda$. The Galois representation $\rho_{\Pi,\lambda}$ is normalized so that if $\sigma_{\fq}$  denotes an \emph{arithmetic} Frobenius element at $\fq$ for a place $\fq\notin \tR\cup\tR_{\lambda}$, then the characteristic polynomial of $\rho_{\Pi,\lambda}(\sigma_{\fq})$ is given by
\[
X^2-\phi_{\Pi}^{\tR}(\rT_{\fq}) X+ \rN_{F/\QQ}(\fq)\phi_{\Pi}^{\tR}(\rS_{\fq}).
\]
Let $\fm^\tR_{\Pi,\lambda}$ be the kernel of the composite map $\bZ[\dT^\tR]\xrightarrow{\phi_\Pi}\cO_\bE\to\cO_\bE/\lambda$. From now on, we suppose that the following is satisfied.

\begin{assumption}\label{A:assumption-wp}
Let $\ell$ be the underlying rational prime of $\lambda$. Then we assume that
\begin{enumerate}
  \item $\ell$ is coprime to $\tR$, $\disc F$, and the cardinality of $F^\times\backslash\bA_F^{\infty,\times}/(\bA_F^{\infty,\times}\cap K)$;

  \item $\ell\geq g+2$;


  \item the representation $\bar\rho_{\Pi,\lambda}\coloneqq\rho_{\Pi,\lambda}\mod\lambda$ satisfies the condition $(\b{LI}_{\Ind\bar\rho_{\Pi,\lambda}})$ in \cite{Dim05}*{Proposition~0.1};

  \item $\rH^g(\Sh(G,K)_{\bQ^{\ac}},\cO_\bE/\lambda)/\fm^\tR_{\Pi,\lambda}$ has dimension $2^g\dim(\Pi_B^\infty)^K$ over $\cO_\bE/\lambda$, where $\Pi_B$ is the automorphic representation of $G(\bA)$ whose  Jacquet--Langlands transfer to $\GL_2(\bA_F)$ is $\Pi$.

\end{enumerate}
\end{assumption}

\begin{remark}\label{R:assumption-wp}
We have following remarks concerning Assumption \ref{A:assumption-wp}.
\begin{enumerate}

  \item Assumption \ref{A:assumption-wp}(3) implies that $\bar\rho_{\Pi,\lambda}$ is absolutely irreducible.

  \item If $\Pi$ is not dihedral and not isomorphic to a twist by a character of any of its internal conjugates, then Assumption \ref{A:assumption-wp}(3) holds for all but finitely many $\lambda$ by \cite{Dim05}*{Proposition~0.1}. In particular, for such a $\Pi$, the entire Assumption \ref{A:assumption-wp} holds for all but finitely many $\lambda$.

  \item In general, the dimension of $\rH^g(\Sh(G,K)_{\bQ^{\ac}},\cO_\bE/\lambda)/\fm^\tR_{\Pi,\lambda}$ is at least $2^g\dim_E(\Pi_B^\infty)^K$ over $\cO_\bE/\lambda$.
\end{enumerate}
\end{remark}

Let $p$ be a rational prime inert in $F$, coprime to $\tR\cup\{2,\ell\}$. Denote by $\fp$ the unique prime of $F$ above $p$. To ease notation, we put
\[
\phi\coloneqq\phi_{\Pi}^{\tR\cup\{\fp\}}\colon\ZZ[\dT^{\tR\cup\{\fp\}}]\to\cO_{\bE},\qquad
\fm\coloneqq\fm_{\Pi,\lambda}^{\tR\cup\{\fp\}}\subseteq\ZZ[\dT^{\tR\cup\{\fp\}}].
\]
For a $\ZZ[\dT^{\tR\cup\{\fp\}}]$-module $M$, we denote by $M_\fm$ its localization at $\fm$. Write $K=K_p K^p$ where $K_p$ is a hyperspecial maximal subgroup of $G(\bQ_p)$ as $p\not\in\tR$. We have the integral model $\bfSh(G,K^p)$ over $\bZ_p$ defined in Section \ref{S:quaternion-PEL} for the Shimura variety $\Sh(G,K^p)=\Sh(G,K)$. Put $\fB\coloneqq\fB(\emptyset,(g-1)/2)$, the set of periodic semi-meanders attached to $\tS=\emptyset$ with $g$-nodes and $(g-1)/2$-arcs. We note that $k_0$ defined in Subsection~\ref{S:notation} is $\FF_{p^{2g}}$ in the current case. Then Theorem~\ref{T:supersingular-quaternion} asserts that
\[
\bfSh(G,K^p)_{\FF_p}^{\mathrm{ss}}\otimes\bF_{p^{2g}}=\bigcup_{\fa\in\fB}W(\fa),
\]
where each $W(\fa)=W_{\emptyset}(\fa)$ is equipped with a $(g-1)/2$-th iterated $\dP^1$-fibration
\[
\pi_\fa\colon W(\fa)\to\bfSh(G_{\tS_{\max}},K^p)_{\bF_{p^{2g}}}.
\]
Let
\[
\bfSh(G,K^{p})^{\r{sp}}_{\FF_p}\subseteq \bfSh(G,K^p)_{\FF_p}
\]
be the superspecial locus as in Definition \ref{D:superspecial}. By Theorem~\ref{P:superspecial-quaternion}, each $W(\fa)$ for $\fa\in\fB$ contains $\bfSh(G,K^{p})^{\r{sp}}_{\FF_{p^{2g}}}$, and the morphism $\pi_\fa$ induces an isomorphism
\[
\bfSh(G,K^p)^{\r{sp}}_{\bF_{p^{2g}}}\xrightarrow{\sim}\bfSh(G_{\tS_{\max}},K^{p})_{\bF_{p^{2g}}}
\]
which is equivariant under prime-to-$p$ Hecke correspondences, and independent of $\fa$.

Consider the set $\fB\times\bfSh(G_{\tS_{\max}},K^p)(\bF_p^\ac)$, equipped with the diagonal action by $\rG_{\bF_p}$. The Hecke monoid $\dT^{\tR\cup\{\fp\}}$ acts through the second factor. We have a Chow cycle class map
\begin{align}\label{E:AJ0}
\Gamma(\fB\times\bfSh(G_{\tS_{\max}},K^{p})(\bF_p^\ac),\bZ)\to\CH^{(g+1)/2}(\bfSh(G,K^p)_{\FF_p^\ac})
\end{align}
sending a function $f$  on $\fB\times\bfSh(G_{\tS_{\max}},K^p)(\bF_p^\ac)$ to the Chow class of $\sum_{\fa, s} f(\fa,s)\pi_\fa^{-1}(s)$.

\begin{lem}\label{L:equivariant}
The map \eqref{E:AJ0} is equivariant under both $\dT^{\tR\cup\{\fp\}}$ and $\rG_{\bF_p}$.
\end{lem}

\begin{proof}
The equivariance of $\pi_{\fa}$ under prime-to-$p$ Hecke correspondences follows from Theorem~\ref{P:superspecial-quaternion}. The equivariance under $\rG_{\FF_p}$ follows from the definition of $\rG_{\FF_p}$ on $\bfSh(G_{\tS_{\max}},K^p)(\bF_p^\ac)$.
\end{proof}

\begin{lem}\label{L:vanishing}
Under the notation above, the following statements hold:
\begin{enumerate}
  \item There exists a canonical isomorphism  $\rH^g(\bfSh(G,K^p)_{\bF_p^{\ac}},\cO_{\bE_\lambda})_\fm\to\rH^g(\Sh(G,K)_{\bQ^{\ac}},\cO_{\bE_\lambda})_\fm$ compatible with Galois actions. In particular, we have a canonical isomorphism
      \[
      \rH^1(\bF_{p^h},\rH^g(\bfSh(G,K^p)_{\bF_p^{\ac}},\cO_{\bE}/\lambda((g+1)/{2}))_\fm)
      \cong\rH^1_\unr(\bQ_{p^h},\rH^g(\Sh(G,K)_{\bQ^{\ac}},\cO_{\bE}/\lambda((g+1)/{2}))_\fm)
      \]
      for every integer $h\geq 1$;

  \item $\rH^i(\bfSh(G,K^p)_{\bF_p^\ac},\cO_{\bE_\lambda})_\fm=0$ unless $i=g$;

  \item $\rH^g(\bfSh(G,K^p)_{\bF_p^{\ac}},\cO_{\bE_\lambda})_\fm$ is a finite free $\cO_{\bE_\lambda}$-module.
\end{enumerate}
\end{lem}

\begin{proof}
By \cite{LS16}*{Corollary~4.6}, no matter whether the Shimura variety $\bfSh(G,K^p)$ is proper over $\ZZ_{(p)}$, the canonical maps
\[
\rH^{i}(\bfSh(G,K^p)_{\bF_p^\ac},\cO_{\bE_\lambda})\xra{\sim} \rH^i(\bfSh(G,K^p)_{\QQ_p^{\ac}},\cO_{\bE_\lambda})\xleftarrow{\sim}\rH^i(\Sh(G,K^p)_{\QQ^{\ac}},\cO_{\bE_\lambda})
\]
for all $i\geq 0$
are isomorphisms compatible with Hecke and Galois actions. One gets thus Statement (1) by localizing the Hecke action at $\fm$. Statements (2) and (3) then follow from Assumption \ref{A:assumption-wp}(3) and \cite{Dim05}*{Theorem~0.3}.
\end{proof}

To ease notation, put $\rG'\coloneqq\Gal(\bF_p^\ac/\bF_{p^{2g}})$. Lemma \ref{L:equivariant} induces the following map
\begin{align}\label{E:AJ1}
\Gamma(\fB\times\bfSh(G_{\tS_{\max}},K^p)(\bF_p^\ac),\bZ)^{\rG'}
\to\CH^{(g+1)/2}(\bfSh(G,K^p)_{\bF_{p^{2g}}})
\end{align}
which is equivariant under both $\dT^{\tR\cup\{\fp\}}$ and $\Gal(\bF_{p^{2g}}/\bF_p)$. On the other hand, one has a cycle class map
\begin{align*}
\CH^{(g+1)/2}(\bfSh(G,K^p)_{\FF_{p^{2g}}})
\to\rH^{g+1}(\bfSh(G,K^p)_{\FF_{p^{2g}}},\cO_{\bE_\lambda}((g+1)/{2})).
\end{align*}
However, by the Hochschild--Serre spectral sequence and Lemma \ref{L:vanishing}(1), we have a canonical isomorphism
\begin{align*}
\rH^{g+1}(\bfSh(G,K^p)_{\bF_{p^{2g}}},\cO_{\bE_\lambda}((g+1)/{2}))_\fm
\cong\rH^{1}(\FF_{p^{2g}},\rH^{g}(\bfSh(G,K^p)_{\FF^{\ac}_p},\cO_{\bE_\lambda}((g+1)/{2}))_\fm).
\end{align*}
Therefore, composing with (the localization of) \eqref{E:AJ1} and modulo $\lambda$, we obtain a morphism
\begin{align}\label{E:AJ2}
\Phi_\fm\colon\Gamma(\fB\times\bfSh(G_{\tS_{\max}},K^p)(\bF_p^\ac),\cO_{\bE}/\lambda)_\fm^{\rG'}
\to\rH^1(\bF_{p^{2g}},\rH^g(\bfSh(G,K^p)_{\FF^{\ac}_p},\cO_\bE/\lambda((g+1)/2))_\fm),
\end{align}
called the \emph{unramified level raising map at $\fm$}. It is equivariant under the action of $\Gal(\bF_{p^{2g}}/\bF_p)$.

\begin{definition}\label{D:level-raising-prime}
We say that a rational prime $p$ is a \emph{$\lambda$-level raising prime} (with respect to $\Pi,B,K,\tR$) if
\begin{description}
  \item[(L1)] $p$ is inert in $F$, and coprime to  $\tR\cup\{2,\ell\}$;

%

  \item[(L2)] $\ell\nmid\prod_{i=1}^g(p^{2gi}-1)$;

  \item[(L3)] $\phi_{\Pi}^\tR(\rT_\fp)^2\equiv (p^{g}+1)^{2}\mod\lambda$ and $\phi_{\Pi}^\tR(\rS_\fp)\equiv 1\mod\lambda$.
\end{description}
\end{definition}

\begin{remark}\label{R:level-raising-prime}
We have the following remarks concerning level raising primes.
\begin{enumerate}
  \item By a similar argument of \cite{Liu2}*{Lemma~4.11}, one can show there are infinitely many $\lambda$-level raising primes with positive density, as long as there exist rational primes inert in $F$ and $\lambda$ satisfies Assumption \ref{A:assumption-wp}.

  \item By the Eichler--Shimura congruence relation, Definition~\ref{D:level-raising-prime}(L3) is equivalent to saying that $\bar\rho_{\Pi,\lambda}(\sigma_{\fp})$ is conjugate to $\pm\big(\begin{smallmatrix}1 &0\\ 0& p^g\end{smallmatrix}\big)$.

  \item By the Eichler--Shimura congruence relation and the Chebotarev's density theorem, we know that the canonical map
      \[
      \rH^g(\Sh(G,K)_{\bQ^{\ac}},\cO_\bE/\lambda)/\fm
      \to\rH^g(\Sh(G,K)_{\bQ^{\ac}},\cO_\bE/\lambda)/\fm^\tR_{\Pi,\lambda}
      \]
      is an isomorphism of $\cO_\bE/\lambda[\rG_\bQ]$-modules.

\end{enumerate}
\end{remark}

\begin{theorem}[Arithmetic level raising]\label{T:level_raising}
Let $\lambda$ be a prime of $\cO_E$ such that Assumption \ref{A:assumption-wp} is satisfied. Let $p$ be a $\lambda$-level raising prime. Then $\rG'$ acts trivially on $\Gamma(\fB\times\bfSh(G_{\tS_{\max}},K^p)(\bF_p^\ac),\cO_{\bE}/\lambda)_\fm$; and the induced map
\begin{align}\label{E:AJ4}
\Gamma(\fB\times\bfSh(G_{\tS_{\max}},K^p)(\bF_p^\ac),\cO_{\bE}/\lambda)/\fm
\to\rH^1(\bF_{p^{2g}},\rH^g(\bfSh(G,K^p)_{\FF^{\ac}_p},\cO_\bE/\lambda((g+1)/2))/\fm)
\end{align}
is surjective.
\end{theorem}

\subsection{Proof of arithmetic level raising}
\label{S:proof}

This section is devoted to the proof of Theorem~\ref{T:level_raising}. We assume that we are not in the case where $F=\bQ$ and $B$ is the matrix algebra, since this is already known by Ribet.


For $\fa\in \fB$, denote $\tau(\fa)\in\Sigma_{\infty}$ the end point of the unique semi-line in $\fa$. By the construction in Section \ref{S:application}, for each $\fa\in\fB$, the stratum $W(\fa)$ fits into the following commutative diagram
\begin{align}\label{eq:reduction}
\xymatrix{
W(\fa)\ar@{^(->}[r]\ar[d] &Z_{\emptyset}(\fa)\ar@{^(->}[r] \ar[d]^{\pi_{\fa}} & \bfSh(G,K^p)_{\FF_{p^{2g}}}\\
\bfSh(G_{\emptyset_{\fa}}, K^p)_{\FF_{p^{2g}}, \tau(\fa)}\ar[d]^{\cong} \ar@{^(->}[r] &\bfSh(G_{\emptyset_{\fa}}, K^p)_{\FF_{p^{2g}}}\\
\bfSh(G_{\tS_{\max}},K^p)_{\FF_{p^{2g}}},
}
\end{align}
where the square is Cartesian. Note that $\bfSh(G_{\emptyset_{\fa}},K^p)$ is a \emph{proper} Shimura curve over $\cO_{F,\fp}$ (with $F$ regarded as a subfield of $\bQ^\ac$ determined by $\fa$), and $\bfSh(G_{\emptyset_{\fa}},K^p)_{\FF_{p^{2g}},\tau(\fa)}\cong\bfSh(G_{\tS_{\max}},K^p)_{\FF_{p^{2g}}}$ is exactly its ``supersingular'' locus. Similarly to \eqref{E:AJ1}, we have a Chow class map
\[
\Gamma(\bfSh(G_{\tS_{\max}},K^p)(\FF_p^{\ac}),\ZZ)\to\CH^1(\bfSh(G_{\emptyset_{\fa}},K^p)_{\FF_p^{\ac}}),
\]
which induces an unramified level raising map for the Shimura curve $\bfSh(G_{\emptyset_{\fa}}, K^p)$:
\begin{align}\label{E:curve}
\Phi_\fm(\fa)\colon
\Gamma(\bfSh(G_{\tS_{\max}}, K^p)(\FF_p^{\ac}), \cO_{\bE}/\lambda)_\fm^{\rG'}\to
\rH^1(\FF_{p^{2g}},\rH^1(\bfSh(G_{\emptyset_{\fa}},K^p)_{\FF_{p}^{\ac}},\cO_{\bE}/\lambda(1))_\fm).
\end{align}
The following is an analogue of Theorem~\ref{T:level_raising} for Shimura curves.

\begin{proposition}\label{P:level_raising_curve}
Under the hypothesis of Theorem~\ref{T:level_raising}, the map $\Phi_\fm(\fa)$ is surjective.
\end{proposition}

To prove this Proposition, we need some preparation. Let $\Iw_p\subseteq G_{\emptyset_{\fa}}(\QQ_p)\cong \GL_2(F_{\fp})$  be the standard Iwahoric subgroup, and $\Sh(G_{\emptyset_{\fa}}, K^p\mathrm{Iw}_p)$ be the Shimura curve attached to $G_{\emptyset_{\fa}}$ of level $K^p\Iw_p$.
By  \cite{Ca86}, $\Sh(G_{\emptyset_{\fa}}, K^p\mathrm{Iw}_p)$ admits an integral model $\bfSh(G_{\emptyset_{\fa}}, K^p\Iw_p)$ over $\cO_{F,\fp}$ with semi-stable reduction. The special fiber $\bfSh(G_{\emptyset_{\fa},K^p\Iw_p})_{\FF_{p^g}}$ consists of two copies of $\bfSh(G_{\emptyset_{\fa}}, K^p\Iw_p)_{\FF_{p^g}}$ cutting transversally at supersingular points. There are two natural degeneracy maps
\[
\pi_1,\pi_2\colon\bfSh(G_{\emptyset_{\fa}},K^p\Iw_p)\to \bfSh(G_{\emptyset_{\fa}},K^p).
\]
Recall the following generalization of Ihara's Lemma to Shimura curves over totally real fields.

\begin{lem}[\cite{Che13+}*{Theorem~1.2}]\label{L:Ihara-Lemma}
Let $\fm$ be a non-Eisenstein maximal ideal of the Hecke algebra $\ZZ[\dT^{\tR\cup\{\fp\}}]$. Then the canonical map
\[
\pi_1^*+\pi^*_2\colon\rH^1(\bfSh(G_{\emptyset_{\fa}},K^p)_{\bQ^{\ac}},\cO_{\bE}/\lambda)_\fm^{\oplus 2}
\to\rH^1(\bfSh(G_{\emptyset_{\fa}}, K^p\Iw_p)_{\bQ^\ac},\cO_{\bE}/\lambda)_\fm
\]
is injective.
\end{lem}

\begin{proof}[Proof of Proposition~\ref{P:level_raising_curve}]
To simplify notation, let us put $X\coloneqq\bfSh(G_{\emptyset_{\fa}},K^p)$ viewed as a proper smooth scheme over $\cO_{F,\fp}$, the ``supersingular'' locus
\[
X^{\mathrm{ss}}_{\FF_{p^{2g}}}\coloneqq\bfSh(G_{\emptyset_{\fa}},K^p)_{\FF_{p^{2g}},\tau_{\fa}}\cong\bfSh(G_{\tS_{\max}},K^p)_{\FF_{p^{2g}}},
\]
and $X_0(p)=\bfSh(G_{\emptyset_{\fa}},K^p\Iw_p)$. We put also $k_{\lambda}\coloneqq\cO_{\bE}/\lambda$. Consider the canonical short exact sequence
\[
\rH^0(X_{\FF_p^{\ac}},k_{\lambda})\to \rH^0(X_{\FF_p^{\ac}}^{\mathrm{ss}}, k_{\lambda})\to \rH^1_c(X^{\ord}_{\FF_p^{\ac}}, k_{\lambda})\to \rH^1(X_{\FF_p^{\ac}},k_{\lambda})\to0
\]
equivariant under the action of $\rG(\bF_p^\ac/\FF_{p^g})\times\ZZ[\dT^{\tR\cup \{\fp\}}]$, where $X^{\ord}_{\FF_p^{\ac}}=X_{\FF_p^{\ac}}-X^{\mathrm{ss}}_{\FF_p^{\ac}}$ is the ``ordinary'' locus. The first term vanishes after localizing at $\fm$ by Assumption~\ref{A:assumption-wp}(3). Taking Galois cohomology $\rH^i(\FF_{p^{2g}},-)$, one deduces a boundary map
\[
\Phi_\fm^*(\fa)\colon \rH^1(X_{\FF_p^{\ac}}, k_{\lambda})^{\rG'}_\fm\to \rH^1(\FF_{p^{2g}},\rH^0(X^{\mathrm{ss}}_{\FF_p^{\ac}},k_{\lambda})_\fm).
\]
By the Poincar\'e duality and the duality of Galois cohomology over finite fields, it is easy to see that $\Phi^*_\fm(\fa)$ is identified with the dual map of $\Phi_\fm(\fa)$. Therefore, to finish the proof of Proposition~\ref{P:level_raising_curve}, it suffices to show that $\Phi^*_\fm(\fa)$ is injective.

Recall that $X_0(p)_{\FF_{p^g}}$ consists of two copies of $X_{\FF_{p^g}}$. Let $i_1\colon X_{\FF_{p^g}}\to X_0(p)_{\FF_{p^g}}$ be the copy such that $\pi_1\circ i_1$ is the identity, and $i_2\colon X_{\FF_{p^g}}\to X_0(p)_{\FF_{p^g}}$ be the one such that $\pi_2\circ i_2$ is the identity. It is well known that $\pi_2\circ i_1$ is the Frobenius endomorphism of $X_{\FF_{p^g}}$ relative to $\FF_{p^g}$, and $\pi_1\circ i_2$ is the Frobenius endomorphism of $X_{\FF_{p^g}}$ relative to $\FF_{p^g}$ composed with the Hecke action $\rS_{\fp}^{-1}$. Consider the normalization $\delta\colon \widetilde X_0(p)_{ \FF_{p^{g}}}=X_{\FF_{p^g}}\coprod X_{\FF_{p^g}}\to X_0(p)_{\FF_{p^{g}}}$. Then one has an exact sequence of \'etale sheaves
\[
0\to k_{\lambda}\to \delta_*k_{\lambda}\to i_*^{\mathrm{ss}} k_{\lambda}\to 0
\]
on $X_{0}(p)_{\FF_{p^{g}}}$, where $i^{\mathrm{ss}}\colon X^{\mathrm{ss}}_{\FF_{p^{g}}}\to X_0(p)_{\FF_{p^{g}}}$ denotes the closed immersion of the singular locus of $X_0(p)_{\FF_{p^{g}}}$, and the second map $\delta_*k_{\lambda}\to i_*^{\mathrm{ss}}k_{\lambda}$ is given as follows: If $x\in X^{\mathrm{ss}}_{\FF_{p^{g}}}(\FF_p^{\ac})$ is a supersingular geometric point with preimage $\delta^{-1}(x)=(x_1,x_2)$ with $x_j\in i_j(X(\FF_p^{\ac}))$ for $j=1,2$, then $(\delta_*k_{\lambda})_x=k_{\lambda,x_1}\oplus k_{\lambda,x_2}\to k_{\lambda,x}$ is given by $(a,b)\mapsto a-b$. By the functoriality of cohomology, we get
\begin{equation}\label{E:kernel-res}
0=\rH^0(X_{\FF_p^{\ac}}, k_{\lambda})_\fm\to \rH^0(X^{\mathrm{ss}}_{\FF_p^{\ac}},k_{\lambda})_\fm\to \rH^1(X_0(p)_{\FF_p^{\ac}},k_{\lambda})_\fm\xra{(i_1^*, i_2^*)} \rH^1(X_{\FF_p^{\ac}}, k_{\lambda})^{\oplus 2}_\fm\to 0.
\end{equation}
Consider the map
\begin{equation}\label{E:pullback-special}
\pi_1^*+\pi^*_2\colon \rH^1(X_{\FF_p^{\ac}},k_{\lambda})_\fm^{\oplus 2}\to \rH^1(X_0(p)_{\FF_p^{\ac}}, k_{\lambda})_\fm
\end{equation}
induced by the two degeneracy maps $\pi_1,\pi_2\colon X_0(p)\to X$. If $\Fr_{\fp}$ denotes the action on $\rH^1(X_{\FF_{p^g}},k_{\lambda})$ induced by the Frobenius endomorphism relative to $\FF_{p^g}$, then $\Fr_{\fp}=\sigma^{-1}_{\fp}$ and  the composite map
\[
\theta\colon\rH^1(X_{\FF_p^{\ac}},k_{\lambda})_\fm^{\oplus 2}\xra{\pi^*_1+\pi_2^*}\rH^1(X_0(p)_{\FF_p^{\ac}},k_{\lambda})_\fm\xra{(i_1^*, i_2^*)} \rH^1(X_{\FF_p^{\ac}},k_{\lambda})_\fm^{\oplus 2}
\]
is given by the matrix $\big(\begin{smallmatrix}1& \Fr_{\fp} \\ \Fr_{\fp}\rS_{\fp}^{-1} & 1\end{smallmatrix}\big)$.

By Definition \ref{D:level-raising-prime}(L3), we know that the Hecke operator $\rS_\fp$ acts trivially on $\rH^1(X_{\FF_p^{\ac}},k_{\lambda})_\fm$ since the trivial action is the only lifting of the trivial action modulo $\fm$ by Assumption \ref{A:assumption-wp}(1). We see that $\ker\theta$ is identified with the image of the injective morphism
\[
\rH^1(X_{\FF_p^{\ac}}, k_{\lambda})_\fm^{\Fr_{\fp}^2=1}\xra{(-\Fr_{\fp}, \mathrm{Id})}\rH^1(X_{\FF_p^{\ac}}, k_{\lambda})_\fm^{\oplus 2}.
\]
However, by Ihara's Lemma~\ref{L:Ihara-Lemma} and the proper base change, the map $\pi_1^*+\pi_2^*$ in \eqref{E:pullback-special} is injective. Thus, it induces an injection
\[
\Phi^*\colon \rH^1(X_{\FF_p^{\ac}}, k_{\lambda})_\fm^{\Fr_{\fp}^2=1}\cong \ker\theta\to  \ker(i_1^*, i_2^*)\cong  \rH^0(X^{\mathrm{ss}}_{\FF_p^{\ac}}, k_{\lambda})_\fm.
\]
To finish the proof of Proposition~\ref{P:level_raising_curve}, it suffices to show the following claims:
\begin{enumerate}
  \item The action of $\Fr_{\fp}^2$ on $\rH^0(X^{\mathrm{ss}}_{\FF_p^{\ac}}, k_{\lambda})_\fm$ is trivial so that the natural projection
      \[
      \rH^0(X^{\mathrm{ss}}_{\FF_p^{\ac}}, k_{\lambda})_\fm\to \rH^1(\FF_{p^{2g}}, \rH^0(X^{\mathrm{ss}}_{\FF_p^{\ac}}, k_{\lambda})_\fm)\cong \rH^0(X^{\mathrm{ss}}_{\FF_p^{\ac}},k_{\lambda})_\fm/(\Fr^2_{\fp}-1)
      \]
      is an isomorphism.

  \item The morphism $\Phi^*$ is identified with $\Phi^*_\fm(\fa)$.
\end{enumerate}

Claim (1) follows from Assumption \ref{A:assumption-wp}(1), Definition \ref{D:level-raising-prime}(L3) and the observation that $\Fr_{\fp}^2$ acts through the Hecke translation by $(1,\dots,1,p,1,\dots)$ where $p$ is placed at the prime $\fp$.

To prove Claim (2), consider the following commutative diagram
\[
\xymatrix{
& & \rH^1_c(X^{\ord}_{\FF_p^{\ac}}, k_{\lambda})_\fm\ar[r]\ar[d]^{\pi^*_2-\pi^*_1\Fr_{\fp}} & \rH^1(X_{\FF_p^{\ac}}, k_{\lambda})\ar[d]^{\pi^*_2-\pi^*_1\Fr_{\fp}} \ar[r] & 0\\
0\ar[r] & \rH^0(X^{\mathrm{ss}}_{\FF_p^{\ac}}, k_{\lambda})_\fm\ar[r]\ar[d]^{\alpha} & \rH^1_c(X^{\ord}_{\FF_p^{\ac}},k_{\lambda})_\fm^{\oplus 2}\ar[r] \ar@{=}[d] & \rH^1(X_0(p)_{\FF_p^{\ac}}, k_{\lambda})_\fm\ar[r] \ar[d]^{(i^*_1, i^*_2)}  &0\\
0\ar[r] &\rH^0(X^{\mathrm{ss}}_{\FF_p^{\ac}}, k_{\lambda})_\fm^{\oplus 2}
\ar[r] & \rH^1_c(X^{\ord}_{\FF_p^{\ac}},k_{\lambda})_\fm^{\oplus 2} \ar[r] & \rH^1(X_{\FF_p^{\ac}}, k_{\lambda})_\fm^{\oplus 2}\ar[r] &0,
}
\]
where $\alpha$ is the diagonal map, and horizontal rows are exact. Then the coboundary isomorphism $\Ker(i_1^*,i_2^*)\cong\rH^0(X^{\mathrm{ss}}_{\FF_p^{\ac}}, k_{\lambda})_\fm$ given by \eqref{E:kernel-res} coincides with
\[
\Ker(i_1^*,i_2^*)\xra{\sim} \coker\alpha\xleftarrow{\sim} \rH^0(X^{\mathrm{ss}}_{\FF_p^{\ac}},k_{\lambda})_\fm,
\]
where the first isomorphism is deduced from the commutative diagram above by the Snake Lemma, and the second is induced by the injection $\rH^0(X^{\mathrm{ss}}_{\FF_p^{\ac}},k_{\lambda})_\fm\hookrightarrow\rH^0(X^{\mathrm{ss}}_{\FF_p^{\ac}},k_{\lambda})_\fm^{\oplus 2}$ to the second component.

Now take $x\in \rH^1(X_{\FF_p^{\ac}},k_{\lambda})^{\Fr_{\fp}^2=1}_\fm\cong\Ker\theta$, and let $\tilde{x}\in\rH^1_c(X^{\ord}_{\FF_p^{\ac}},k_{\lambda})_\fm$ be a lift of $x$ that is fixed by $\rS_{\fp}$. This is possible as the action of $\rS_{\fp}$ on $\rH^1_c(X^{\ord}_{\FF_p^{\ac}},k_{\lambda})$ is semisimple. Then $\pi_2^*(\tilde{x})-\pi^*_1\Fr_{\fp}(\tilde{x})\in\rH^1_c(X^{\ord}_{\FF_p^{\ac}},k_{\lambda})^{\oplus 2}$ is an element lifting
$\pi_2^*(x)-\pi^*_1\Fr_{\fp}(x)\in\Ker(i_1^*,i_2^*)$, and $\pi_2^*(\tilde x)-\pi^*_1\Fr_{\fp}(\tilde x)$ lies actually in the image of $\rH^0(X^{\mathrm{ss}}_{\FF_p^{\ac}},k_{\lambda})_\fm^{\oplus2} $. Note that
\[
\pi_2^*(\tilde x)-\pi^*_1\Fr_{\fp}(\tilde x)=(\rS_{\fp}^{-1}\Fr_{\fp}(\tilde x),\tilde x)-(\Fr_{\fp}(\tilde x),\Fr_{\fp}^2(\tilde x))
=(0,(1-\Fr^2_{\fp})(\tilde x)).
\]
Since $\Phi^*(x)$ is by definition the image of $\pi_2^*(\tilde x)-\pi^*_1\Fr_{\fp}(\tilde x)$ in $\coker\Delta\cong\rH^1(X^{\mathrm{ss}}_{\FF_p^{\ac}}, k_{\lambda})_\fm$, we get $\Phi^*(x)=(1-\Fr_{\fp}^2)(\tilde x)$. However, this is nothing but the image of $x\in\rH^1(X_{\FF_p^{\ac}},k_{\lambda})^{\rG'}_\fm$ via the coboundary map $\Phi^*_\fm(\fa)$. This finishes the proof of claim, hence also the proof of Proposition~\ref{P:level_raising_curve}.
\end{proof}

Recall that we have, for each $\fa\in \fB$,  an algebraic correspondence
\[
\bfSh(G_{\emptyset_{\fa}},K^p)_{\FF_{p^{2g}}}\xleftarrow{\pi_{\fa}} Z_{\emptyset}(\fa)\xrightarrow{i_{\fa}} \bfSh(G, K^p)_{\FF_{p^{2g}}}.
\]
Let $\Lambda$ be $\cO_{\bE_{\lambda}}$, $\cO_{\bE}/\lambda$ or $\QQ_{\ell}^{\ac}$. We define $\Gys_{\fa}(\Lambda)$ to be the composite map
\[
\rH^1(\bfSh(G_{\emptyset_{\fa}}, K^p)_{\FF_{p}^{\ac}},\Lambda)_\fm\xra{\pi_{\fa}^*}\rH^1(W(\fa)_{\FF_p^{\ac}},\Lambda)_\fm
\xra{\mathrm{Gysin}}\rH^{g}(\bfSh(G,K^p)_{\FF_p^{\ac}},\Lambda((g-1)/2))_\fm,
\]
where the first map is an isomorphism since $\pi_{\fa}$ is a $(g-1)/2$-th iterated $\dP^1$-fibrations, and the second map is the Gysin map induced by the closed immersion $i_{\fa}$. Taking sum, we get a map
\begin{equation*}
\Gys(\Lambda)\coloneqq\sum_{\fa}\Gys_{\fa}(\Lambda)\colon \bigoplus_{\fa\in \fB}\rH^1(\bfSh(G_{\emptyset_{\fa}}, K^p)_{\FF_{p}^{\ac}}, \Lambda)_\fm\to\rH^{g}(\bfSh(G,K^p)_{\FF_p^{\ac}},\Lambda((g-1)/2))_\fm.
\end{equation*}

\begin{proposition}\label{P:injective-Gysin}
Under the assumption of Theorem~\ref{T:level_raising},
\begin{enumerate}
  \item the map $\Gys(\Lambda)$ is injective for $\Lambda=\cO_{\bE_{\lambda}},\cO_{\bE}/\lambda,\QQ_{\ell}^{\ac}$;

  \item the induced map
     \[
     \Gys(\cO_\bE/\lambda)/\fm\colon\bigoplus_{\fa\in \fB}\rH^1(\bfSh(G_{\emptyset_{\fa}}, K^p)_{\FF_{p}^{\ac}}, \cO_\bE/\lambda)/\fm\to\rH^{g}(\bfSh(G,K^p)_{\FF_p^{\ac}},\cO_\bE/\lambda((g-1)/2))/\fm
     \]
     is injective.
\end{enumerate}
\end{proposition}

Before giving the proof of the proposition, we introduce some notation. Let $R_\fm$ be the set of all automorphic representations that contribute to $\rH^{g}(\bfSh(G,K^p)_{\FF_p^{\ac}},\Lambda((g-1)/2))_\fm$. Then it is the same as the set of all automorphic representations that contribute to $\rH^1(\bfSh(G_{\emptyset_{\fa}}, K^p)_{\FF_{p}^{\ac}},\Lambda)_\fm$ for every $\fa$ by the Jacquet--Langlands correspondence. It is finite and contains $\Pi$. We may enlarge $\bE$ such that every automorphic representation $\Pi'\in R_\fm$ is defined over $\bE$. Fix an embedding $\bE_\lambda\hookrightarrow\bQ_\ell^\ac$. Let $\alpha_{\Pi'},\beta_{\Pi'}\in \ZZ_{\ell}^{\ac}$ be the eigenvalues of $\rho_{\Pi',\lambda}(\sigma_{\fp})$, where $\ZZ_{\ell}^{\ac}$ denotes the ring of integers of $\QQ_{\ell}^{\ac}$. By Remark~\ref{R:level-raising-prime}(2), we may assume that $\alpha_{\Pi'}^2$ and $\beta_{\Pi'}^2$ are respectively congruent to $1$ and $p^{2g}$  (modulo the maximal ideal of $\ZZ_{\ell}^{\ac}$); in particular, $\alpha_{\Pi'}/\beta_{\Pi'}$ is not congruent to any $i$-th root of unity for $1\leq i\leq 2g$ by Definition~\ref{D:level-raising-prime}(L2).

\begin{proof}[Proof of Proposition~\ref{P:injective-Gysin}]
Following \cite{TX2}, we consider the composite map
\[
\Res_{\fa}(\Lambda)\colon \rH^{g}(\bfSh(G,K^p)_{\FF_p^{\ac}}, \Lambda)_\fm\xra{i_{\fa}^*}\rH^{g}(W(\fa)_{\FF_p^{\ac}},\Lambda)_\fm\xra{\pi_{\fa!}}\rH^1(\bfSh(G_{\emptyset_{\fa}},K^p), \Lambda)_\fm
\]
for each $\fa\in \fB$, and put
\[
\Res(\Lambda)\coloneqq\bigoplus_{\fa\in\fB}\Res_{\fa}(\Lambda) \colon \rH^{g}(\bfSh(G,K^p)_{\FF_p^{\ac}}, \Lambda)_\fm\to
\bigoplus_{\fa\in \fB}\rH^1(\bfSh(G_{\emptyset_{\fa}}, K^p)_{\FF_{p}^{\ac}}, \Lambda)_\fm.
\]
To prove that $\Gys(\Lambda)$ is injective, it suffices to show that the composite map $\Res(\Lambda)\circ\Gys(\Lambda)$, which is an endomorphism of $\bigoplus_{\fa\in \fB}\rH^1(\bfSh(G_{\emptyset_{\fa}}, K^p)_{\FF_{p}^{\ac}}, \Lambda)_\fm$,  is injective. 

It follows from Lemma~\ref{L:vanishing} that
\begin{equation}\label{E:base-change}
\rH^g(\bfSh(G,K^p)_{\FF_p^{\ac}}, \Lambda)_\fm=\rH^g(\bfSh(G,K^p)_{\FF_p^{\ac}}, \cO_{\bE_{\lambda}})_\fm\otimes_{\cO_{\bE_{\lambda}}}\Lambda,
\end{equation}
and it is a finite free $\Lambda$-module. Note that we have
\[
\rH^g(\bfSh(G,K^p)_{\FF_p^{\ac}},\QQ_{\ell}^{\ac})_\fm=\bigoplus_{\Pi'\in R_\fm}\rH^g(\bfSh(G,K^p)_{\FF_p^{\ac}},\QQ_{\ell}^{\ac})[\Pi'^\infty]
\]
as modules over $\ZZ[\dT^{\tR\cup\{\fp\}}]$. Then it was shown in the proof of \cite{TX2}*{Theorem~4.4(2)} that on each $\Pi'^\infty$-isotypic component, $\det(\Res(\Lambda)\circ\Gys(\Lambda))$ is equal to a power of
\[
\pm p^{\frac{g-1}{2}\cdot\binom{g}{(g-1)/2}}[(\alpha_{\Pi'}-\beta_{\Pi'})^2/(\alpha_{\Pi'}\beta_{\Pi'})]^{t_{g,(g-1)/2}}
\]
for $\Lambda=\QQ^{\ac}_{\ell}$, where $t_{g,(g-1)/2}=\sum_{i=0}^{(g-1)/2-1}\binom{g}{i}$. By \eqref{E:base-change}, it is clear that the same formula also holds for $\Lambda=\cO_{\bE_{\lambda}}$. Therefore, we see that $\det(\Res(\cO_{\bE_{\lambda}})\circ\Gys(\cO_{\bE_{\lambda}}))$ is non-vanishing modulo $\lambda$ by Definition~\ref{D:level-raising-prime}(L2). It follows that $\Res(\Lambda)\circ \Gys(\Lambda)$ is an isomorphism for all choices of $\Lambda$, and hence $\Gys(\Lambda)$ is injective and (1) follows.

The above argument also implies (2).
\end{proof}

We can now finish the proof of Theorem~\ref{T:level_raising}. The assertion that $\rG'$ acts trivially on $\Gamma(\fB\times\bfSh(G_{\tS_{\max}},K^p)(\bF_p^\ac),\cO_{\bE}/\lambda)_\fm$ follows from Theorem~\ref{T:supersingular-quaternion}(2) and Definition~\ref{D:level-raising-prime}(L3).  We focus  now on the surjectivity of $\Phi_\fm$ \eqref{E:AJ2}.

We write $k_{\lambda}=\cO_{\bE}/\lambda$ for simplicity as before. Via the canonical isomorphism
\[
\Gamma(\fB\times \bfSh(G_{\tS_{\max}},K^p)(\FF_p^{\ac}), k_{\lambda})_\fm\cong \bigoplus_{\fa\in\fB} \Gamma(\bfSh(G_{\tS_{\max}},K^p)_{\FF_p^{\ac}}, k_\lambda)_\fm,
\]
the map \eqref{E:AJ4} is identified with the composite map
\[
\xymatrix{
\bigoplus_{\fa\in\fB} \Gamma(\bfSh(G_{\tS_{\max}},K^p)_{\FF_p^{\ac}},k_\lambda)/\fm \ar[rr]^-{\oplus_{\fa}\Phi_\fm(\fa)/\fm}\ar[rrd]_{\Phi_\fm/\fm}
&& \bigoplus_{\fa\in \fB} \rH^1(\FF_{p^{2g}},\rH^1( \bfSh(G_{\emptyset_{\fa}}, K^p)_{\FF_p^{\ac}}, k_\lambda(1)/\fm)\ar[d]^{\Gys}\\
&& \rH^1(\FF_{p^{2g}}, \rH^g(\bfSh(G,K^p)_{\FF_p^{\ac}}, k_\lambda((g+1)/2))/\fm),
}
\]
where the vertical map $\Gys$ is simply $\rH^1(\FF_{p^{2g}},(\Gys(k_\lambda)/\fm)(1))$. Here, we use the fact that the canonical maps
\begin{align*}
\rH^1(\FF_{p^{2g}},\rH^1( \bfSh(G_{\emptyset_{\fa}}, K^p)_{\FF_p^{\ac}}, k_\lambda(1))/\fm&\to
\rH^1(\FF_{p^{2g}},\rH^1( \bfSh(G_{\emptyset_{\fa}}, K^p)_{\FF_p^{\ac}}, k_\lambda(1)/\fm) \\
\rH^1(\FF_{p^{2g}}, \rH^g(\bfSh(G,K^p)_{\FF_p^{\ac}}, k_\lambda((g+1)/2)))/\fm&\to
\rH^1(\FF_{p^{2g}}, \rH^g(\bfSh(G,K^p)_{\FF_p^{\ac}}, k_\lambda((g+1)/2))/\fm)
\end{align*}
are isomorphisms since $\rH^2(\FF_{p^{2g}},-)$ vanishes. By Proposition~\ref{P:level_raising_curve}, the map $\oplus_{\fa}\Phi_\fm(\fa)/\fm$ is surjective. To prove that $\Phi_\fm/\fm$ is surjective, it suffices to show that so is $\Gys$.

First, we have a description of $\rH^1(\bfSh(G_{\emptyset_{\fa}}, K^p), k_{\lambda}(1))/\fm$ in terms of $\bar\rho_{\Pi,\lambda}$, which is the residue representation of \eqref{eq:galois} as we recall. Since $\bar\rho_{\Pi,\lambda}$ is absolutely irreducible by Remark~\ref{R:assumption-wp}(1), the $k_\lambda[\rG_F]$-module $\rH^1(\bfSh(G_{\emptyset_{\fa}},K^p)_{\QQ^{\ac}},k_\lambda(1))/\fm$ is isomorphic to $r$ copies of $\bar\rho_{\Pi,\lambda}^{\vee}(1)\cong\bar\rho_{\Pi, \lambda}$ with $r\geq\dim(\Pi_B^\infty)^K$ by \cite{BLR91} and the theory of old forms. By Remark~\ref{R:level-raising-prime}(2), one has an isomorphism of $k_{\lambda}[\rG']$-modules
\[
\bar\rho_{\Pi, \lambda}\cong k_\lambda\oplus k_\lambda(1).
\]
In particular, $\rH^1(\bfSh(G_{\emptyset_{\fa}}, K^p)_{\FF_p^{\ac}}, k_{\lambda}(1))/\fm$ is the direct sum of the eigenspaces of  $\sigma_{\fp}^2$ with eigenvalues $1$ and $p^{2g}$ both with multiplicity $r$.

By \cite{BL84}, Assumption~\ref{R:assumption-wp}(4), Remark~\ref{R:level-raising-prime}(3) and the similar argument as above, the (generalized) eigenvalues of $\sigma_{\fp}^2$ on $\rH^g(\bfSh(G,K^p)_{\FF_p^{\ac}},\QQ_{\ell}^{\ac}((g+1)/2))/\fm$ are $p^{g(g+1)}\alpha_{\Pi}^{-2i}\beta_{\Pi}^{-2(g-i)}$ with multiplicity $\binom{g}{i}\dim(\Pi_B^\infty)^K$. Note that $p^{g(g+1)}\alpha_{\Pi}^{-2i}\beta_{\Pi}^{-2(g-i)}$ has image $p^{g(1+2i-g)}$ in $\FF_{\ell}^{\ac}$, which are distinct for different $i$ under Definition~\ref{D:level-raising-prime}(L2). For every $\mu\in k_{\lambda}$, let
\[
(\rH^g(\bfSh(G,K^p)_{\FF_p^{\ac}},k_{\lambda}((g+1)/2))/\fm)^{\sigma_{\fp}^2\approx\mu}
\subseteq\rH^g(\bfSh(G,K^p)_{\FF_p^{\ac}},k_{\lambda}((g+1)/2))/\fm
\]
denote the generalized eigenspace of $\sigma_{\fp}^2$ with eigenvalue $\mu$, that is, the maximal subspace annihilated by $(\sigma_{\fp}^2-\mu)^{\ell^N}$ for $N=1,2,\dots$. Then by the base change property \eqref{E:base-change}, one has a canonical decomposition
\[
\rH^g(\bfSh(G,K^p)_{\FF_p^{\ac}},k_{\lambda}((g+1)/2))/\fm=\bigoplus_{i=0}^{g}
(\rH^g(\bfSh(G,K^p)_{\FF_p^{\ac}},k_\lambda((g+1)/2))/\fm)^{\sigma_{\fp}^2\approx p^{g(1+2i-g)}},
\]
where the $i$-th direct summand has dimension $\binom{g}{i}\dim(\Pi_B^\infty)^K$ over $k_{\lambda}$. The direct summand with $\sigma_{\fp}^2\approx1$ corresponds to the term with $i=(g-1)/2$, and it has dimension $\binom{g}{(g-1)/2}\dim(\Pi_B^\infty)^K$.
Note that
\[
\rH^1(\FF_{p^{2g}},(\rH^g(\bfSh(G,K^p)_{\FF_p^{\ac}}, k_\lambda((g+1)/2))/\fm)^{\sigma_{\fp}^2\approx p^{g(1+2i-g)}})=0\]
for $i\neq (g-1)/2$. It follows that the natural map
\begin{align}\label{eq:gys}
(\rH^g(\bfSh(G,K^p)_{\FF_p^{\ac}}, k_{\lambda}((g+1)/2))/\fm)^{\sigma_{\fp}^2\approx1}\to
\rH^1(\FF_{p^{2g}},\rH^g(\bfSh(G,K^p)_{\FF_p^{\ac}}, k_\lambda((g+1)/2))/\fm)
\end{align}
is surjective. One gets a commutative diagram:
\[
\xymatrix{
\bigoplus_{\fa\in \fB}(\rH^1( \bfSh(G_{\emptyset_{\fa}}, K^p)_{\FF_p^{\ac}}, k_\lambda(1))/\fm)^{\sigma_{\fp}^2=1}
\ar[d]^-{(\Gys(k_{\lambda})/\fm)(1)}_-{\cong}\ar[r]^-{\cong}
& \bigoplus_{\fa\in \fB}\rH^1(\FF_{p^{2g}},\rH^1(\bfSh(G_{\emptyset_{\fa}}, K^p)_{\FF_p^{\ac}}, k_\lambda(1))/\fm)\ar[d]^-{\Gys} \\
(\rH^g(\bfSh(G,K^p)_{\FF_p^{\ac}}, k_{\lambda}((g+1)/2))/\fm)^{\sigma_{\fp}^2\approx 1}\ar[r]^-{\eqref{eq:gys}}
& \rH^1(\FF_{p^{2g}},\rH^g(\bfSh(G,K^p)_{\FF_p^{\ac}}, k_\lambda((g+1)/2))/\fm).
}
\]
Here, $(\Gys(k_{\lambda})/\fm)(1)$ is injective by Proposition~\ref{P:injective-Gysin}(2), and we deduce that it is an isomorphism for dimension reasons. It follows immediately that $\Gys$ is surjective. This finishes the proof of Theorem~\ref{T:level_raising}.

\section{Selmer groups of triple product motives}
\label{ss:5}

In this chapter, we study Selmer groups of certain triple product motives of elliptic curves in the context of the Bloch--Kato conjecture, which can be viewed as an application of the level raising result established in the previous chapter.

From now on, we fix a cubic totally real number field $F$, and let $\TF$ be the normal closure of $F$ in $\bC$.

\subsection{Main theorem}
\label{ss:selmer}

Let $E$ be an elliptic curve over $F$. We have the $\bQ$-motive $\otimes\Ind^F_\bQ\sfh^1(E)$ (with coefficient $\bQ$) of rank $8$, which is the multiplicative induction of the $F$-motive $\sfh^1(E)$ to $\bQ$. The \emph{cubic-triple product motive} of $E$ is defined to be
\begin{align*}
\sfM(E)\coloneqq\(\otimes\Ind^F_\bQ\sfh^1(E)\)(2).
\end{align*}
It is canonically polarized. For every prime $p$, the $p$-adic realization of $\sfM(E)$, denoted by $\sfM(E)_p$, is a Galois representation of $\bQ$ of dimension $8$ with $\bQ_p$-coefficients. In fact, up to a twist, it is the multiplicative induction from $F$ to $\bQ$ of the rational $p$-adic Tate module of $E$.

Now we assume that $E$ is modular. Then it gives rise to an irreducible cuspidal automorphic representation $\Pi_E$ of $(\Res_{F/\bQ}\GL_{2,F})(\bA)$ with trivial central character. In particular, the set $\Sigma(\Pi_E,\tau)$ defined in Section \ref{S:main3} contains $\infty$. We have $L(s,\sfM(E))=L(s+1/2,\Pi_E,\tau)$ (again see Section \ref{S:main3}).

Put $\Delta^\flat\coloneqq\Sigma(\Pi_E,\tau)-\{\infty\}$. Let $\Delta$ (resp.\ $\Delta'$, $\Delta''$) be the set of primes of $F$ above $\Delta^\flat$ that is of degree either $1$ or $3$ (resp.\ unramified of degree $2$, ramified of degree $2$). We write the conductor of $E$ as $\fc\fc'\fc''\fc^+$ such that $\fc$ (resp.\ $\fc'$, $\fc''$, $\fc^+$) has factors in $\Delta$ (resp.\ $\Delta'$, $\Delta''$, elsewhere).

\begin{assumption}\label{as:elliptic_curve}
We consider the following assumptions.
\begin{description}
  \item[(E0)] The cardinality of $\Sigma(\Pi_E,\tau)$ is odd and at least $3$.

  \item[(E1)] For every finite place $w$ of $F$ over some prime in $\Sigma(\Pi_E,\tau)$, the elliptic curve $E$ has either good or multiplicative reduction at $w$.

  \item[(E2)] For distinct embeddings $\tau_1,\tau_2\colon F\hookrightarrow\TF$, the $\TF$-elliptic curve $E\otimes_{F,\tau_1}\TF$ is not isogenous to any (possibly trivial) quadratic twist of $E\otimes_{F,\tau_2}\TF$.

\end{description}
\end{assumption}

\begin{remark}
Assumption \ref{as:elliptic_curve} (E0) implies that $\Delta$ is not empty. Assumption \ref{as:elliptic_curve} (E1) implies that $E$ has multiplicative reduction at $w\in\Delta$. Together, they imply that the geometric fiber $E\otimes_FF^\ac$ does not admit complex multiplication.
\end{remark}

We now assume that $E$ is modular and satisfies Assumption \ref{as:elliptic_curve}. Then Assumption \ref{as:elliptic_curve}(E1) implies that $\fc\fc'$ is square-free, and $\fc''=\cO_F$ by \cite{Liu2}*{Lemma~4.8}. We take an ideal $\fr$ of $\cO_F$ contained in $N\fc^+$ for some integer $N\geq 4$ and coprime to $\Delta^\flat$.

Assumption \ref{as:elliptic_curve}(E0) implies that $\Delta$ is a non-empty finite set of even cardinality.
Let $B$ be a quaternion algebra over $F$, unique up to isomorphism, with ramification set $\Delta$, and $\cO\subseteq B$ be an $\cO_F$-maximal order. Let $\fr_0$ and $\fr_1$ be two ideals of $\cO_F$ such that $\fr_0$, $\fr_1$ and $\Delta$ are mutually coprime. We recall the definition of the Hilbert modular  stack $\cX(\Delta)_{\fr_0,\fr_1}$ over $\Spec(\ZZ[\rN_{F/\QQ}(\fr_0\fr_1)^{-1}(\disc F)^{-1}])$ defined in \cite{Liu2}*{Definition B.3}. For every $\ZZ[\rN_{F/\QQ}(\fr_0\fr_1)^{-1}(\disc F)^{-1}]$-scheme $T$, $\cX(\Delta)_{\fr_0,\fr_1}(T)$ is the groupoid of quadruples $(A,\iota_A,C_A,\alpha_A)$ where
\begin{itemize}
  \item $A$ is a projective abelian scheme over $T$;

  \item $\iota_A\colon\cO\to \End(A)$ is an injective homomorphism satisfying
      \[\Tr(\iota_A(b)|\Lie(A))=\Tr_{F/\QQ}\Tr^{\circ}_{B/F}(b)\]
      for all $b\in \cO$;

  \item $C_A$ is an $\cO$-stable finite flat subgroup of $A[\fr_0]$ which is \'etale locally isomorphic to $(\cO_{F}/\fr_0)^2$ as $\cO/\fr_0\cO\cong \rM_{2}(\cO_F/\fr_0)$-modules,

  \item $\alpha_{A}\colon(\cO_{F}/\fr_1)^2_T\to A$ is an $\cO$-equivariant injective homomorphism of group schemes over $T$.
\end{itemize}
If $\fr_1=\cO_F$,  $\alpha_A$ is trivial and we usually omit it from the notation.
If $\fr_1$ is contained in $N\cO_F$ for some integer $N\geq4$, then $\cX(\Delta)_{\fr_0,\fr_1}$ is a scheme.

We put $\cX_{\fr}\coloneqq\cX(\Delta)_{\fc',\fr}$. Let $\fD(\fr,\fc^+)$ be the set of all ideals of $O_F$ containing $\fr(\fc^+)^{-1}$ as in \cite{Liu2}*{Notation~A.5}. For every $\fd\in\fD(\fr,\fc^+)$, we have the following composite map
\begin{align}\label{eq:correspondence0}
\tilde\delta^\fd\colon\cX_\fr=\cX(\Delta)_{\fc',\fr}\to\cX(\Delta)_{\fc'\fr,\cO_F}\xrightarrow{\delta^\fd}\cX(\Delta)_{\fc'\fc^+,\cO_F}
\end{align}
which is a finite \'{e}tale morphism of Deligne--Mumford stacks, where $\delta^\fd$ is the degeneracy map defined as follows. If $(A,\iota_A,C_A)$ is an object of $\cX(\Delta)_{\fc'\fr, \cO_F}(T)$ for some $\Spec(\ZZ[\rN_{F/\QQ}(\fc'\fr)^{-1}\disc(F)^{-1}])$-scheme $T$, then its image by $\delta^{\fd}$ is given by the objet $(A',\iota_{A'},C_{A'})$, where
\begin{itemize}
  \item  $A'$  is the quotient $A$ by the finite flat subgroup $C_{A}[\fd]$,

  \item $\iota_{A'}$ is the  induced $\cO$-action on $A'$ from $A$,

  \item $C_{A'}$ is the unique subgroup scheme of $C_{A}/C_A[\fd]$ \'etale locally isomorphic to $(\cO_F/\fc'\fc^+)^2$,

\end{itemize}
See \cite{Liu2}*{Section  B.1} for more details.

\begin{remark}
The requirement that $|\Sigma(\Pi_E,\tau)|\geq 3$, that is, $\Delta\neq\emptyset$ is not essential. The reason we require this is \emph{not} to make the relevant Shimura variety $\cX_\fr$ proper. In fact, it is used to obtain a refinement (Proposition \ref{pr:refinement}) of Theorem \ref{T:level_raising} so that the map \eqref{E:AJ4} is also \emph{injective} in order to deduce Lemma \ref{le:center} which is needed for the \emph{first explicit reciprocity law} back in \cite{Liu2}, through a trick using Jacquet--Langlands correspondence. However, it is not clear to us what are optimal conditions for the map \eqref{E:AJ4} to be injective.
\end{remark}

From now on, we fix an element $\fw\in\Delta$. Let $\cB$ be the totally definite quaternion algebra over $F$, ramified exactly at $\Delta\setminus\{\fw\}$. Put
\begin{align*}
\cY_\fr\coloneqq\cB^\times\backslash\widehat\cB^\times/K_{0,1}(\fw\fc',\fr)
\end{align*}
where $K_{0,1}(\fw\fc',\fr)\subseteq\widehat\cB^\times$ is an open compact subgroup defined similarly as in Example \ref{ex:level}.

For every ideal $\fs$ contained in $\fc^+$, we let $\tR(\fs)$ be the union of primes dividing $\fs$ and primes above $\Delta^\flat$. In particular, we have the homomorphism
\[
\phi^\fs\coloneqq\phi^{\tR(\fs)}_{\Pi_E}\colon\bZ[\dT^{\tR(\fs)}]\to\bZ
\]
such that $\phi^\fs(\rT_\fq)=a_\fq(E)$ and $\phi^\fs(\rS_\fq)=1$ for every prime $\fq\not\in\tR(\fs)$. Here we recall that $\dT^\tR$ is the Hecke monoid away from $\tR$ \cite{Liu2}*{Notation~3.1}.

Let $p$ be a rational prime\footnote{The readers may notice that we switch the roles of $p$ and $\ell$ (or $\lambda$) in Section \ref{ss:5} from Section \ref{ss:4}. This is due to a different convention in the study of Selmer groups.}. Let $\fm^\fs_p$ be the kernel of the composite map $\bZ[\dT^{\tR(\fs)}]\xrightarrow{\phi^\fs}\bZ\to\bF_p$. We also have an induced Galois representation
\[
\rho_{\Pi_E,p}\colon\rG_F\to\GL(T_p(E))\cong\GL_2(\bZ_p),
\]
where $T_p(E)$ is the $p$-adic Tate module of $E$. Put $\bar\rho_{\Pi_E,p}\coloneqq\rho_{\Pi_E,p}\mod p$.

\begin{definition}[Perfect pair]\label{de:perfect_pair}
We say that
\begin{enumerate}
  \item $p$ is \emph{generic} if $(\Ind^\bQ_F\bar\rho_{\Pi_E,p})\res_{\rG_{\TF}}$ has the largest possible image, which is isomorphic to $\rG(\SL_2(\bF_p)\times\SL_2(\bF_p)\times\SL_2(\bF_p))$;

  \item the pair $(p,\fr)$ is \emph{$\fs$-clean}, for an ideal $\fs$ of $\cO_F$ contained in $\fr$, if
      \begin{enumerate}
        \item the canonical map $\Gamma(\cY_\fr,\bZ_p)/\fm^\fs_p$ has dimension $|\fD(\fr,\fc^+)|$ over $\bF_p$;

        \item $\rH^3(\cX(\Delta)_{\fc'\fc^+,\cO_F}\otimes\bQ^\ac,\bZ_p)/{\fm^\fs_p}$ has dimension $8$ over $\bF_p$, and the canonical map
        \[
        \bigoplus_{\fd\in\fD(\fr,\fc^+)}\tilde\delta^\fd_*\colon\rH^3(\cX_\fr\otimes\bQ^\ac,\bZ_p)/{\fm^\fs_p}
        \to\bigoplus_{\fd\in\fD(\fr,\fc^+)}\rH^3(\cX(\Delta)_{\fc'\fc^+,\cO_F}\otimes\bQ^\ac,\bZ_p)/{\fm^\fs_p}
        \]
        is an isomorphism;
      \end{enumerate}

  \item the pair $(p,\fr)$ is \emph{perfect} if
      \begin{enumerate}
        \item $p\geq 11$ and $p\neq 13,19$;

        \item $p$ is coprime to $\Delta^\flat$ and $\fr\cdot |(\ZZ/\fr\cap \ZZ)^{\times}|\cdot\mu(\fr,\fc^+)\cdot|\Cl(F)_\fr|\cdot\disc F $, where $\disc F$ is the discriminant of $F$,  $\Cl(F)_\fr$ is the ray class group of $F$ with respect to $\fr$, and
        \[\mu(\fr,\fc^+)=\rN_{F/\QQ}(\fr(\fc^{+})^{-1})\prod_{\fq}\bigg(1+\frac{1}{\rN_{F/\QQ}(\fq)}\bigg);\] with $\fq$ running through the prime ideals of $\cO_F$  dividing $\fr$ but not $\fc^+$;

        \item $p$ is generic;

        \item it is $\fr$-clean; and

        \item $\bar\rho_{\Pi_E,p}$ is ramified at $\fw$.
      \end{enumerate}
\end{enumerate}
\end{definition}

\begin{remark}\label{R:freeness-cohomology}
Note that the condition that $p$ is generic implies that the condition $(\bL\bI_{\mathrm{Ind}_{\bar\rho_{\Pi_E, p}}})$ in \cite{Dim05}*{Proposition~0.1} is satisfied. Consequently, $\rH^3(\cX_{\fr}\otimes \QQ^{\ac}, \ZZ_p)_{\fm_p^{\fs}}$ is finite free over $\ZZ_p$ for any ideal $\fs$ of $\cO_F$ containing $\fr$ by \cite{Dim05}*{Theorem~0.3}.
\end{remark}

Let $B^{\flat}$ be a quaternion algebra over $\QQ$, unique up to isomorphisms, with ramification set $\Delta^{\flat}$ so that $B\cong B^{\flat}\otimes_{\QQ} F$. We have similarly  a moduli scheme $\cX_\fr^\flat\coloneqq\cX(\Delta^\flat)_{\bZ,\fr\cap\bZ}$ attached to $B^{\flat}$. Then we obtain a canonical morphism
\begin{align*}\label{eq:special}
\theta\colon\cX_\fr^\flat\to\cX_\fr
\end{align*}
over $\bZ[(\fr\disc F)^{-1}]$ similar to \cite{Liu2}*{(4.1.1)}. It is a finite morphism. Denote by $\Theta_{p,\fr}$ the image of $\theta_*[\cX_\fr^\flat\otimes\bQ]\in\CH^2(\cX_\fr\otimes\bQ)$ under the Abel--Jacobi map
\[
\AJ_p\colon\CH^2(\cX_\fr\otimes\bQ)\to\rH^1(\bQ,\rH^3(\cX_\fr\otimes\bQ^\ac,\bQ_p(2))/\Ker\phi^\fr).
\]

By \cite{Liu2}*{Lemma~4.6}, we have $\rH^1(\bQ_v,\sfM(E)_p)=0$ for all primes $v\nmid p$. Thus, we recall the following definition.

\begin{definition}[\cite{BK90} or \cite{Liu2}*{Definition~4.7}]\label{de:selmer}
The \emph{Bloch--Kato Selmer group} for the representation $\sfM(E)_p$ is the subspace $\rH^1_f(\bQ,\sfM(E)_p)$ consisting of classes $s\in\rH^1(\bQ,\sfM(E)_p)$ such that
\[\loc_p(s)\in\rH^1_f(\bQ_p,\sfM(E)_p)\coloneqq\Ker[\rH^1(\bQ_p,\sfM(E)_p)\to\rH^1(\bQ_p,\sfM(E)_p\otimes_{\bQ_p}\rB_{\r{cris}})].\]
\end{definition}


\begin{theorem}\label{th:selmer}
Let $E$ be a modular elliptic curve over $F$ satisfying Assumption \ref{as:elliptic_curve}. For a rational prime $p$, if there exists a perfect pair $(p,\fr)$ (Definition \ref{de:perfect_pair}) such that $\Theta_{p,\fr}\neq 0$, then
\[\dim_{\bQ_p}\rH^1_f(\bQ,\sfM(E)_p)=1.\]
\end{theorem}

\begin{remark}\label{re:perfect}
By a similar argument of \cite{Liu2}*{Lemma~4.10}, given an ideal $\fr$ of $\cO_F$ contained in $N\fc^+$ for some integer $N\geq 4$ and coprime to $\Delta^\flat$, there exists a finite set $\cP_{E,\fr}$ of rational primes such that $(p,\fr)$ is a perfect pair for every $p\not\in\cP_{E,\fr}$. An upper bound for $\cP_{E,\fr}$ can be computed effectively.
\end{remark}

\begin{remark}
Assuming the (conjectural) triple product version of the Gross--Zagier formula and the Beilinson--Bloch conjecture on the injectivity of the Abel--Jacobi map, the following two statements should be equivalent:
\begin{itemize}
  \item $L'(0,\sfM(E))\neq 0$ (note that $L(0,\sfM(E))=0$ by Assumption \ref{as:elliptic_curve}(E0)); and
  \item there exists some $\fr_0$ such that for every other $\fr$ contained in $\fr_0$, we have $\Theta_{p,\fr}\neq 0$ as long as $(p,\fr)$ is a perfect pair.
\end{itemize}
Here, we need to use (the proof of) \cite{Liu2}*{Proposition~4.9}. Then Theorem \ref{th:selmer} implies that if $L'(0,\sfM(E))\neq 0$, that is, $\ord_{s=0}L(s,\sfM(E))=1$, then $\dim_{\bQ_p}\rH^1_f(\bQ,\sfM(E)_p)=1$ for all but finitely many $p$.
\end{remark}

\subsection{A refinement of arithmetic level raising}
\label{ss:refinement}

For now on, we fix a perfect pair $(p,\fr)$ (Definition \ref{de:perfect_pair}), and put $\fm^\fs\coloneqq\fm^\fs_p$ for short.

\begin{definition}\label{de:admissible}
Let $\nu\geq 1$ be an integer. We say that a prime $\ell$ is \emph{$(p^\nu,\fr)$-admissible} if
\begin{description}
  \item[(A1)] $\ell$ is inert in $F$ (with $\fl=\ell\cO_F$), unramified in $\TF$, and coprime $\tR(\fr)\cup\{2,p\}$;

  \item[(A2)] $(p,\fr)$ is $\fr\fl$-clean;

  \item[(A3)] $p\nmid(\ell^{18}-1)(\ell^6+1)$;

  \item[(A4)] $\phi^\fr(\rT_\fl)\equiv\ell^3+1\mod p^\nu$.
\end{description}
\end{definition}

\begin{notation}\label{N:notation-rho}
For now on, we fix an integer $\nu\geq 1$ and put $\Lambda\coloneqq\bZ/p^\nu$. Let $\rho\colon\rG_F\to\GL(\rN_\rho)$ be the reduction of $\rho_{\Pi_E,p}$ modulo $p^\nu$, where $\rN_\rho=T_p(E)\otimes\Lambda$. We have the multiplicatively induced representation $\rho^\sharp\colon\rG_\bQ\to\GL(\rN_\rho^\sharp)$ with $\rN_\rho^\sharp=\rN_\rho^{\otimes3}$.
\end{notation}

\begin{lem}\label{L:rank-1}
   Let $\ell$ be a $(p^\nu,\fr)$-admissible prime. Then the cohomology group $\rH^1_\unr(\bQ_\ell,\rH^3(\cX(\Delta)_{\fc'\fc^+,\cO_F}\otimes\bQ^\ac,\Lambda(2))/\Ker\phi^{\fr\fl})$ (resp.\ $\rH^1_\unr(\bQ_\ell,\rH^3(\cX_\fr\otimes\bQ^\ac,\Lambda(2))/\Ker\phi^{\fr\fl})$) is a free $\Lambda$-module of rank $1$ (resp.\ $|\fD(\fr,\fc^+)|$);
\end{lem}
\begin{proof}
By Definition~\ref{de:admissible}(A2), the Nakayama lemma and \cite{BL84}, we have isomorphisms of $\Lambda[\rG_{\bQ_\ell}]$-modules:
\begin{align*}
\rH^3(\cX(\Delta)_{\fc'\fc^+,\cO_F}\otimes\bQ_\ell^\ac,\Lambda(2))/\Ker\phi^{\fr\fl}&\cong\rN_\rho^\sharp(-1),\\
\rH^3(\cX_\fr\otimes\bQ_\ell^\ac,\Lambda(2))/\Ker\phi^{\fr\fl}&\cong\rN_\rho^\sharp(-1)^{\oplus|\fD(\fr,\fc^+)|}.
\end{align*}
If $\sigma_{\fl}\in \rG_F$ denotes an arithmetic Frobenius element at $\fl$, then $\rho(\sigma_{\fl})$ is conjugate to $\big(\begin{smallmatrix}1 & 0\\ 0& \ell^3\end{smallmatrix}\big)$ by Definition~\ref{de:admissible}(A4).
Hence, the $\Lambda[\rG_{\bQ_\ell}]$-module $\rN_\rho^\sharp(-1)$ is unramified and isomorphic to $\Lambda(-1)\oplus\Lambda\oplus \rR\oplus \Lambda(1)\oplus \rR(1)\oplus\Lambda(2)$, where $\rR\cong \Lambda^{\oplus 2}$ is the rank $2$ unramified representation of $\rG_{\QQ_{\ell}}$ with the action of the arithmetic Frobenius $\sigma_{\ell}$ given by $\big(\begin{smallmatrix}0 & -1 \\ 1&-1\end{smallmatrix}\big)$. By Definition~\ref{de:admissible}(A3), it follows that $\rH^1_{\unr}(\QQ_{\ell}, \rN^{\sharp}_{\rho}(-1))\cong\rH^1_{\unr}(\QQ_{\ell},\Lambda)$, which is free of rank 1 over $\Lambda$.

\end{proof}

Let $\ell$ be a $(p^\nu,\fr)$-admissible prime. Then $\cX_\fr\otimes\bZ_{(\ell)}$ is canonically isomorphic to $\bfSh(G,K_{0,1}(\fc',\fr)^\ell)$ with $G=\Res_{F/\bQ}B^\times$ considered in Section \ref{S:quaternion-PEL} (See Remark~\ref{R:polarization-free} on the issue of polarizations and Example \ref{ex:level} for the open compact subgroup  $K_{0,1}(\fc',\fr)$), and   $\cX_\fr^\flat\otimes\bZ_{(\ell)}$ is canonically isomorphic to $\bfSh(G^\flat,K_{0,1}(\bZ,\fr\cap\bZ)^\ell)$ with $G^\flat=(B^\flat)^\times$.
Put $X_\fr\coloneqq\cX_\fr\otimes\bF_\ell$. As before, we denote by $X_\fr^{\r{sp}}$ the superspecial locus of $X_\fr$. By Theorem~\ref{P:superspecial-quaternion}, we may identify $X_\fr^{\r{sp}}(\bF_\ell^\ac)$ with $\bfSh(G_{\tS_{\max}},K_{0,1}(\fc',\fr)^\ell)(\bF_\ell^\ac)$.

The following proposition is a refinement of Theorem~\ref{T:level_raising} in our situation.

\begin{proposition}\label{pr:refinement}
Let $\ell$ be a $(p^\nu,\fr)$-admissible prime. Then the level raising map
\begin{equation}\label{E:level-raising-refine}
\Gamma(\fB\times X_\fr^{\r{sp}}(\bF_\ell^\ac),\Lambda)/\Ker\phi^{\fr\fl}\to
\rH^1(\bF_{\ell^6},\rH^3(X_\fr\otimes\bF_\ell^\ac,\Lambda(2))/\Ker\phi^{\fr\fl})
\end{equation}
 defined similarly as \eqref{E:AJ4} is an isomorphism.
\end{proposition}

\begin{proof}
In the proof of Lemma~\ref{L:rank-1}, we have seen that, as a $\Lambda[\rG_{\FF_{\ell}}]$-module,  $\rH^3(X_\fr\otimes\bF_\ell^\ac,\Lambda(2))/\Ker\phi^{\fr\fl}$ is isomorphic to $|\fD(\fr,\fc^+)|$-copies of
\[
\rN^{\sharp}_{\rho}(-1)\cong \Lambda(-1)\oplus\Lambda\oplus \rR\oplus \Lambda(1)\oplus \rR(1)\oplus\Lambda(2).
\]
We get thus an isomorphism of $\Lambda[\Gal(\FF_{\ell^6}/\FF_{\ell})]$-modules:
\begin{equation}\label{E:galois-cohomology}
\rH^1(\bF_{\ell^6},\rH^3(X_\fr\otimes\bF_\ell^\ac,\Lambda(2))/\Ker\phi^{\fr\fl})\cong \rH^1(\FF_{\ell^6}, \Lambda\oplus \rR)^{\oplus |\fD(\fr,\fc^+)|}\cong (\Lambda\oplus \rR)^{\oplus |\fD(\fr,\fc^+)|},
\end{equation}
which is  free of rank  $3|\fD(\fr,\fc^+)|$ over $\Lambda$. By Theorem \ref{T:level_raising} and the Nakayama Lemma, the map \eqref{E:level-raising-refine} is surjective. Thus it suffices to show that $\Gamma(X_\fr^{\r{sp}}(\bF_\ell^\ac),\Lambda)/\Ker\phi^{\fr\fl}$ is a free $\Lambda$-module of rank $|\fD(\fr,\fc^+)|$. By the Nakayama lemma, it suffices to show that $\Gamma(X_\fr^{\r{sp}}(\bF_\ell^\ac),\bF_p)/\fm^{\fr\fl}$ has dimension $|\fD(\fr,\fc^+)|$ over $\bF_p$.

Recall that so far, we have three quaternion algebras over $F$ in the story: $\cB$ ramified at $\Sigma_\infty\cup\Delta\setminus\{\fw\}$, $B$ ramified at $\Delta$, and $B_{\tS_{\max}}$ ramified at $\Sigma_\infty\cup\{\fl\}\cup\Delta$. Now we let $B'$ be the fourth quaternion algebra over $F$ ramified at $\Sigma\cup\{\fl\}\cup\Delta\setminus\{\fw\}$ where $\Sigma$ is a fixed subset of $\Sigma_\infty$ of cardinality $2$. Let $C$ be the corresponding proper Shimura curve over $F$ (with the embedding into $\bQ^\ac$ given by the unique element in $\Sigma_\infty\setminus\Sigma$) of the similarly defined level $K_{0,1}(\fw\fc',\fr)$. As in Step 4 of the proof of \cite{Liu2}*{Proposition~3.32}, $C$ has a natural strictly semistable model at $\fl$. The corresponding weight spectral sequence provides us with a canonical isomorphism
\[
\Gamma(\cY_\fr,\bZ_p)/\fm^{\fr\fl}\simeq\rH^1_\sing(\bQ_{\ell^6},\rH^1(C\otimes\bQ^\ac,\bZ_p)/\fm^{\fr\fl})
\]
as in the proof of \cite{Liu2}*{Proposition~3.32}. By Definition~\ref{de:admissible}(A2), $\rH^1_\sing(\bQ_{\ell^6},\rH^1(C\otimes\bQ^\ac,\bZ_p)/\fm^{\fr\fl})$ has dimension $|\fD(\fr,\fc^+)|$. By \cite{BLR91}, we conclude that $\rH^1(C\otimes\bQ^\ac,\bZ_p)/\fm^{\fr\fl}$ is isomorphic to $\bar\rho_{\Pi_E,p}^{\oplus|\fD(\fr,\fc^+)|}$ as an $\bF_p[\rG_F]$-module. In particular, $\rH^1(C\otimes\bQ^\ac,\bZ_p)/\fm^{\fr\fl}$ has dimension $2|\fD(\fr,\fc^+)|$.

Now consider the semistable reduction of $C$ at $\fw$. Let $C_0$ be the proper Shimura curve over $F$ associated to $B'$ of the level $K_{0,1}(\fc',\fr)$. Then $\rH^1(C_0\otimes\bQ^\ac,\bZ_p)/\fm^{\fr\fl}=0$ by Definition \ref{de:perfect_pair}(3e). Therefore, we have a canonical isomorphism
\[
\rH^1(\rI_\fw,\rH^1(C\otimes\bQ^\ac,\bZ_p)/\fm^{\fr\fl})\simeq\Gamma(X_\fr^{\r{sp}}(\bF_\ell^\ac),\bF_p)/\fm^{\fr\fl}
\]
from the weight spectral sequence, as the supersingular set of $C$ at $\fw$ is also $X_\fr^{\r{sp}}(\bF_\ell^\ac)$. Therefore, $\Gamma(X_\fr^{\r{sp}}(\bF_\ell^\ac),\bF_p)/\fm^{\fr\fl}$ has dimension $|\fD(\fr,\fc^+)|$. The proposition follows.
\end{proof}

\subsection{Second explicit reciprocity law}
\label{ss:second}

Let $\ell$ be a $(p^{\nu}, \fr)$-admissible prime, and $\fl=\ell\cO_F$.  Recall that $\Sigma_{\infty}$ denotes the set of archimedean places of $F$.
For every ideal $\fs$ of $\cO_F$ coprime to $\Delta\cup\{\fl\}$, let  $\cS_{\ell,\fs}\coloneqq\cS(\Sigma_\infty\cup\Delta\cup\{\fl\})_{\fs}$ be the set of isomorphism classes of oriented $\cO_F$-Eichler order of discriminant $\Sigma_\infty\cup\Delta\cup\{\fl\}$ and  level $\fs$ (see \cite{Liu2}*{Definition~A.1}). It has an action by $\rG_{\bF_{\ell}}$ such that the arithmetic Frobenius  $\sigma_{\ell}$ acts by switching the orientation at $\fl$.

\begin{lem}
There is a canonical isomorphism $X_\fr^{\r{sp}}(\bF_\ell^\ac)/\Cl(F)_\fr\cong\cS_{\ell,\fc'\fr}$. Moreover, the induced action of $\rG_{\bF_\ell}$ on $\cS_{\ell,\fc'\fr}$ factors through $\Gal(\bF_{\ell^2}/\bF_\ell)$ and is given by the map $\op_\ell$ switching the orientation at $\fl$.
\end{lem}

\begin{proof}
It is a special case of \cite{Liu2}*{Proposition~A.13(1)}.
\end{proof}

Denote by $\psi\colon X_\fr^{\r{sp}}(\bF_\ell^\ac)\to\cS_{\ell,\fc'\fr}$ the canonical projection from the above lemma.

\begin{lem}\label{le:ribet}
The canonical map
\[
\psi^*\colon\Gamma(\cS_{\ell,\fc'\fr},\Lambda)/\Ker\phi^{\fr\fl}\to\Gamma(X_\fr^{\r{sp}}(\bF_\ell^\ac),\Lambda)/\Ker\phi^{\fr\fl}
\]
is an isomorphism.
\end{lem}

\begin{proof}
It follows similarly as \cite{Liu2}*{Lemma~3.24}.
\end{proof}

\begin{proposition}\label{pr:level_raising}
Under the notation above, the following statements hold:
\begin{enumerate}

  \item The action of $\op_\ell$ on $\Gamma(\cS_{\ell,\fc'\fr},\Lambda)/\Ker\phi^{\fr\fl}$ is trivial.
  \item There exists a unique isomorphism $\Phi$ such that the following diagram
    \[\xymatrix{
    \Gamma(\cS_{\ell,\fc'\fr},\Lambda)/\Ker\phi^{\fr\fl}\ar[r]^-{\Phi}
    \ar[d]^{\psi^*}&\rH^1_\unr(\bQ_{\ell},\rH^3(\cX_\fr\otimes\bQ^\ac,\Lambda(2)/\Ker\phi^{\fr\fl})\ar[d]^{\cong}\\
    \Gamma(X_\fr^{\r{sp}}(\bF_\ell^\ac),\Lambda)/\Ker\phi^{\fr\fl}\ar[d] & \rH^1_\unr(\bQ_{\ell^6},\rH^3(\cX_\fr\otimes\bQ^\ac,\Lambda(2)/\Ker\phi^{\fr\fl}))^{\Gal(\QQ_{\ell^6}/\QQ_{\ell})}\ar@{^(->}[d]\\
    \Gamma(\fB\times X_\fr^{\r{sp}}(\bF_\ell^\ac),\Lambda)/\Ker\phi^{\fr\fl} \ar[r]^-{\eqref{E:level-raising-refine}}  & \rH^1_\unr(\bQ_{\ell^6},\rH^3(\cX_\fr\otimes\bQ^\ac,\Lambda(2)/\Ker\phi^{\fr\fl}))
   } \]
   is commutative, where the lower left vertical arrow is the diagonal map.
\end{enumerate}
\end{proposition}

\begin{proof}
Consider the action of $\Gal(\bQ_{\ell^6}/\bQ_\ell)$ on both sides of the isomorphism
\[
\Gamma(\fB\times \cS_{\ell,\fc'\fr},\Lambda)/\Ker\phi^{\fr\fl}\xrightarrow{\psi^*}
\Gamma(\fB\times X_\fr^{\r{sp}}(\bF_\ell^\ac),\Lambda)/\Ker\phi^{\fr\fl}\to
\rH^1(\bF_{\ell^6},\rH^3(X_\fr\otimes\bF_\ell^\ac,\Lambda(2))/\Ker\phi^{\fr\fl}).
\]
By \eqref{E:galois-cohomology},
we obtain an isomorphism
\[(\Gamma(\cS_{\ell,\fc'\fr},\Lambda)/\Ker\phi^{\fr\fl})^{\op_\ell=1}
\cong\rH^1_\unr(\bQ_\ell,\rH^3(\cX_\fr\otimes\bQ^\ac,\Lambda(2)/\Ker\phi^{\fr\fl}).
\]
By Lemma~\ref{L:rank-1}, $\rH^1_\unr(\bQ_\ell,\rH^3(\cX_\fr\otimes\bQ^\ac,\Lambda(2)/\Ker\phi^{\fr\fl})$ is a free $\Lambda$-module of rank $|\fD(\fr,\fc^+)|$. Therefore, the inclusion
\[
(\Gamma(\cS_{\ell,\fc'\fr},\Lambda)/\Ker\phi^{\fr\fl})^{\op_\ell=1}\subseteq\Gamma(\cS_{\ell,\fc'\fr},\Lambda)/\Ker\phi^{\fr\fl}
\]
is an isomorphism as both sides are free $\Lambda$-module of rank $|\fD(\fr,\fc^+)|$. Thus both (1) and (2) follow.
\end{proof}

Denote by $\Theta_{p,\fr}^\nu$ the image of $\theta_*[\cX_\fr^\flat\otimes\bQ]\in\CH^2(\cX_\fr\otimes\bQ)$ under the Abel--Jacobi map
\[
\AJ_p\colon\CH^2(\cX_\fr\otimes\bQ)\to\rH^1(\bQ,\rH^3(\cX_\fr\otimes\bQ^\ac,\Lambda(2))/\Ker\phi^{\fr\fl}).
\]

For any ideal $\fs\subseteq \cO_F$, let  $\cS^\flat_{\ell,\fs}=\cS(\{\infty\}\cup\Delta^\flat\cup\{\ell\})_{\fs\cap\bZ}$ denote the set of isomorphism classes of oriented $\ZZ$-Eichler order of discriminant $\{\infty\}\cup\Delta^\flat\cup\{\ell\}$  and level $\fs\cap\ZZ$ (\cite{Liu2}*{Definition~A.1}). We have a natural map given by extension of scalars
\begin{align}\label{eq:special1}
\vartheta\colon \cS^\flat_{\ell,\fr}\to \cS_{\ell,\fc'\fr}.
\end{align}
We have a bilinear pairing $(\;,\;)\colon\Gamma(\cS_{\ell,\fc'\fr},\bZ)\times\Gamma(\cS_{\ell,\fc'\fr},\bZ)\to\bZ$
defined by the formula $(f_1,f_2)=\sum_{h\in\cS_{\ell,\fc'\fr}}f_1(h)f_2(h)$. It induces a perfect pairing
\[
(\;,\;)\colon\Gamma(\cS_{\ell,\fc'\fr},\Lambda)/\Ker\phi^{\fr\fl}\times\Gamma(\cS_{\ell,\fc'\fr},\Lambda)[\Ker\phi^{\fr\fl}]\to\Lambda.
\]

\begin{theorem}[Second explicit reciprocity law]\label{th:second}
Let $\ell$ be an $(p^\nu,\fr)$-admissible prime. Then $\loc_{\ell}(\Theta_{p,\fr}^{\nu})$ lies in $\rH^1_{\unr}(\QQ_{\ell},\rH^3(\cX_\fr\otimes\bQ^\ac,\Lambda(2))/\Ker\phi^{\fr\fl} )$, and we have
\[
(\Phi^{-1}\loc_\ell\Theta_{p,\fr}^\nu,f)=\frac{|(\bZ/\fr\cap\bZ)^\times|}{(\ell-1)^2|\Cl(F)_\fr|}\cdot
\sum_{x\in \cS^\flat_{\ell,\fr}}f(\vartheta(x))
\]
for every $f\in\Gamma(\cS_{\ell,\fc'\fr},\Lambda)[\Ker\phi^{\fr\fl}]$. Here, $\Phi$ is the isomorphism in Proposition \ref{pr:level_raising}.
\end{theorem}

We note that $(\ell-1)^2|\Cl(F)_\fr|$ is invertible in $\Lambda$.

\begin{proof}
The fact that $\Theta_{p, \fr}^{\nu}$ is unramified follows from the fact that both $\fX_{\fr}$ and $\fX^{\flat}_{\fr}$ have good reduction at $\ell$.
Recall that  $X_\fr=\cX_\fr\otimes\bF_\ell$. Similarly, we put $X_\fr^\flat\coloneqq\cX_\fr^\flat\otimes\bF_\ell$. Then we have the morphism $\theta\colon X_\fr^\flat\to X_\fr$ over $\bF_\ell$. Let $\bar\Theta$ be the image of $\theta_*[X_{\fr}^{\flat}]\in \CH^2(X_{\fr})$ in $\rH^1(\bF_\ell,\rH^3(X_\fr\otimes\bF_\ell^\ac,\Lambda(2)/\Ker\phi^{\fr\fl})$ defined similarly as for $\Theta_{p,\fr}^\nu$. Then under the canonical identification
\[
\rH^1(\bF_\ell,\rH^3(\cX_\fr\otimes\bF_\ell^\ac,\Lambda(2))/\Ker\phi^{\fr\fl})
\cong\rH^1_\unr(\bQ_\ell,\rH^3(\cX_\fr\otimes\bQ^\ac,\Lambda(2))/\Ker\phi^{\fr\fl}),
\]
$\bar\Theta$ coincides with $\loc_\ell\Theta_{p,\fr}^\nu$.

From Proposition \ref{pr:refinement}, we have an isomorphism
\[
\Gamma(\fB\times X_{\fr}^{\r{sp}}(\bF_{\ell}^{\ac}), \Lambda)/\Ker\phi^{\fr\fl}=\bigoplus_{\fa\in \fB}\Gamma(X_\fr^{\r{sp}}(\bF_\ell^\ac),\Lambda)/\Ker\phi^{\fr\fl}
\xrightarrow{\cong}\rH^1(\bF_{\ell^6},\rH^3(X_\fr\otimes\bF_\ell^\ac,\Lambda(2))/\Ker\phi^{\fr\fl}).
\]
For each $\fa\in\fB$, we denote by
\[
\Psi_\fa\colon\rH^1(\bF_{\ell^6},\rH^3(X_\fr\otimes\bF_\ell^\ac,\Lambda(2))/\Ker\phi^{\fr\fl})
\to\Gamma(X_\fr^{\r{sp}}(\bF_\ell^\ac),\Lambda)/\Ker\phi^{\fr\fl}
\]
the map obtained by taking the inverse of the previous isomorphism followed by the canonical projection to the direct summand indexed by $\fa$. By a similar proof of \cite{Liu2}*{Proposition~4.3}, we have the following commutative diagram
\[
\xymatrix{
X_\fr^{\flat,\r{sp}}(\bF_\ell^\ac) \ar[r]^-{\theta}\ar[d]_-{\psi^\flat} &  X_\fr^{\r{sp}}(\bF_\ell^\ac) \ar[d]^-{\psi} \\
\cS^\flat_{\ell,\fr} \ar[r]^-{\vartheta} & \cS_{\ell,\fc'\fr}
}
\]
where $\psi^\flat$ is obtained similarly as $\psi$, but for $X_\fr^\flat$. Therefore, the theorem will follow if we can show that for every $f\in\Gamma(X_\fr^{\r{sp}}(\bF_\ell^\ac),\Lambda)[\Ker\phi^{\fr\fl}]$, we have
\begin{align}\label{eq:reciprocity0}
(\Psi_\fa\bar\Theta,f)=\frac{1}{(\ell-1)^2}\sum_{x\in X_\fr^{\flat,\r{sp}}(\bF_\ell^\ac)}f(\theta(x)).
\end{align}

For every $\fa\in\fB$, we have the following commutative diagram as \eqref{eq:reduction}
\begin{align*}
\xymatrix{
& W_{\emptyset}(\fa)\ar@{^(->}[r]\ar[d] & Z_{\emptyset}(\fa)\ar@{^(->}[r]^-{i_\fa} \ar[d]^{\pi_{\fa}} &
\bfSh(G)_{\FF_{\ell^6}}\cong X_\fr\otimes\bF_{\ell^6} \\
X_\fr^{\r{sp}}\otimes\FF_{\ell^6} \ar[r]^-{\cong} & \bfSh(G_{\tS_{\max}})_{\FF_{\ell^6}} \ar@{^(->}[r] ^-{j_\fa} &\bfSh(G_{\emptyset_{\fa}, \tT_{\fa}})_{\FF_{\ell^6}},\\
}
\end{align*}
where the square is Cartesian.
Here, we omit the away-from-$\ell$ level structure $K_{0,1}(\fc',\fr)^{\ell}$ in the notation. However, in this case, $Z_{\emptyset}(\fa)$ coincides with the Goren--Oort divisor $\bfSh(G)_{\FF_{\ell^6},\tau(\fa)}$ for some $\tau(\fa)\in\Sigma_\infty$ determined by $\fa$. Thus it is easy to see that the (scheme-theoretical) intersection $\Gamma_\theta\cap\pr_2^*Z_{\emptyset}(\fa)$ is contained in $X_\fr^{\flat,\r{sp}}\times X_\fr^{\r{sp}}$, where $\Gamma_\theta\subseteq X_\fr^\flat\times X_\fr$ is the graph of $\theta$ and $\pr_2\colon X_\fr^\flat\times X_\fr\to X_\fr$ is the canonical projection. More precisely, it is the graph of the restricted morphism $\theta\colon X_\fr^{\flat,\r{sp}}\to X_\fr^{\r{sp}}$. Therefore, we have
\begin{align}\label{eq:reciprocity}
\pi_{\fa*}i_\fa^*\theta_*[X_\fr^\flat]=\theta_{\fa*}[X_\fr^{\flat,\r{sp}}\otimes\bF_{\ell^6}]
\end{align}
in $\CH^1(\bfSh(G_{\emptyset_{\fa}})_{\FF_{\ell^6}})$, where $\theta_\fa$ is the composite morphism
\[
X_\fr^{\flat,\r{sp}}\otimes\bF_{\ell^6}\xrightarrow{\theta}X_\fr^{\r{sp}}\otimes\bF_{\ell^6}\cong\bfSh(G_{\tS_{\max}})_{\FF_{\ell^6}}
\xrightarrow{j_\fa}\bfSh(G_{\emptyset_{\fa},\tT_{\fa}})_{\FF_{\ell^6}}.
\]

Recall that we have two morphisms
\begin{align*}
\Gys_{\fa}&=i_{\fa!}\circ\pi_\fa^*\colon \rH^1(\bfSh(G_{\emptyset_{\fa}, \tT_{\fa}})_{\bF_\ell^\ac},\Lambda(1))/\Ker\phi^{\fr\fl}\to
\rH^3(X_\fr\otimes\bF_\ell^\ac,\Lambda(2))/\Ker\phi^{\fr\fl}, \\
\Res_{\fa}&=\pi_{\fa!}\circ i_{\fa}^*\colon \rH^3(X_\fr\otimes\bF_\ell^\ac,\Lambda(2))/\Ker\phi^{\fr\fl} \to \rH^1(\bfSh(G_{\emptyset_{\fa},\tT_{\fa}})_{\bF_\ell^\ac},\Lambda(1))/\Ker\phi^{\fr\fl}.
\end{align*}
We write $\fB=\{\fa_1, \fa_2, \fa_3\}$ with $\fa_{i-1}=\sigma(\fa_{i})$  for all $i$ viewed as elements in $\ZZ/3\ZZ$, where $\sigma(\fa_i)$ means the translate of $\fa_i$ by the Frobenius as defined just above Definition~\ref{D:superspecial}. By \cite{TX2}*{Theorem~4.3} and the proof of \cite{TX2}*{Theorem~4.4}, the intersection matrix $(\Res_{\fa_{i}}\circ\Gys_{\fa_j})_{1\leq i,j\leq 3}$ is given by
\[
\left(
  \begin{array}{ccc}
    -2\ell & \ell\eta_1^{-1} & \ell\eta_3\\
    \ell\eta_{1} & -2\ell & \ell\eta_{2}^{-1}\\
    \ell\eta_{3}^{-1} & \ell \eta_2 &-2 \ell,
  \end{array}
\right)
\]
where $\eta_i\colon\rH^1(\bfSh(G_{\emptyset_{\fa_{i}},\tT_{\fa_i}})_{\bF_\ell^\ac},\Lambda(1))/\Ker\phi^{\fr\fl}\to
\rH^1(\bfSh(G_{\emptyset_{\fa_{i+1}},\tT_{\fa_{i+1}}})_{\bF_\ell^\ac},\Lambda(1))/\Ker\phi^{\fr\fl}$ is a certain normalized link morphism introduced in \cite{TX2}*{Section~2.25} which commutes with the Galois action and such that the product $\eta_{i+2}\eta_{i+1}\eta_i$ for $i\in\ZZ/3\ZZ$ is the endomorphism on $\rH^1(\bfSh(G_{\emptyset_{\fa_{i}},\tT_{\fa_i}})_{\bF_\ell^\ac},\Lambda(1))/\Ker\phi^{\fr\fl}$
given as follows. Let $\sigma_{\ell}\in \rG_{\FF_{\ell}}$ denotes an arithmetic Frobenius element. By \cite{BL84} and Definition~\ref{de:admissible}(A4),  one has a decomposition of $\Lambda[\rG_{\FF_{\ell^{3}}}]$-modules
\[
\rH^1(\bfSh(G_{\emptyset_{\fa_{i}},\tT_{\fa_i}})_{\bF_\ell^\ac},\Lambda(1))/\Ker\phi^{\fr\fl}= \rM_{i}^{1}\oplus \rM_i^{\ell^3},
\]
where each $\rM_i^{\lambda}$ for $\lambda=1,\ell^3$ is a finite free $\Lambda$-module on which the action of $\sigma^3_{\ell}-\lambda$ is nilpotent. Then the action of $\eta_{i+2}\eta_{i+1}\eta_{i}$ on $\rM_{i}^1$ (respectively on $\rM_i^{\ell^3}$) is the multiplication by $\ell^{-3}$ (respectively $\ell^3$). Since the roles of $\fa_i$ are symmetric,    $\rH^1(\bfSh(G_{\emptyset_{\fa_{i}},\tT_{\fa_i}})_{\bF_\ell^\ac},\Lambda(1))/\Ker\phi^{\fr\fl}$ for $i=1,2,3$ must be isomorphic.
Thus, we can identify  $\rM^{\lambda}_i$ with $\lambda=1,\ell^3$ for different $i$ and write it commonly as $\rM^{\lambda}$ in such a way that  $\eta_i$'s are identified with the same endomorphism $\eta$ on $\rM^1\oplus\rM^{\ell^3}$, where  $\eta$ acts by $\ell^{-1}$ on $\rM^1$ and by $\ell$ on $\rM^{\ell^3}$ respectively. With these identification, the intersection matrix writes as
\begin{equation}\label{E:intersection-matrix}
(\Res_{\fa_{i}}\circ \Gys_{\fa_j})_{1\leq i,j\leq 3}=\ell\begin{pmatrix}-2 & \eta^{-1} &\eta\\ \eta & -2 & \eta^{-1}\\
\eta ^{-1} & \eta & -2\end{pmatrix}.
\end{equation}
Note also that $\rH^1(\FF_{\ell^6}, \rH^1(\bfSh(G_{\emptyset_{\fa_{i}},\tT_{\fa_i}})_{\bF_\ell^\ac},\Lambda(1))/\Ker\phi^{\fr\fl})\cong \rH^1(\FF_{\ell^6}, \rM^{1})$ on which $\eta$ acts as the scalar $\ell^{-1}$.

By the proof of Theorem~\ref{T:level_raising} in Section \ref{S:proof}, we have a commutative diagram
\[\xymatrix{
\rH^1(\bF_{\ell^6},\rH^3(X_\fr\otimes\bF_\ell^\ac,\Lambda(2))/\Ker\phi^{\fr\fl}) \ar[d]_-{\Psi_{\fa_i}}
& \rH^1(\bF_{\ell^6},\rH^1(\bfSh(G_{\emptyset_{\fa_i},\tT_{\fa_i}})_{\bF_\ell^\ac},\Lambda(1))/\Ker\phi^{\fr\fl}))  \ar[l]_-{\r{Gys}_{\fa_i}} \\
\Gamma(X_\fr^{\r{sp}}(\bF_\ell^\ac),\Lambda)/\Ker\phi^{\fr\fl}
& \Gamma(\bfSh(G_{\tS_{\max}})(\bF_\ell^\ac),\Lambda)/\Ker\phi^{\fr\fl} \ar[l]_-{\cong}\ar[u]^-{\Phi_{\fa_i}}
}\]
where the bottom isomorphism is the one induced by the identification $X_\fr^{\r{sp}}\otimes\FF_{\ell^6}\cong\bfSh(G_{\tS_{\max}})_{\FF_{\ell^6}}$, and $\Phi_{\fa_i}$ is the map induced from \eqref{E:curve}. We claim that $\Phi_{\fa_i}$ is an isomorphism. Indeed, by Proposition~\ref{P:level_raising_curve} and the Nakayama Lemma, $\Phi_{\fa_i}$ is surjective. On the other hand, we have a commutative diagram
\[
\xymatrix{\bigoplus_{i=1}^3\Gamma(\bfSh(G_{\tS_{\max}})(\bF_\ell^\ac),\Lambda)/\Ker\phi^{\fr\fl} \ar[r]^-{\oplus_i \Phi_{\fa_i}} \ar[rd]_{\eqref{E:level-raising-refine}} & \bigoplus_{i=1}^3\rH^1(\bF_{\ell^6},\rH^1(\bfSh(G_{\emptyset_{\fa_i},\tT_{\fa_i}})_{\bF_\ell^\ac},\Lambda(1))/\Ker\phi^{\fr\fl})) \ar[d]^{\sum_{i} \Gys_{\fa_i}}\\
&  \rH^1(\bF_{\ell^6},\rH^3(X_\fr\otimes\bF_\ell^\ac,\Lambda(2))/\Ker\phi^{\fr\fl})}
\]
where the composite map is an isomorphism by Proposition~\ref{pr:refinement}. It follows that each $\Phi_{\fa_i}$ is injective, and hence an isomorphism.

Now, we have $\bar \Theta=\sum_{i=1}^3\Gys_{\fa_i}\circ\Phi_{\fa_i}\circ\Psi_{\fa_i}(\bar\Theta)$ and
\[
\Phi^{-1}_{\fa_1}\circ \Res_{\fa_1}\bar \Theta=\ell (-2\Psi_{\fa_1}(\bar \Theta)+\ell \Psi_{\fa_2}(\bar \Theta)+\ell^{-1} \Psi_{\fa_3}(\bar \Theta))=(\ell-1)^2\Psi_{\fa_1}(\bar\Theta)
\]
by \eqref{E:intersection-matrix}. Here, the last equality uses $\Phi_{\fa_1}(\bar\Theta)=\Psi_{\fa_2}(\bar\Theta)=\Psi_{\fa_3}(\bar \Theta)$ by symmetry. On the other hand, by \eqref{eq:reciprocity}, we have
\[
\Phi_\fa^{-1}\circ\Res_\fa\bar\Theta=\theta_*\b{1}^\flat
\]
for all $\fa\in \fB$, where $\b{1}^\flat$ is the characteristic function on $X_\fr^{\flat,\r{sp}}(\bF_\ell^\ac)$. Thus \eqref{eq:reciprocity0} follows immediately, and the theorem is proved.

\end{proof}

The following lemma will be needed in the next section.

\begin{lem}\label{le:center}
Let $\fs$ be an ideal of $\cO_F$ contained in $\fr\fl$. The map
\[
\bigoplus_{\fd\in\fD(\fr,\fc^+)}\delta^\fd_*\colon\Gamma(\cS_{\ell,\fc'\fr},\Lambda)/\Ker\phi^\fs
\to\bigoplus_{\fd\in\fD(\fr,\fc^+)}\Gamma(\cS_{\ell,\fc'\fc^+},\Lambda)/\Ker\phi^\fs
\]
is an isomorphism of free $\Lambda$-modules of rank $|\fD(\fr,\fc^+)|$, for $\fs=\fr\fl$.
\end{lem}

\begin{proof}
The ideal of proof is similar to \cite{Liu2}*{Lemma~3.33}. Recall that we have morphisms $\tilde\delta^\fd$ in \eqref{eq:correspondence0} for each $\fd\in \fD(\fr,\fc^+)$. As usual, we put $\tilde\delta\coloneqq\tilde\delta^{\cO_F}$. Form the following pullback square
\[
\xymatrix{
\cX_\fr^\fd \ar[r]^-{\varepsilon}\ar[d]_-{\varepsilon^\fd} &  \cX_\fr \ar[d]^-{\tilde\delta^\fd} \\
\cX_\fr \ar[r]^-{\tilde\delta} & \cX(\Delta)_{\fc'\fc^+,\cO_F}
}
\]
of schemes over $\bZ_{(\ell)}$, where all morphisms are finite \'{e}tale. The scheme $\cX_\fr^\fd$ has a natural action by $\dT^{\tR(\fr\fl)}$ under which the above diagram is equivariant. By an argument similar to \cite{Liu2}*{Lemma~3.33}, we obtain a commutative diagram
\begin{align}\label{eq:correspondence1}
\xymatrix{
\Gamma(\cS_{\ell,\fc'\fr},\Lambda)/\Ker\phi^{\fr\fl} \ar[r]^-{\Phi}\ar[d]_-{|(\cO_F/\fr)^\times|\cdot\delta^*\circ\delta^\fd_*}
& \rH^1_\unr(\bQ_\ell,\rH^3(\cX_\fr\otimes\bQ^\ac,\Lambda(2))/\Ker\phi^{\fr\fl})  \ar[d]^-{\varepsilon^\fd_*\circ\varepsilon^*} \\
\Gamma(\cS_{\ell,\fc'\fr},\Lambda)/\Ker\phi^{\fr\fl} \ar[r]^-{\Phi}
& \rH^1_\unr(\bQ_\ell,\rH^3(\cX_\fr\otimes\bQ^\ac,\Lambda(2))/\Ker\phi^{\fr\fl})
}
\end{align}
where $\Phi$ is the isomorphism in Proposition~\ref{pr:level_raising}. By proper base change, the endomorphism $\varepsilon^\fd_*\circ\varepsilon^*$ of $\rH^3(\cX_\fr\otimes\bQ^\ac,\Lambda(2))$ coincides with the composite map
\[
\rH^3(\cX_\fr\otimes\bQ^\ac,\Lambda(2))\xrightarrow{\tilde\delta^\fd_*}
\rH^3(\cX(\Delta)_{\fc'\fc^+,\cO_F}\otimes\bQ^\ac,\Lambda(2))\xrightarrow{\tilde\delta^*}\rH^3(\cX_\fr\otimes\bQ^\ac,\Lambda(2)).
\]
Definition~\ref{de:admissible}(A2) and Proposition~\ref{pr:level_raising}(1) imply that the image of
\[
\varepsilon^\fd_*\circ\varepsilon^*\colon\rH^1_\unr(\bQ_\ell,\rH^3(\cX_\fr\otimes\bQ^\ac,\Lambda(2))/\Ker\phi^{\fr\fl})
\to\rH^1_\unr(\bQ_\ell,\rH^3(\cX_\fr\otimes\bQ^\ac,\Lambda(2))/\Ker\phi^{\fr\fl})
\]
is a free $\Lambda$-module of rank $1$. Here, we use the fact that $\tilde\delta^*$ is injective, as $p\nmid\mu(\fr,\fc^+)$ in Definition~\ref{de:perfect_pair}(3b).

By the commutative diagram \eqref{eq:correspondence1}, we know that the image of
\[
\delta^*\circ\delta^\fd_*\colon\Gamma(\cS_{\ell,\fc'\fr},\Lambda)/\Ker\phi^{\fr\fl}\to\Gamma(\cS_{\ell,\fc'\fr},\Lambda)/\Ker\phi^{\fr\fl}
\]
is a free $\Lambda$-module of rank $1$. Since $\delta^\fd_*$ is surjective and $\delta^*$ is injective, $\Gamma(\cS_{\ell,\fc'\fc^+},\Lambda)/\Ker\phi^{\fr\fl}$ is a free $\Lambda$-module of rank $1$. Similarly, we may deduce that the map
\begin{align}\label{eq:correspondence2}
\bigoplus_{\fd\in\fD(\fr,\fc^+)}\delta^\fd_*\colon\Gamma(\cS_{\ell,\fc'\fr},\Lambda)/\Ker\phi^{\fr\fl}
\to\bigoplus_{\fd\in\fD(\fr,\fc^+)}\Gamma(\cS_{\ell,\fc'\fc^+},\Lambda)/\Ker\phi^{\fr\fl}
\end{align}
is injective. However, since the source of \eqref{eq:correspondence2} a free $\Lambda$-module of rank $|\fD(\fr,\fc^+)|$ by Proposition \ref{pr:level_raising}, the map \eqref{eq:correspondence2} has to be an isomorphism. The lemma follows.
\end{proof}

\subsection{First explicit reciprocity law}
\label{ss:first}

We keep the notation in Section \ref{ss:second}. Let $\underline{\ell}=(\ell,\ell')$ be a pair of distinct $(p^\nu,\fr)$-admissible primes (Definition \ref{de:admissible}) such that Lemma \ref{le:center} holds for $\fs=\fr\fl\fl'$, where $\fl'\coloneqq\ell'\cO_F$.

Put $\cX_{\fr,\underline\ell}\coloneqq\cX(\Delta\cup\{\fl,\fl'\})_{\fc',\fr}$ and $\cX_{\fr,\underline\ell}^\flat\coloneqq\cX(\Delta^\flat\cup\{\ell,\ell'\})_{\bZ,\fr\cap\bZ}$ (\cite{Liu2}*{Definition~B.1}), as schemes over $\bZ_{(\ell')}$. Then we obtain a canonical morphism
\begin{align}\label{eq:special2}
\theta_{\underline\ell}\colon\cX_{\fr,\underline\ell}^\flat\to\cX_{\fr,\underline\ell}.
\end{align}
Denote by $\Theta_{p,\fr,\underline\ell}^\nu$ the image of $\theta_{\underline\ell*}[\cX_{\fr,\underline\ell}^\flat\otimes\bQ]\in\CH^2(\cX_{\fr,\underline\ell}\otimes\bQ)$ under the Abel--Jacobi map
\[
\AJ_p\colon\CH^2(\cX_{\fr,\underline\ell}\otimes\bQ)\to
\rH^1(\bQ,\rH^3(\cX_{\fr,\underline\ell}\otimes\bQ^\ac,\Lambda(2))/\Ker\phi^{\fr\fl\fl'}).
\]

\begin{theorem}[First explicit reciprocity law]\label{th:first}
Let $\underline{\ell}=(\ell,\ell')$ be as above.
\begin{enumerate}
  \item There is a canonical decomposition of the $\Lambda[\rG_\bQ]$-module
      \[\rH^3(\cX_{\fr,\underline\ell}\otimes\bQ^\ac,\Lambda(2))/\Ker\phi^{\fr\fl\fl'}=\bigoplus_{\fd\in\fD(\fr,\fc^+)}\rM_0\]
      where $\rM_0$ is isomorphic to $\rN_\rho^\sharp(-1)$  (Notation~\ref{N:notation-rho}) as a $\Lambda[\rG_\bQ]$-module.

  \item There is a canonical isomorphism
      \[
      \rH^1_\sing(\bQ_{\ell'},\rH^3(\cX_{\fr,\underline\ell}\otimes\bQ^\ac,\Lambda(2))/\Ker\phi^{\fr\fl\fl'})
      \cong\Gamma(\cS_{\ell,\fc'\fr},\Lambda)/\Ker\phi^{\fr\fl},
      \]
      under which we have
      \[
      (\partial\loc_{\ell'}\Theta_{p,\fr,\underline\ell}^\nu,f)=(\ell'+1)\cdot
      \frac{|(\bZ/\fr\cap\bZ)^\times|}{|\Cl(F)_\fr|}\cdot\sum_{x\in \cS^\flat_{\ell,\fr}}f(\vartheta(x))
      \]
      for every $f\in\Gamma(\cS_{\ell,\fc'\fr},\Lambda)[\Ker\phi^{\fr\fl}]$.
\end{enumerate}
\end{theorem}

\begin{proof}
We will use results from \cite{Liu2}*{Sections 3 \& 4}. Put $\nabla^\flat\coloneqq\Delta^\flat\cup\{\infty,\ell\}$ as in the setup of \cite{Liu2}*{Section 4.1}. By Lemma \ref{le:center}, $(\rho,\fc'\fc^+,\fc',\fr)$ is a perfect quadruple in the sense of \cite{Liu2}*{Definition~3.2}, satisfying \cite{Liu2}*{Assumption~4.1}. Moreover, $\ell'$ is a cubic-level raising prime for $(\rho,\fc'\fc^+,\fc',\fr)$ in the sense of \cite{Liu2}*{Definition~3.3}.

Note that the morphism \eqref{eq:special2} is nothing but $\theta\colon\cX(\ell')_\fr^\flat\to\cX(\ell')_{\fc',\fr}$ in \cite{Liu2}*{(4.1.1)}; and the map \eqref{eq:special1} is nothing but $\vartheta\colon\cS^\flat_\fr\to\cS_{\fc'\fr}$ in \cite{Liu2}*{(4.1.2)}. Therefore, (1) follows from \cite{Liu2}*{Theorem~3.5(2)}; and (2) follows from \cite{Liu2}*{Theorem~3.5(3) \& Theorem~4.5}.
\end{proof}

\subsection{Proof of main theorem}

Recall that we have the multiplicatively induced representation $\rN_\rho^\sharp$ and the $\bZ/p^\nu[\rG_\bQ]$-module $\rM_0$ as in Theorem \ref{th:first}. We have a $\rG_\bQ$-equivariant pairing
\[
\rN_\rho^\sharp(-1)\times\rM_0\to\bZ/p^\nu(1)
\]
which induces, for every prime power $v$, a local Tate pairing
\[
\langle\;,\;\rangle_v\colon\rH^1(\bQ_v,\rN_\rho^\sharp(-1))\times\rH^1(\bQ_v,\rM_0)\to\rH^2(\bQ_v,\bZ/p^\nu(1))\simeq\bZ/p^\nu.
\]
For $s\in\rH^1(\bQ,\rN_\rho^\sharp(-1))$ and $r\in\rH^1(\bQ,\rM_0)$, we will write $\langle s,r\rangle_v$ rather than $\langle\loc_v(s),\loc_v(r)\rangle_v$.

\if false

\yichao{Here, can we get rid of the quadratic character $\chi$?  If $\chi$ is non-trivial, I  think we do not have  $\sum_{v}\langle s,r\rangle_v=0$ for $s\in\rH^1(\bQ,\rN_\rho^\sharp(-1))$ and $r\in\rH^1(\bQ,\rM_0)$, since the  sequence
\[\rH^2(\bQ,\bZ/p^{\nu}\bZ(\chi)(1))\xra{\oplus_{v}\loc_v} \rH^2(\bQ_v, \bZ/p^{\nu}\bZ(\chi)(1))\xra{\sum}\bZ/p^{\nu}\bZ\to0,
\]
is not exact.
}
\yifeng{I am confused: we only need the composition to be zero; why do we need exactness? In fact, one can run the argument over $F_0$.}
\yichao{Of course, one can run the argument over $F_0$, but I don't see how to descend the results to $\bQ$. One can identify $\rH^2(\bQ,\bZ/p^{\nu}\bZ(1)(\chi))$ with  $\rH^2(F_0,\bZ/p^{\nu}\bZ(1)(\chi))^{\Gal(F_0/\bQ)}=\rH^2(F_0,\bZ/p^{\nu}\bZ(1))[\chi]$, the $\chi$-isotypic component of  $\rH^2(F_0,\bZ/p^{\nu}\bZ(1))$. For $x\in \rH^2(F_0,\bZ/p^{\nu}\bZ(1))[\chi]$, then $\inv_{w}(x)=0$ if $w$ is a place of $F_0$ ramified or inert over $\bQ$, and $\inv_{w}(x)+\inv_{w^c}(x)=0$ if the underlying place of $w$ in $\bQ$ splits in $F_0/\bQ$. When we view $x$ as an element of $\rH^2(\bQ, \bZ/p^{\nu}(1)(\chi))$, and define its invariant at a local place $v$ of $\bQ$ as you did above, then we have $\inv_v(x)=0$ if $v$ is inert or ramified in $F_0$, and $\inv_v(x)$ is either identified with $\inv_w(x)$ or $\inv_{w^c}(x)$ depending how you embed $\Gal_{\bQ_v}$ into $\Gal_{\bQ}$. In particular, $\inv_v(x)$ seems only well-defined up to sign.     }

\fi

\begin{proof}[Proof of Theorem \ref{th:selmer}]
We assume that $\Theta_{p,\fr}$ is nonzero. Regard $\Theta_{p,\fr}$ as an element in $\rH^1_f(\bQ,\rH^3(\cX_\fr\otimes\bQ^\ac,\bZ_p(2))/\Ker\phi^\fr)$, which is not torsion. By \cite{BL84} and the assumption that $(p,\fr)$ is $\fr$-clean (Definition \ref{de:perfect_pair}), we know that $\rN_p\coloneqq\rH^3(\cX(\Delta)_{\fc'\fc^+,\cO_F}\otimes\bQ^\ac,\bZ_p(2))/\Ker\phi^\fr$ is a $\rG_\bQ$-stable lattice in $\sfM(E)_p$; and there exists some $\fd\in\fD(\fr,\fc^+)$ such that $\delta^\fd_*\Theta_{p,\fr}\in\rH^1_f(\bQ,\rN_p)$ is not torsion. Here, $\rH^1_{f}(\bQ,\rN_p)$ is by definition of the preimage of $\rH^1_f(\bQ,\sfM(E)_p)$ via the natural map $\rH^1(\bQ, \rN_p)\to \rH^1(\bQ,\sfM(E)_p)$. We fix such an element $\fd$. Let $\nu_0\geq 0$ be the largest integer such that $\delta^\fd_*\Theta_{p,\fr}\in p^{\nu_0}\rH^1_f(\bQ,\rN_p)$.

We prove by contradiction, hence assume $\dim_{\bQ_p}\rH^1_f(\bQ,\sfM(E)_p)\geq 2$. In what follows, we fix a sufficiently large integer $\nu$ as before, and will give a lower bound on $\nu$ for which a contradiction emerges at the end of proof.

By \cite{Liu1}*{Lemma~5.9}, we may find a free $\bZ/p^\nu$-submodule $\rS$ of $\rH^1_f(\bQ,\rN_\rho^\sharp(-1))$ of rank $2$ with a basis $\{s,s'\}$ such that $p^{\nu_0}s=\delta^\fd_*\Theta_{p,\fr}^\nu$. By the same discussion in \cite{Liu2}*{Section 4.3 (after Lemma~4.12)}, we have tower of fields $\bL_\rS/\bL/\bQ$ contained in $\bQ^\ac$. Let $\Box$ be the (finite) set of rational primes that are either ramified in $\bL_\rS$ or not coprime to $\Delta$ or $\fr\disc F$. Put $\nu_\Box\coloneqq\max\{\nu_v\res v\in\Box\}$ where $\nu_v$ is in \cite{Liu2}*{Lemma~4.12(2)}. We choose a prime $\ell_0\not\in\Box$ such that $\ell_0$ is $(p^\nu,\fr)$-admissible (Definition \ref{de:admissible}) -- this is possible by \cite{Liu2}*{Lemma~4.11}. Let $\gamma\in\Gal(\bL/\bQ)$ be the image of $\Frob_{w_0}$ under $\rho^\sharp(-1)$ (the image of $\rho^\sharp(-1)$ has been identified with $\Gal(\bL/\bQ)$), where $w_0$ is some prime of $\bL$ above $\ell_0$. Then $\gamma$ has order coprime to $p$; and $(\rN_\rho^\sharp(-1))^{\langle\gamma\rangle}$ is a free $\bZ/p^\nu$-module of rank $1$.

By \cite{Liu2}*{Lemma~4.16} and (the argument for) \cite{Liu2}*{Lemma~4.11}, we may choose two $(\Box,\gamma)$-admissible places (in the sense of \cite{Liu2}*{Definition~4.15}) $w,w'$ of $\bL$ such that
\begin{enumerate}
  \item $\Psi_w(s')=0$, $\Psi_w(s)=t$, $\Psi_{w'}(s')=t'$ with $t,t'\in(\rN_\rho^\sharp(-1))^{\langle\gamma\rangle}$ that are not divisible by $p$;
  \item the underlying prime $\ell$ of $w$ and the underlying prime $\ell'$ of $w'$ are distinct $(p^\nu,\fr)$-admissible primes, such that Lemma \ref{le:center} holds for $\fs=\fr\fl\fl'$.
\end{enumerate}
Put $\underline\ell\coloneqq(\ell,\ell')$. Then there are elements $\Theta_{p,\fr,\underline\ell}^\nu\in\rH^1(\bQ,\rH^3(\cX_{\fr,\underline\ell}\otimes\bQ^\ac,\Lambda(2))/\Ker\phi^{\fr\fl\fl'})$ from Section \ref{ss:first}, and $\delta^\fd_*\Theta_{p,\fr,\underline\ell}^\nu\in\rH^1(\bQ,\rM_0)$. We have
\begin{enumerate}\setcounter{enumi}{2}
  \item $\loc_v\Theta_{p,\fr,\underline\ell}^\nu\in\rH^1_\unr(\bQ_v,\rM_0)$ for a prime $v\not\in\Box\cup\{p,\ell,\ell'\}$, by \cite{Liu1}*{Lemma~3.4};

  \item $\loc_p\Theta_{p,\fr,\underline\ell}^\nu\in\rH^1_f(\bQ_p,\rM_0)$, by \cite{Nek00}*{Theorem~3.1(ii)}.
\end{enumerate}
By \cite{Liu2}*{Lemma~4.6} and \cite{Liu1}*{Lemma~3.4}, we have $\loc_v(s')\in\rH^1_\unr(\bQ_v,\rN_\rho^\sharp(-1))$ for every prime $v\not\in\Box\cup\{p,\ell,\ell'\}$. By \cite{Liu1}*{Definition~4.6, Remark~4.7}, we have $\loc_p(s')\in\rH^1_f(\bQ_v,\rN_\rho^\sharp(-1))$. Then by \cite{Liu2}*{Lemma~4.12(2,3,5)} and (3,4) above, we have
\begin{align}\label{eq:annihilate1}
p^{\nu-\nu_\Box}\mid\sum_{v\not\in\{\ell,\ell'\}}\langle s',\Theta_{p,\fr,\underline\ell}^\nu\rangle_v.
\end{align}
Since $\Psi_w(s')=0$ by (1), we also have
\begin{align}\label{eq:annihilate2}
\langle s',\Theta_{p,\fr,\underline\ell}^\nu\rangle_\ell=0.
\end{align}

Let $\phi_0$ be a generator of $\Gamma(\cS_{\ell,\fc'\fc^+},\bZ/p^\nu)[\Ker\phi^{\fr\fl\fl'}]$ which is a free $\bZ/p^\nu$-module of rank $1$. Then by the choice of $s$, $w$ in (1), and Theorem \ref{th:second}, we have
\[
\sum_{\cS^\flat_{\ell,\fr}}\phi_0(\delta^\fd(\vartheta(x)))\in p^{\nu_0}\bZ/p^\nu-p^{\nu_0+1}\bZ/p^\nu.
\]
By the choice of $w'$ in (1) and Theorem \ref{th:first}, we have
\begin{align}\label{eq:annihilate3}
\langle s',\Theta_{p,\fr,\underline\ell}^\nu\rangle_{\ell'}\in p^{\nu_0}\bZ/p^\nu-p^{\nu_0+1}\bZ/p^\nu.
\end{align}
Here, we have used the fact that $p$ is coprime to $|(\bZ/\fr\cap\bZ)^\times|$, $|\Cl(F)_\fr|$, $(\ell-1)$, and $\ell'+1$.

Take $\nu\in\bZ$ such that $\nu>\nu_0+\nu_\Box$. Then the combination of \eqref{eq:annihilate1}, \eqref{eq:annihilate2} and \eqref{eq:annihilate3} contradicts with the following well-known fact:
\[\sum_{v}\langle s',\Theta_{p,\fr,\underline\ell}^\nu\rangle_v=0\]
due to the global class field theory and the fact that $p$ is odd, where the sum is taken over all primes $v$. Theorem \ref{th:selmer} is proved.
\end{proof}

\end{document}